\documentclass[a4paper,11pt,oneside]{amsart}
\usepackage[english]{babel}
\usepackage[colorlinks,citecolor=green,linkcolor=red]{hyperref}
\usepackage{amsthm}
\usepackage{color}
\usepackage{mathrsfs}
\usepackage{amsmath}
\usepackage{mathtools}
\usepackage{amsfonts}
\usepackage{amssymb}
\usepackage{bm}
\usepackage{physics}
\usepackage{enumitem}
\usepackage[margin=0.7in]{geometry}
\usepackage{cleveref}
\usepackage{esint}
\usepackage{nicefrac}
\usepackage{yfonts}
\usepackage{dirtytalk}
\usepackage{verbatim}
\usepackage{xfrac}
\usepackage{stmaryrd,mathrsfs,bm,amsthm,mathtools,yfonts,amssymb,color}

\newtheorem{thm}{Theorem}
\newtheorem*{thm*}{Theorem}
\newtheorem{lem}[thm]{Lemma}
\newtheorem{prop}[thm]{Proposition}
\newtheorem{cor}[thm]{Corollary}

\theoremstyle{definition}
\newtheorem{defn}[thm]{Definition}

\theoremstyle{remark}
\newtheorem{rem}[thm]{Remark}
\newtheorem{ass}[thm]{Assumption}
\newtheorem{ex}[thm]{Example}

\numberwithin{equation}{section}
\numberwithin{thm}{section}

\newcommand{\fr}{\penalty-20\null\hfill$\blacksquare$}

\newcommand{\defeq}{\coloneqq}
\newcommand{\eqdef}{\eqqcolon}

\def\XXint#1#2#3{{\setbox0=\hbox{$#1{#2#3}{\int}$}
		\vcenter{\hbox{$#2#3$}}\kern-.5\wd0}}

\renewcommand{\dv}{{\text {div}}}

\newcommand{\RR}{\mathbb{R}}
\newcommand{\NN}{\mathbb{N}}

\newcommand\supp{{\rm spt}}
\newcommand{\sing}{{\rm Sing}}
\newcommand{\reg}{{\rm Reg}}

\newcommand\res{\mathop{\hbox{\vrule height 7pt width .3pt depth 0pt
			\vrule height .3pt width 5pt depth 0pt}}\nolimits}
		\newcommand{\mres}{\res}
\newcommand{\im}{{\rm Im}}
\newcommand{\gr}{{\rm Gr}}
\newcommand{\bT}{\mathbf{T}}
\newcommand{\bG}{\mathbf{G}}

\newcommand{\p}{{\mathbf{p}}}

\newcommand{\sep}{{\mathrm{sep}}}

\newcommand{\sW}{{\mathscr{W}}}
\newcommand{\sC}{{\mathscr{C}}}
\newcommand\sS{{\mathscr S}}

\newcommand{\bGam}{{\bm \Gamma}}
\newcommand\bmo{{\bm m}_0}
\newcommand{\cM}{{\mathcal{M}}}
\newcommand{\bU}{{\mathbf{U}}}

\newcommand{\Phii}{{\bm{\Phi}}}

\newcommand{\cL}{{\mathcal{L}}}
\newcommand{\cK}{{\mathcal{K}}}
\newcommand{\bS}{{\mathbf{S}}}

\newcommand{\bD}{{\mathbf{D}}}
\newcommand{\bH}{{\mathbf{H}}}
\newcommand{\bI}{{\mathbf{I}}}
\newcommand{\bE}{{\mathbf{E}}}

\newcommand{\bSigma}{{\mathbf{\Sigma}}}

\newcommand{\bB}{{\mathbf{B}}}
\newcommand{\bC}{{\mathbf{C}}}

\newcommand\Z{{\mathbb Z}}

\newcommand\N{{\mathbb N}}

\newcommand\R{{\mathbb R}}

\newcommand{\eps}{{\varepsilon}}

\newcommand{\bPsi}{{\mathbf{\Psi}}}

\newcommand{\Lip}{{\rm {Lip}}}
\newcommand{\diam}{{\rm {diam}}}
\newcommand{\dist}{{\rm {dist}}}

\newcommand{\cB}{{\mathcal{B}}}
\newcommand{\cG}{{\mathcal{G}}}

\newcommand{\cF}{{\mathcal{F}}}

\newcommand{\cR}{{\mathcal{R}}}

\newcommand{\mass}{{\mathbf{M}}}

\newcommand{\Iq}{{\mathcal{A}}_Q}
\def\a#1{\llbracket{#1}\rrbracket}

\newcommand{\D}{\textup{Dir}}
\newcommand{\de}{\partial}

\newcommand{\etaa}{{\bm{\eta}}}

\newcommand\B{\bB}

\newcommand{\bs}{\mathbf{s}}

\makeatletter 
\def\@tocline#1#2#3#4#5#6#7{\relax
	\ifnum #1>\c@tocdepth 
	\else
	\par \addpenalty\@secpenalty\addvspace{#2}%
	\begingroup \hyphenpenalty\@M
	\@ifempty{#4}{%
		\@tempdima\csname r@tocindent\number#1\endcsname\relax
	}{%
		\@tempdima#4\relax
	}%
	\parindent\z@ \leftskip#3\relax \advance\leftskip\@tempdima\relax
	\rightskip\@pnumwidth plus4em \parfillskip-\@pnumwidth
	#5\leavevmode\hskip-\@tempdima
	\ifcase #1
	\or\or \hskip 1em \or \hskip 2em \else \hskip 3em \fi%
	#6\nobreak\relax
	\dotfill\hbox to\@pnumwidth{\@tocpagenum{#7}}\par
	\nobreak
	\endgroup
	\fi}
\makeatother

\begin{document}
\begin{abstract}
      The main goal of this work is to prove an instance of the unique continuation principle for area minimizing integral currents. More precisely, consider an $m$-dimensional area minimizing integral current and an $m$-dimensional minimal surface, both contained in $\R^{n+m}$ with $n\geq 1$. We show that if, in an integral sense, the current has infinite order of contact with the minimal surface at a point, then the current and the minimal surface coincide in a neighbourhood of that point.
 \end{abstract}
	\title{Unique continuation for area minimizing currents}

	\author[C.\ Brena]{Camillo Brena}
	\address{C.~Brena: Scuola Normale Superiore, Piazza dei Cavalieri 7, 56126 Pisa} 
	\email{\tt camillo.brena@sns.it}
        \author[S.\ Decio]{Stefano Decio}
 	\address{S.~Decio: School of Mathematics, Institute for Advanced Study, 1 Einstein Dr., Princeton NJ 08540} 
	\email{\tt sdecio@ias.edu}
	\maketitle
	\setcounter{tocdepth}{1}
 
	\tableofcontents
\section{Introduction}

One of the key features of analytic functions is the unique continuation property. This can take several shapes - what is usually called \emph{weak} unique continuation postulates that if the function vanishes on an open set it must vanish everywhere, while \emph{strong} unique continuation keeps the same conclusion but under the assumption that the function vanishes to infinite order at a point. 

Starting with the work of Carleman \cite{Carleman}, it was realized that unique continuation is not really a feature of analyticity but rather of \emph{ellipticity}, and that solutions to relatively nice systems of elliptic PDEs should enjoy unique continuation properties. The first proofs of strong unique continuation for solutions of a class of second-order elliptic systems in any dimension are due to Aronszajn and collaborators (see \cite{Aronszajn}, \cite{AKS}), using the rather technical method of weighted integral inequalities developed by Carleman (these inequalities are now usually called Carleman inequalities). Independently and around the same time Agmon (see \cite{Agmon}) gave a completely different proof based on convexity properties of integral averages and an abstract ODE framework.\\

A conceptually clearer and technically simpler proof of the strong unique continuation property for solutions of second-order elliptic PDEs was given by Garofalo and Lin in \cite{GL1} and \cite{GL2}, and independently by Kazdan in \cite{Kazdan}. These new proofs were based on the remarkable discovery by Almgren in his ``big regularity paper'' (\cite{Almgren}) of what he called the frequency function and its monotonicity property. In the simplest setting of a harmonic function $u$ in $\R^d$, the frequency function is defined as
$$
\bI(r)=\frac{r\int_{\B_r(0)}|\nabla u|^2}{\int_{\de \B_r(0)}|u|^2}.
$$
If $u$ is a homogeneous harmonic polynomial, $\bI(r)$ is the degree of the polynomial -- for instance in $\R^2$ if $u(r,\theta)=a_kr^k\sin(k\theta)$, $\bI(r)=k$ (the reason for the name frequency should now be apparent). It follows that one can think of the frequency function as a far-reaching generalization of the degree of a polynomial. It is an elementary computation using the expansion in homogeneous harmonic polynomials that the vanishing order at $0$ of a harmonic function $u$ is given by $\lim_{r\to 0}\bI(r)$. Almgren's key observation is that $\bI(r)$ is a \emph{non-decreasing} function of $r$ -- in the case of classical harmonic functions this gives immediately the boundedness of the order of vanishing, and an effective bound is the frequency at scale $1$.

The works of Garofalo--Lin and Kazdan further explore Almgren's idea and prove (almost) monotonicity of a generalization of the frequency function for solutions of second-order linear elliptic PDEs with Lipschitz coefficients. The strong unique continuation property follows then from the almost monotonicity of the frequency and in fact the frequency function controls precisely the order of vanishing.\\

Almgren's goal was to prove the dimension bound on the singular set of higher codimension area minimizing integral currents, which are much more complicated objects than solutions of linear elliptic PDEs, and classical unique continuation is not explicitly mentioned -- but in fact, Almgren used the frequency function to prove that the sheets of the current cannot have infinite order of contact with each other, which is really a unique continuation result and a key step in obtaining the dimension bound (see Brian White's MathSciNet review of \cite{Almgren} for more extensive comments on this point). 

In the present work we study unique continuation in the original setting of Almgren, that of area minimizing integral currents (of codimension greater or equal to one). The question we deal with is whether an area minimizing current can have infinite order of contact with a fixed classical minimal surface, and the answer is that this can only happen if the current identically coincides with the surface. To fix ideas, note that if the classical minimal surface is a hyperplane and the current is a graph over that hyperplane this corresponds to a relatively classical unique continuation result. Let us now state our main result, and we will then further comment on it. We assume throughout the paper that the reader is familiar with the concept of an area minimizing integral current. When we say minimal surface below, we mean in the sense of having zero mean curvature.

\begin{thm}[Unique continuation]\label{mainthm}
Let $\cM\subseteq\RR^{n+m}$ be an  $m$-dimensional minimal surface and let $T$ be an $m$-dimensional (locally) area minimizing current in $\RR^{n+m}$ with $0\in\supp(T)$. Any ball in this statement will be centered at $0$. Assume that $\cM$ and $T$ have no boundary in $\bar\bB_1$. If
\begin{equation}\label{mainthmeq}
	\liminf_{r\searrow 0} \frac{1}{r^{N}}\int_{\bB_r} \dist^2(x,\cM)\dd\|T\|(x)=0\quad\text{for every }N\in\NN,
\end{equation}
then there exist $r\in (0,1)$ and an integer $Q$ such that 
\begin{equation}\label{mainthmeq1}
	T\mres\bB_r=Q\a{\cM}\mres \bB_r.
\end{equation}
In particular, $T$ is regular in a neighbourhood of $0$.

Assume moreover that $\cM\cap \B_1$ is connected and that $T\mres\B_1$ is indecomposable. Then $\cM\cap \B_1$  is orientable and
\begin{equation}\label{mainallball}
    T\mres \B_1=Q\a{\cM}\mres \bB_1.
\end{equation}
\end{thm}

Note that the fact that $T\mres\B_1$ is indecomposable is essential to have \eqref{mainallball}, as the following easy example shows.
\begin{ex}
	Endow $\RR^{2+2}=\RR^4\simeq\mathbb{C}^2$ with coordinates $(z,w)\in\mathbb{C}\times\mathbb{C}$ and, for $\alpha > 0$, let $T_\alpha$ be the $2$-dimensional (locally) area-minimizing current defined as $T_\alpha\defeq \a{\{w=0\}}+\a{\{w=\alpha (z-\sfrac{1}{2})\}}$. Then, for every $z\in\mathbb{C}\setminus \{\sfrac{1}{2}\}$, $T$ is regular in a neighbourhood of $(z,0)$ and has infinite order of contact with the minimal submanifold given by the real $2$-plane $\{w=0\}$. However, there is no neighbourhood of $\{w=0\}$ in which $T_\alpha$ is regular. Notice that as $\alpha\rightarrow0$, $\bE(T_\alpha,\bB_1)\rightarrow 0$, (see Definition \ref{d:excess_and_height} below for the excess) so that smallness of the excess is still not enough to give strong unique continuation.\fr
\end{ex}

Note that \eqref{mainthmeq} is a way of saying that $T$ and $\cM$ have infinite order of contact at $0$, and we conclude that if this happens $T$ and $\cM$ coincide in a neighbourhood of $0$. This is quite close in spirit to Almgren's argument -- a little informally, one can say that he proves that the different sheets of the current cannot have infinite order of contact between each other unless they coincide and the current is in fact regular near that point. Here we study contact with another surface rather than within the sheets, but the point we care to make is that unique continuation and regularity are inextricably intertwined, as already apparent in Almgren's work. 

\subsection{Short outline of the proof}
Almgren's daunting ``big regularity paper'' was substantially simplified and sharpened in a series of influential papers by De Lellis and Spadaro (\cite{DSq},\cite{DSsns},\cite{DS1},\cite{DS2},\cite{DS3}), which have in the years since become the foundational texts for approaching the regularity theory of higher codimension area minimizing integral currents. Our approach for the proof of Theorem \ref{mainthm} also follows those works and we make liberal use of lemmas and theorems from them. This makes the paper unfortunately not fully self-contained, as a careful reading of the proofs will require to go back and forth between the present paper and the articles of De Lellis and Spadaro -- we tried our best to indicate where we take things as a ``black box'', where we make trivial modifications, and where we differ substantially. \\

We believe, however, that the idea of the proof is not terribly complicated and we attempt to summarize it here. First, since due to \eqref{mainthmeq} a tangent cone to the origin is flat, we can assume that the excess of the current is small. Then we approximate the current with a multi-valued Lipschitz map $N$ constructed on the normal bundle of $\cM$. This approximation actually coincides with the current on a large set and comes with estimates on the Lipschitz constant and the Dirichlet energy -- see the statements in Section \ref{sectnorm}. The normal approximation construction is based on a Whitney cube decomposition of $\cM$. The square of the $L^2$ norm of $N$ is loosely comparable to $\int\dist^2(x,\cM)\dd\|T\|$ in each Whitney region, see Section \ref{sectcomp} and Proposition \ref{small}. In turn, the multi-valued Lipschitz map is close to a Dir-minimizing function, which is a multi-valued generalization of a harmonic function introduced by Almgren -- this fact is used to prove the important estimates in Sections \ref{sectsep}, which indeed allow to control quantities like Dirichlet energy and excess in terms of the same quantities in slightly different locations, much in the same way as it happens for harmonic functions.  
We then introduce a modified version of the frequency function (Section \ref{s:frequency}), again due to De Lellis and Spadaro, and prove its boundedness and almost monotonicity in certain intervals (corresponding to the scales in which $N$ is a good enough approximation of the current), see the statements in Section \ref{freqbound}. Finally, using some differential inequalities that relate logarithmic derivatives of $\int_{\cB_r} |N|^2$ with the frequency, we are able to prove quantitative boundedness from below of terms like $\int_{\cB_r} |N|^2$, which, together with the aforementioned comparison with terms of the form $\int_{\bB_r}\dist^2(x,\cM)\dd\|T\|$, concludes the argument. This last step actually comes first in the paper and occupies Section \ref{s:proof}.\\

The reader familiar with the works of De Lellis and Spadaro will have noticed that the steps described here mirror quite closely the program of those papers, and in fact the normal approximation construction is operationally almost the same -- the key difference is that they construct a ``center manifold'' (specifically the construction is in \cite{DS2}), which is essentially a smoothened out version of the ``average of the sheets'' of the current, and their Lipschitz approximation is on the normal bundle of this center manifold. Conceptually, in our setup the fixed minimal surface $\cM$ 
plays the same role that the center manifold has in \cite{DS2} for the construction of the normal approximation. It is worth highlighting that, despite the construction of the normal approximation on $\cM$ here and in \cite{DS2} following the same scheme, there is a deep difference: while in \cite{DS2} much effort is spent on actually building the center manifold (one should better say: \emph{a} center manifold, since this object might very well be non-unique) and indeed De Lellis and Spadaro have to perform an involved gluing procedure, we are given right from the beginning a smooth manifold $\cM$. Even though not having the necessity to build a center manifold somewhat simplifies our work, it also causes a drawback in the construction of the normal approximation: since we are not free to choose the manifold, we lose some flexibility. As an effect of this fact, notice that the estimate  on the barycenter $|\eta\circ N|$ of \cite[Equation (2.4)]{DS2} does not have a parallel in our Theorem \ref{t:approx}, and, indeed, is not expected to hold in our context. This for instance prevents us from handling the error estimates as in \cite{DS3} (see in particular the proof of \cite[Proposition 3.5]{DS3}) for the case in which our ambient space is not Euclidean but a smooth submanifold. \\

In a nutshell, our proof of Theorem \ref{mainthm} is really a combination of the construction of De Lellis and Spadaro in \cite{DS2}, \cite{DS3} and computations on the frequency function \textit{{\`a} la} Garofalo and Lin (\cite{GL1}, \cite{GL2}). This is most apparent in a model case in which there is only one ``interval of flattening'' (see Section \ref{sectif} below), as would happen if we were treating solutions of second-order linear elliptic PDEs, where the computations in Section \ref{s:proof} simplify and are analogous to those of \cite{GL1}. To fix ideas, let us further simplify the situation of Theorem \ref{mainthm} and assume that $\cM$ is an hyperplane $\pi_0=\RR^m\times \{0\}\subseteq\RR^m\times \RR$ and that the area minimizing current $T$ is graphical over $\pi_0$, say that $T$ is the graph of $N:\pi_0\rightarrow\RR=\{0\}\times \RR\subseteq\RR^m\times\RR$. Being $T$ minimal, $N$ solves the minimal surface equation, and we know, by classical arguments, that $N$ is real analytic. It is then easy to realize, by a Taylor expansion, that \eqref{mainthmeq} implies that the derivatives of $N$, of any order, vanish at $0$, so the result follows from the classical unique continuation for analytic functions. One could, however, forget about the analiticity and prove Theorem \ref{mainthm} in this setting by following (with several simplifications) the computations with the frequency function in Section \ref{s:proof}, building upon the estimates in Sections \ref{freq} and \ref{freqbound} (again, those estimates in this setting being quite a bit easier) -- in fact, we might even be tempted to invite the reader to follow the arguments in the present paper thinking about this situation at first reading.
In the general case, the analiticity argument clearly breaks down and in fact analiticity need not hold (indeed, it usually fails), and we have to deal with the following complications: 
\begin{itemize}
\item Even though $T$ may be graphical, it needs not be a single graph, but rather a multigraph,
\item $T$ may be not graphical.
\end{itemize}
The first consideration pushed Almgren to deal with Q-valued functions (see \cite{DSq}), whereas the second consideration led to the process of normal approximation mentioned above. In particular, it may happen that $T$ is only \emph{approximately} a graph of a multi-valued map. It is a testament to the strength of the arguments in the works of Almgren and De Lellis--Spadaro that the frequency function proof can indeed be adapted to cover such general case, as we set out to do here.

The structure of our paper is based on the articles \cite{DS2} and \cite{DS3} -- in particular, Part \ref{sect:normapprox}, which contains the construction of the normal approximation, mirrors \cite{DS2}, while Part \ref{part2}, containing the estimates for the modified frequency function, mirrors \cite{DS3}. For the benefit of the reader, we tried to maintain some loose correspondence between the Sections in our paper and those in \cite{DS2} and \cite{DS3}, and this will be indicated at the beginning of the respective Parts.

\subsection{Future directions}
There are several possible interesting directions stemming from this work, and we believe that studying unique continuation questions in the context of minimal surfaces (in a sufficiently general sense) can prove very fruitful. We identify here two problems that we believe are far from easy and would lead to several interesting consequences.\\

\noindent \textbf{Question 1.} In the classical case, i.e. for solutions of second-order elliptic PDEs, one can quantitatively control the measure of the zero set in a ball around a point in which the order of vanishing is bounded in terms of an explicit function of the order of vanishing itself -- this is a well-known result of Hardt and Simon, \cite{HS}. Is there any hope of establishing an analogous result in our setting, meaning quantitatively bound the measure of the intersection set of $T$ with $\cM$? We believe the answer is not known and not trivial for the corresponding question in the simpler case of Dir-minimizing functions. \\

\noindent \textbf{Question 2.} Can we establish an analogue of Theorem \ref{mainthm} for stationary integral varifolds? The problem here becomes interesting also in codimension one. The proof of our theorem uses the fact that the current is area minimizing and not merely stationary chiefly in the construction of the normal approximation. It would be extremely interesting to find a way to bypass this difficulty.

\section*{Acknowledgments} We are extremely grateful to Camillo de Lellis for suggesting the problem and for sharing with us his ideas, as well as for his encouragement and support. 

The second named author is partially supported by the National Science Foundation and the Giorgio and Elena Petronio Fellowship Fund.

Most of this work was developed while the first named author was visiting the Institute for Advanced Study in Princeton. He wishes to thank the institute for the wonderful hospitality, and the Balzan project led by Luigi Ambrosio for financial support during the visit.

\section{Notation, height and excess}

For open balls in $\R^{m+n}$ we use $\B_r (p)$. 
For any linear subspace $\pi\subseteq \R^{m+n}$, $\pi^\perp$ is
its orthogonal complement, $\p_\pi$ the orthogonal projection onto $\pi$, $B_r (q,\pi)$ the disk $\bB_r (q) \cap (q+\pi)$
and  $\bC_r (p, \pi)$ the cylinder $\{(x+y):x\in \B_r (p), y\in \pi^\perp\}$ (in both cases $q$ is 
omitted if it is the origin and $\pi$ is omitted if it is $\pi_0\defeq\RR^m\times\{0\}$). We extend to affine subspaces the definition of $\p_\pi$.
We also assume that each $\pi$ is {\em oriented} by a
$k$-vector $\vec\pi :=v_1\wedge\ldots \wedge v_k$
(thereby making a distinction when the same plane is given opposite orientations) and with a slight abuse of notation we write $|\pi_2-\pi_1|$ for
$|\vec{\pi}_2 - \vec{\pi}_1|$ (where $|\cdot |$ stands for the norm associated to the usual inner product of $k$-vectors).

A primary role will be played by the $m$-dimensional plane $\mathbb R^m \times \{0\}$ with the standard orientation:
for this plane we use the symbol $\pi_0$ throughout the whole paper.

\begin{defn}[Excess and height]\label{d:excess_and_height}
	Given an integer rectifiable $m$-dimensional current $T$ in $\mathbb R^{m+n}$  and (oriented) $m$-planes $\pi, \pi'$, we define the {\em excess} of $T$ in balls and cylinders as
	\begin{align*}
		\bE(T,\B_r (x),\pi) &\defeq \frac{1}{2\omega_mr^m}\int_{\B_r (x)} |\vec T - \vec \pi|^2  \dd\|T\|,\\
		 \bE (T, \bC_r (x, \pi'), \pi) &\defeq \ \frac{1}{2\omega_mr^m} \int_{\bC_r (x, \pi')} |\vec T - \vec \pi|^2 \dd\|T\|.
	\end{align*}
	%
	Finally, if also $q\in\RR^{m+n}$, we define the {\it separation function} in a set $A \subseteq \R^{m+m}$ as
	\[
	\bs(T,A,q+\pi) \defeq \sup_{x\in\supp(T)\cap A} |\p_{\pi^\perp}(x-q)| .
	\]
\end{defn}

 Notice that only $\bs$ takes, as a third argument, affine planes instead of planes (indeed, it is the only quantity sensitive to translations of the plane). We will often need to compute $\bs$ on the tangent space  at $p$ of a Riemannian manifold $\cM$: we understand tangent spaces to be linear spaces, hence, the affine plane tangent to $\cM$ at $p$ will be denoted by $p+T_p\cM$.


\bigskip

For a Riemannian manifold $\cM$, the geodesic distance on $\cM$ will be denoted by $d$. Geodesic balls are denoted by $\cB_r (q)\subseteq\cM$. If $0\in\cM$, $\cB_r$  denotes $\cB_r(0)$.

\section{Proof of the main result}\label{s:proof}
\subsection{Strong unique continuation}

	Let $\cM$ and $T$ be as in the first part of Theorem \ref{mainthm}. In this subsection we prove \eqref{mainthmeq1}. 
	
	\bigskip

We can clearly assume that $0\in\supp(T)$. Then $0\in\cM$ (by  \eqref{mainthmeq} and the monotonicity formula); by possibly rotating the coordinates, we assume that $T_0\cM=\pi_0$.
 With a diagonal argument, we take $r_k\searrow0$ such that 
\begin{equation}\notag
\lim_k \frac{1}{r_k^{m+2}}\int_{\B_{r_k}} \dist^2(x,\cM)\dd\|T\|(x)=0
\end{equation}
and such that $T_k\defeq (\iota_{0,r_k})_\sharp T\to S$, where $S$ is an area minimizing cone.
We compute
\begin{align*}
	\int_{\B_1} \dist^2(x,\pi_0)^2\dd\|S\|(x)&=\lim_k 	\int_{\B_1} \dist^2(x,\pi_0)^2\dd\|T_k\|(x)=\lim_k \frac{1}{r_k^{m+2}}\int_{\B_{r_k}}\dist^2(x,\pi_0)\dd\|T\|(x)
	\\&\le \limsup_k \frac{1}{r_k^{m+2}}\int_{\B_{r_k}}4\dist^2(x,\cM)^2\dd\|T\|(x)+ \limsup_k \frac{1}{r_k^{m+2}}C(\cM)\|T\|(\B_{r_k})r_k^4=0,
\end{align*}
so that $S$ is flat, say $S=Q\a{\pi_0}$. Notice that this implies that $\Theta(0,T)=Q$ and that 
$\bE(T_k,\B_{6\sqrt{m}},\pi_0)\rightarrow 0$. 
Fix any choice of the parameters as in Assumption \ref{parametri}  (including $\eps_0$) such that all the results of Part \ref{sect:normapprox} and Part \ref{part2} hold.
 Then there exists $k_0$ such that $T^0\defeq T_{k_0}$ and $\iota_{0,r_{k_0}}(\cM)$ satisfy Assumption \ref{ipotesi} (with this choice of $\eps_0$). We then replace $T$ by the current given by Lemma \ref{height} and we notice that it is enough to prove \eqref{mainthmeq} for this new choice of $T$.  Clearly, also Assumption \ref{ass:sectfour} is in place. 
 
 We record, for future reference, the following fact:
 \begin{equation}\label{bgrvfd}
 \lim_k \bE(T,\B_{3\sqrt{m}r_k},\pi_0)=0\quad\text{if}\quad\lim_k \frac{1}{r_k^{m+2}}\int_{\B_{r_k}} \dist^2(x,\cM)\dd\|T\|(x)=0.
 \end{equation}

 Now  we assume by contradiction that, for every $r\in ]0,1[$, $\supp(T)\cap\B_r$ is not contained in $\cM$ (so that we can apply Theorem \ref{t:boundedness}). Indeed, notice that if, for some  $r\in ]0,1[$, $\supp(T)\cap\B_r\subseteq\cM$, then \eqref{mainthmeq1} holds, by the Constancy Theorem.
 
 Assume that $\gamma_3$ in Proposition \ref{p:variation} is fixed and take $p\in\NN$. Let $[s_j,t_j]$ be an interval of flattening associated to $T$.
 We compute, by the estimates of Proposition \ref{p:variation}, for a.e.\ $r\in[\sfrac{s_j}{t_j},3]$ ($C=C(C_1,\gamma_3)$),
 \begin{equation}
\begin{split}\label{vfeda}
	 	\partial_r\log(r^{-(m-1)-p}\bH)&=\frac{\bH'}{\bH}-\frac{m-1}{r}-\frac{p}{r}\le \frac{2}{\bH}\frac{\bE}{r}+C-\frac{p}{r}\\&\le 2\frac{\bD}{\bH}+C\frac{\bD}{\bH}+C\frac{\bSigma}{\bH}+C-\frac{p}{r}
 	\\&\le C\frac{\bI}{r}+C-\frac{p}{r}.
\end{split}
 \end{equation}
 Now we examine two cases separately. Splitting the two cases is not strictly necessary, but helps with the presentation.
 \begin{itemize}
 	\item[(1)] The intervals of flattening are finitely many,
  	\item[(2)] The intervals of flattening are infinitely many.
 \end{itemize}
 \medskip\textbf{Case (1).} Take $j_0$ such that $s_{j_0}=0$. 
 By Theorem \ref{t:boundedness}, for some $\rho_0\in ]0,1[$, $$\bH_{j_0}>0\text{ on }]0,\rho[\quad\text{and}\quad\bar\bI_{j_0}\defeq\sup_{r\in ]0,\rho_0[}\bI_{j_0}(r)<\infty,$$
 Therefore, by \eqref{vfeda}, for $p\ge 1$ big enough (depending upon $\bar\bI_{j_0}$ and all the parameters except $\eps_0$),  $r^{-(m-1)-p}\bH_{j_0}(r)$ is decreasing in $r$. In particular,
 \begin{equation}\label{c1}
 	\liminf_{r\searrow 0} \frac{\bH_{j_0}(r)}{r^{(m-1)+p}}>0.
 \end{equation}
 But, by Proposition \ref{small}, 
 \begin{equation}\label{c2}
 	\bH_{j_0}(r)\le 4r^{-1}\int_{\cB_r}|N_{j_0}|^2\le 16 r^{-1}\int_{\B_{2r}}\dist^2(x,\cM_{j_0})\dd\|T_{j_0}\|(x)=\frac{16}{rt_{j_0}^{m+2}}\int_{\B_{2r t_{j_0}}}\dist^2(x,\cM)\dd\|T\|(x).
 \end{equation}
 But this is a contradiction, as \eqref{c1} and \eqref{c2} are incompatible with \eqref{mainthmeq}.
\medskip \\\textbf{Case (2).} Take $j_0$ as in Theorem \ref{t:boundedness}, by \eqref{e:finita2}, $$\bar\bI\defeq\sup_{j\ge j_0}\sup_{r\in ]\sfrac{s_j}{t_j},3[}\bI_{j}(r)<\infty. $$
  Therefore, by \eqref{vfeda}, we obtain that,  if $p$ is big enough (depending upon $\bar\bI$ and all the parameters except $\eps_0$, and not depending on $j$),
\begin{equation}\label{fcads}
	  r^{-(m-1)-p}\bH_j(r)\text{ is decreasing}\quad\text{for every }j\ge j_0 \text{ and } r\in [\sfrac{s_j}{t_j},1].
\end{equation}
Let us also recall that $1\in ]\sfrac{s_j}{t_j},3[$ for every $j$,  by item $\rm(i)$ of Proposition \ref{p:flattening}.

Now take any $j> j_0$. Assume that $t_j=s_{j-1}$. We use Proposition \ref{compH} with \eqref{N:rough} to see that
 \begin{align*}
 	C_1^2\bH_{j}(3)\ge C_1 \bmo^{(j-1)}\left(\frac{s_{j-1}}{t_{j-1}}\right)^{2-2\delta_2}\ge \bH_{j-1}\left(\frac{s_{j-1}}{t_{j-1}}\right).
 \end{align*}
 Let now $q\ge 1$. Using the inequality above and \eqref{fcads} twice we get
\begin{align*}
    t_{j-1}^{-(m-1)-p-q} \bH_{j-1}(1)\le C_1 \left(\frac{t_{j}}{t_{j-1}}\right)^{-(m-1)-p} t_{j-1}^{-(m-1)-p-q}\bH_j(3)\le C_1 3^p\left(\frac{t_{j}}{t_{j-1}}\right)^q t_j^{-(m-1)-p-q}\bH_j(1).
\end{align*}
Now, if $q$ is big enough (depending upon a geometric constant and on $p$, see item $\rm(i)$ of Proposition \ref{p:flattening}), we obtain
 \begin{equation}\label{ind1}
 	t_j^{-(m-1)-p-q}\bH_j(1)\ge t_{j-1}^{-(m-1)-p-q} \bH_{j-1}(1)\quad\text{for every } j\ge j_0\text{ such that $t_j=s_{j-1}$}.
 \end{equation}
Instead, in the case $t_j<s_{j-1}$, we have what follows.
By \eqref{e:finita2} and \eqref{fcads}, we see that there exists a constant $C$ (depending on $T$, but not on $j$) such that
\begin{equation*}
	C\bH_j(1)\ge 1\quad\text{for every } j\ge j_0\text{ such that $t_j<s_{j-1}$}.
\end{equation*}
In particular, up to enlarging $j_0$,
\begin{equation}\label{ind2}
	t_j^{-(m-1)-p-q}\bH_j(1)\ge 1\quad\text{for every } j\ge j_0\text{ such that $t_j<s_{j-1}$}.
\end{equation}

Now recall that $\bH_{j_0}(1)>0$ (by \eqref{e:finita2}), so we can conclude by induction, using \eqref{ind1} and \eqref{ind2}, that
\begin{align*}
\boldsymbol{\iota}\defeq	\inf_{j\ge j_0} t_j^{-(m-1)-p-q}\bH_j(1)>0.
\end{align*}
Using \eqref{fcads} again and Proposition \ref{small}, we have that for every $r\in [s_j,t_j] $ with $j\ge j_0$,
\begin{align*}
	\boldsymbol{\iota}&\le \left(\frac{r}{t_j}\right)^{-(m-1)-p-q}\bH_j(\sfrac{r}{t_j})\le4 \left(\frac{t_j}{r}\right)^{m+p+q}\int_{\cB_{\sfrac{r}{t_j}}}|N_j|^2 \\&\le\left(\frac{t_j}{r}\right)^{m+p+q} 16\int_{\B_{\sfrac{2r}{t_j}}}\dist^2(x,\cM_j)\dd\|T\|(x)\\
&=\left(\frac{t_j}{r}\right)^{m+p+q}\frac{16}{t_j^{m+2}}\int_{\B_{2r}}\dist^2(x,\cM)\dd\|T\|(x)\\&\le \frac{16}{r^{m+p+q}}\int_{\B_{2r}}\dist^2(x,\cM)\dd\|T\|(x),
\end{align*}
and hence, recalling \eqref{bgrvfd} (together with item $\rm(ii)$ of Propostition \ref{p:flattening}), we derive a contradiction to \eqref{mainthmeq}.
\subsection{Weak unique continuation}
Now we prove the second part of Theorem \ref{mainthm}. Let then $T$ and $\cM$ be as in Theorem \ref{mainthmeq}. We know that, for some $r\in ]0,1[$, \eqref{mainthmeq1} holds.
Notice also that it is enough to show that $\supp(T)\cap\B_1\subseteq\cM$. Indeed, by the Constancy Theorem, we obtain that for every $x\in\cM$, there exists $Q_x$, $r_x\in ]0,1[$ such that $T\res\B_{r_x}(x)=Q_x\a{\cM}\res\B_{r_x}(x)$. Since $T$ and $\cM$ have no boundary in $\B_1$ and $\cM\cap\B_1$ is connected, it is easy to  conclude that $\cM$ is orientable in $\B_1$ with the orientation inherited by the current $T$ and $Q_x=Q$ is independent of $x$. 

Since $T$ and $\cM$ have no boundary in $\bar\B_1$, they have no boundary in $\B_R$ for some $R>1$.
Define  the open set
$$
\mathcal{U}\defeq \{x\in\cM\cap\B_R:\exists r\in ]0,1[,Q\in\Z\text{ such  that } T\res \B_{r}(x)=Q\a{\cM}\res\B_r(x)\}\subseteq\cM,
$$
and let $\mathcal{U}'$ the connected component of $\mathcal{U}$ containing $0$ (which is non-empty by the first part of the theorem).
Let $x'\in\partial \mathcal{U}'$. Assume that $x'\in\reg(T)$. Then there exist $r=r_{x'}\in ]0,1[$, $Q=Q_{x'}\in\mathbb{Z}$, and a minimal surface $\cM'$ with no boundary in $\B_r(x')$ such that $$T\res \B_r(x')=Q\a{\cM'}\res \B_r(x').$$  Notice that $x'\in\partial \mathcal{U}'$ implies that $x'\in\supp(T)$, hence $x'\in\cM\cap\cM'$. Consider a sequence of points $\mathcal{U}\ni x_k\to x'$ and notice that, for every $k$, $\cM'$ coincides with $\cM$ in a neighbourhood of $x_k$. Up to shrinking $r$, we can assume that for every $y\in\B_r(x')$, $\cM\cap\B_r(y)$ is given by the graph of a smooth function $\Phii_y:B_{2r}(T_y\cM,y)\to\RR^{m+n}$ and, similarly, $\cM'\cap\B_r(y)$ is given by the graph of a smooth function $\Phii_y':B_{2r}(T_y\cM',y)\to\RR^{m+n}$. Choosing $y=x_k$ for $k$ big enough, we immediately see that $T_y\cM=T_y\cM'$ and that $\Phii_y$ and $\Phii_y'$ coincide on an open set. Since graphical minimal surfaces are real analytic by the results of \cite{Morreyanalytic, Morrey2} (see \cite[Section 2]{Lawson:1977aa}), $\Phii_y$ and $\Phii_y'$ are real analytic, and therefore $\Phii_y=\Phii_y'$ by classical unique continuation for real analytic functions. This implies that $x'\in\mathcal{U}$. Therefore, $\partial \mathcal{U}'\subseteq\sing(T)\cup \partial \B_R$.

Now consider $T'\defeq T\res(\mathcal{U}'\cap\B_1)=(T\res\B_1)\res\mathcal{U}'$. Let $x\in\supp(\partial (T\res \B_1-T'))$, and observe that $x\in\partial \B_1\cup\partial\mathcal{U}'$. If $x\in\mathcal{U}'$, then $x\in\partial\B_1$ and there exists a neighbourhood of $x$ with $T'=T\res\B_1$.  Otherwise, $x\in\partial \mathcal{U}'\cap \bar \B_1$. To summarize,  $\partial(T'-T\res\B_1)\subseteq\partial \mathcal{U}'\cap\bar\B_1\subseteq\sing(T)$. 
By the result of \cite{Almgren00} (or \cite{DS1,DS2,DS3}),
$$
\dim_{\mathcal{H}}(\supp(\partial(T'-T\res\B_1)))\le \dim_{\mathcal{H}}(\sing(T)\cap\B_R)\le m-2.
$$
By \cite[Theorem 2.2]{FedererInt}, $\partial(T'-T\res\B_1)=0$. Being $T\mres\B_1$ indecomposable, either $T'=0$ or $T'=T$. As the first alternative is not possible by the first part of the theorem, we must have $T=T'$, so that $\supp(T)\cap\B_1=\supp(T')\cap\B_1\subseteq\cM$ and we are done.

\part{Normal approximation}\label{sect:normapprox}
In this Part of the paper we will construct a Lipschitz approximation $N$ of $T$ on the normal bundle of $\cM$. The construction is based on a delicate Whitney cube decomposition adapted to $\cM$. We will also prove estimates that relate the excess of $T$ with the Dirichlet energy of $N$ and the squared distance between $\cM$ and $T$ with the $L^2$ norm of $N$. As mentioned in the Introduction, this Part is quite similar in spirit (and sometimes in letter) to \cite{DS2} -- we remark that the main difference is that in \cite{DS2} the center manifold is actually constructed (starting from the ``average of the sheets'' of the Lipschitz approximation of \cite{DS1}), while in our case the surface $\cM$ is given and we have to construct the approximation on the normal bundle of $\cM$. Loosely, Sections \ref{s:assumptions} and \ref{sectwh} correspond to (sub)Sections 1.2 and 1.3 in \cite{DS2}, Section \ref{sectnorm} to Section 2 there, Sections \ref{sectsep} and \ref{sectcomp} to Section 3, and finally Section \ref{proofnorm} to Sections 4 to 8.  

\section{Assumptions and parameters}\label{s:assumptions}
We start by defining the main assumptions of this paper. \textbf{Assumption \ref{ipotesi} and Assumption \ref{parametri} will be the standing assumptions for this whole part. We are going to use the current $T$ given by Lemma \ref{height}} throughout the paper. 

\bigskip

In Assumption \ref{ipotesi} and Assumption \ref{parametri} below, we are going to introduce several parameters, namely (respectively) $$\eps_0\quad\text{and}\quad\gamma_1,\gamma_2,c_s,\beta_2,\delta_2,M_0,N_0,C_e,C_h.$$ 
In what follows, $C_0$ denotes a constant that depends only on $m,n,Q$, and may vary throughout the paper. We call such type of constants ``geometric constants''. 
Similarly, $C_1$ denotes a constant that depends only on $m,n,Q$, and the parameters of Assumption \ref{parametri} below \emph{except} $\eps_0$, and may vary throughout the paper.
The dependence of other constants upon the various parameters $p_i$ will be highlighted using the notation $C = C (p_1, p_2, \ldots)$, and it is implicit that such constants depend also on $m,n,Q$.
In particular,
\begin{equation}\label{constants}
	C_0=C(m,n,Q)\quad\text{and}\quad C_1=C(m,n,Q,\gamma_1,\gamma_2,c_s,\beta_2,\delta_2,M_0,N_0,C_e,C_h)
\end{equation}
Notice that there is no dependence on $\eps_0$ in $C_1$ (and that $m,n,Q$ determine $\gamma_1,\gamma_2,c_s,\beta_2,\delta_2$).   Also, whenever we encounter a term of the type $C_1\bmo$, we can make it arbitrarily small, taking $\eps_0$ small enough, depending on $C_1$ and therefore upon all other parameters.

\begin{ass}\label{ipotesi}
We assume that
$\cM  \subseteq \R^{m+n}$ is an $m$-dimensional $C^3$ Riemann manifold with no boundary in $\bB_{6\sqrt{m}}$, with $0\in \cM$ and $T_0\cM=\pi_0$. We moreover assume that  for every $p\in\cM$, $\cM$, in $\bB_{6\sqrt{m}}$ is given by the graph of a $C^3$ map $$\bPsi_p: B_{7\sqrt{m}}(p,T_p\cM)  \to T_p\cM^\perp.$$ 
We denote by $\mathbf{c} (\cM)$ the
number $\sup_{p\in \cM\cap \bB_{6\sqrt{m}}}\|D^2\bPsi_p\|_{C^1}$. 

Also, $T^0$ is an $m$-dimensional area minimizing current in $\RR^{m+n}$ with support in $\bar\B_{6\sqrt{m}}$. Moreover, $T^0$ satisfies, for some positive integer $Q$,
\begin{gather*}
\Theta (0, T^0) = Q\quad\text{and}\quad	\partial T^0 \res \B_{6\sqrt{m}} = 0,\\
\|T^0\|(\B_{6\sqrt{m}{\rho}})\le (\omega_mQ(6\sqrt{m})^m+\eps_0^2)\rho^m\quad\text{for every $\rho\in [0,1]$}
\end{gather*}
where $\eps_0$ is a strictly positive number whose choice will be specified later.
Finally, the assumption 
\begin{equation}
	\bmo := \max \{\mathbf{c} (\cM)^2, \bE(T^0,\B_{6\sqrt{m}},\pi_0)\} \le \eps_0^2 \le 1.\notag
\end{equation}
will be in force.
\fr
\end{ass}
We are going to use throughout the following trivial yet crucial remark:
\begin{equation}\notag
	\sup_{p\in \cM\cap \bB_{6\sqrt{m}}}\|\bPsi_p\|_{C^3}\le C_0 \sup_{p\in \cM\cap \bB_{6\sqrt{m}}}\|D^2\bPsi_p\|_{C^1}\le C_0\bmo^{\sfrac{1}{2}}.
\end{equation}
As a notation, we suppress the subscript $p$ for the point $0$ in $\Psi_p$,  so that we have
$$
\bPsi:\pi_0\cap \bB_{7\sqrt{m}}\to\pi_0^\perp.
$$
Also, setting $$\Phii(x) := (x,\bPsi(x))=(x,\bPsi_0(x))\quad\text{for $x\in \pi_0\cap \B_{7\sqrt{m}}$},$$
notice that 
\begin{equation}\label{asdcdsavs}
\|D\Phii\|_{C^0}\le \sqrt{1+C_0\bmo},
\end{equation}
and therefore
\begin{equation}\label{cdscas}
		|\Phii(x)-\Phii(y)|\le \sqrt{1+C_0\bmo}|x-y|\quad\text{ for every $x,y\in B_{5\sqrt{m}}$}.
\end{equation}
Similar notation will be used for $\Phii_p$. 
Here and after, we are abusing slightly the notation, writing $(x,y)$ instead on $x+y$, for $x\in \pi$, $y\in\pi^\perp$, and similarly, we will write $\pi\times\pi^\perp$ instead of $\pi+\pi^\perp$, for $\pi$ $m$-plane.

\begin{lem}\label{height} Let $T^0$ be as in Assumption \ref{ipotesi} and assume that $\eps_0$ is small enough (depending upon a geometric quantity). Then, if we set $T\defeq T^0\mres \B_{\sfrac{23\sqrt{m}}{4}}$:
	\begin{gather}\notag
		\partial T\res \bC_{\sfrac{11\sqrt{m}}{2}}=0\\\notag
		(\p_{\pi_0})_\sharp T\res \bC_{\sfrac{11\sqrt{m}}{2}}=Q\a{B_{\sfrac{11\sqrt{m}}{2}}}\\\notag
				\supp(T)\cap \bC_{{5\sqrt{m}}}\subseteq \big\{(x,y)\in\RR^{m}\times\RR^n: |y|< C_0 \bmo^{1/2}\big\}.
	\end{gather}
	In particular, for every $x\in B_{\sfrac{11\sqrt{m}}{2}}$, there exists $p\in\supp(T)$ with $\p_{\pi_0}(p)=x$.
\end{lem}
\begin{proof}
	This is \cite[Lemma 1.6]{DS2}, we have only strengthened the last conclusion using the improved  $L^\infty$ estimate, see \cite{AllardFirst} or \cite{SpolAlm} (but still, with the same proof of \cite{DS2}).
\end{proof}
From now, we will always work with the current $T$ of Lemma \ref{height}.
\begin{ass}\label{parametri}
$\gamma_1,\gamma_2$ are defined as follows: $\gamma_1$ is the constant of \cite[Theorem 2.4]{DS1}, which essentially regulates the `goodness' of the Lipschitz approximation constructed there, and simply
 $$\gamma_2 \defeq \frac{\gamma_1}{4}.$$
 $c_s$ is defined as \begin{equation}\label{defncs}
 	 c_s \defeq \frac{1}{64\sqrt{m}}.
 \end{equation}
	$\beta_2,\delta_2$ are defined as follows:
	\begin{equation}\label{freavfad}
				\beta_2 \defeq 4\delta_2 \defeq \min \left\{\frac{1}{2m}, \frac{\gamma_1}{100}\right\}.
	\end{equation}
	The parameters $M_0,C_e,C_h \in [64,\infty)$ and $N_0\in\NN\setminus\{0\}$ are not fixed but are subject to further restrictions in the various statements, respecting the following ``hierarchy''.   
	\begin{itemize}
		\item[(a)]$M_0$ is larger than a geometric constant or larger than a constant $C (\delta_2)$;
		\item [(b)]$N_0$ is larger than $C (\beta_2, \delta_2, M_0)$ and satisfies
	\begin{equation}\label{dcasc}
	  {m} M_0 2^{10-N_0} \le 1;
\end{equation}
		\item [(c)]$C_e$ is larger than $C(\beta_2, \delta_2, M_0, N_0)$;
		\item[(d)]$C_h$ is larger than $C(\beta_2, \delta_2, M_0, N_0, C_e)$;
		\item[(e)] $\eps_0$  (see Assumption \ref{ipotesi}) is smaller than $C(\beta_2, \delta_2, M_0, N_0, C_e, C_h)$ and so small that the conclusions of Remark \ref{intorno_proiezione} below and Lemma \ref{height} hold, and finally so that, with respect to \eqref{asdcdsavs},  $\sqrt{1+C_0 \bmo}\le 2$.
\fr
	\end{itemize}
\end{ass}
Notice that by \eqref{defncs} and \eqref{dcasc}, 
\begin{equation}\label{fdasc}
	\sqrt{m}2^{-N_0} < c_s.
\end{equation}
To simplify our exposition, for smallness conditions on $\eps_0$ as in (e) we will use the sentence ``$\eps_0$ is sufficiently small'', understanding that the smallness depends upon all the other parameters.

\begin{rem}\label{intorno_proiezione}
	If $\eps_0$ is sufficiently small, depending upon a geometric constant, then the following hold.
	Set $$\bU\defeq \{x\in\RR^{m+n}:\exists! y=\p(x)\in \cM\cap \bC_{\sfrac{9\sqrt{m}}{2}}\text{ with }|x-y|<1\text{ and }(x-y)\perp \cM\}.$$
	This defines the map $\p:\bU\rightarrow\cM\cap\bC_{\sfrac{9\sqrt{m}}{2}}$. We assume that $\p$ extends to a smooth map defined on $\bar U$ with $\p^{-1}(y)=y+\bar B_{1}(0,(T_y\cM)^\perp)$ for every $y\in\cM\cap \bC_{\sfrac{9\sqrt{m}}{2}}$.  Denote by $\partial_l \bU\defeq\p^{-1}(\partial(\cM\cap\bC_{\sfrac{9\sqrt{m}}{2}}))$ the \textit{lateral boundary} of $\bU$.

	Finally, 
	\begin{equation}\label{vfcdsa}
	|\p_{\pi_0}(x)-\p_{\pi_0}(y)|\le d(x,y)\le\frac{19}{18}|\p_{\pi_0}(x)-\p_{\pi_0}(y)| \quad\text{for every }x,y\in\cM\cap\bB_{6\sqrt{m}},
	\end{equation} 
	provided $\eps_0$ is small enough, depending upon a geometric constant.\fr
\end{rem}

\section{Whitney decomposition} \label{sectwh}
We specify next some notation which will be recurrent in the paper when dealing with cubes in $\pi_0$.
For each $j\in \N$, $\sC^j$ denotes the family of closed cubes $L$ in $\pi_0$ of the form 
\begin{equation}\label{e:cube_def}
	[a_1, a_1+2\ell] \times\ldots  \times [a_m, a_m+ 2\ell] \times \{0\}\subseteq \pi_0 ,
\end{equation}
where $2\ell = 2^{1-j} \eqdef 2\ell (L)$ is the side-length of the cube, 
$a_i\in 2^{1-j}\Z$ for every $i$ and we require in
addition $-4 \le a_i \le a_i+2\ell \le 4$. 
To avoid cumbersome notation, we will usually drop the factor $\{0\}$ in \eqref{e:cube_def} and treat each cube, its subsets and its points as subsets and elements of $\mathbb R^m$. Thus, for the {\em center $x_L$ of $L$} we will use the notation $x_L=(a_1+\ell, \ldots, a_m+\ell)$, although the precise one is $(a_1+\ell, \ldots, a_m+\ell, 0, \ldots , 0)$.
Next we set $\sC \defeq \bigcup_{j\in \N} \sC^j$. 
If $H$ and $L$ are two cubes in $\sC$ with $H\subseteq L$, then we call $L$ an {\em ancestor} of $H$ and $H$ a {\em descendant} of $L$. When in addition $\ell (L) = 2\ell (H)$, $H$ is {\em a son} of $L$ and $L$ {\em the father} of $H$.

\begin{defn} A Whitney decomposition of $[-4,4]^m\subseteq \pi_0$ consists of a closed set $\bGam\subseteq [-4,4]^m$ and a family $\mathscr{W}\subseteq \sC$ satisfying the following properties:
	\begin{itemize}
		\item[(w1)] $\bGam \cup \bigcup_{L\in \mathscr{W}} L = [-4,4]^m$ and $\bGam$ does not intersect any element of $\mathscr{W}$;
		\item[(w2)] the interiors of any pair of distinct cubes $L_1, L_2\in \mathscr{W}$ are disjoint;
		\item[(w3)] if $L_1, L_2\in \mathscr{W}$ have non-empty intersection, then $\frac{1}{2}\ell (L_1) \le \ell (L_2) \le 2 \ell (L_1)$.
	\end{itemize}
\end{defn}
\begin{defn}[Refining procedure]
	For $L\in \sC$ we set 
	\begin{alignat*}{1}
		r_L&\defeq M_0 \sqrt{m} \ell (L),\\
		p_L&\defeq \Phii(x_L),\\
  		\pi_L&\defeq T_{p_L}\cM,\\
		\bC_L &\defeq \bC_{64 r_L} (p_L,\pi_L).
	\end{alignat*}We next define the families of cubes $\sS\subseteq\sC$ and $\sW = \sW_e \cup \sW_h \cup \sW_n \subseteq \sC$ with the convention that
	$\sS^j = \sS\cap \sC^j, \sW^j = \sW\cap \sC^j$ and $\sW^j_{\square} = \sW_\square \cap \sC^j$ for $\square = e,h,n$. We define $\sW^j = \sS^j= \emptyset $ for $j < N_0$. We proceed with $j\ge N_0$ inductively: if { no ancestor of $L\in \sC^j$ is in $\sW$}, then 
	\begin{itemize}
		\item[(EX)] $L\in \sW^j_e$ if $\bE (T, \bC_L, \pi_L) > C_e \bmo \ell (L)^{2-2\delta_2}$;
		\item[(HT)] $L\in \sW_h^j$ if $L\not \in \mathscr{W}_e^j$ and $\bs (T, \bC_L,p_L+\pi_L) > C_h \bmo^{\sfrac{1}{2}} \ell (L)^{1+\beta_2}$;
		\item[(NN)] $L\in \sW_n^j$ if $L\not\in \sW_e^j\cup \sW_h^j$ but it intersects an element of $\sW^{j-1}$;
	\end{itemize}
	if none of the above occurs, then $L\in \sS^j$.
	We finally set
	\begin{equation}\notag
		\bGam\defeq [-4,4]^m \setminus \bigcup_{L\in \sW} L = \bigcap_{j\ge N_0} \bigcup_{L\in \sS^j} L.
	\end{equation}
\end{defn}
Observe that, if $j>N_0$ and $L\in \sS^j\cup \sW^j$, then necessarily its father belongs to $\sS^{j-1}$.
\begin{rem}
	It would have been more appropriate to call the second stopping condition (SEP) instead of (HT) and to use subscript ``s'' (separation) instead of ``h'' (height)  for $\sW_h$ and $C_h$. However, our choice aims at keeping the same terminology of \cite{DS2,DS3} (see, in particular \cite[Definiton 1.10]{DS2}) and hence to avoid confusion when we recall parts of \cite{DS2,DS3} \textit{verbatim}.
Moreover, for condition (HT) we have used $\bmo^{\sfrac{1}{2}}$ instead of $\bmo^{\sfrac{1}{2m}}$ of \cite{DS2}. This is because we rely on the improved height bound \cite[Theorem 1.5]{SpolAlm} in place of \cite[Theorem A.1]{DS2}. This causes no significant difference in the presentation (however notice that due to this difference, (S3) of Proposition \ref{p:separ} is in integral form, whereas the corresponding (S3) of \cite[Proposition 3.1]{DS2} is in the stronger pointwise form). \fr
\end{rem}

For the following definition, we introduce a slight difference with respect to \cite[Definition 1.18]{DS2}: we do not intersect $H$ with $[-\sfrac{7}{2},\sfrac{7}{2}]$. The reason is that we have a description of $\cM$ in a region that is larger than $[-4,4]^m$.
\begin{defn}[Whitney regions]
	We call $(\bGam, \sW)$ the {\em Whitney decomposition associated to $\cM$}. 
	We call 	$\Phii (\bGam)$ the {\em contact set}.
	Moreover, to each $L\in \sW$ we associate a {\em Whitney region} $\cL$ on $\cM$ as follows:
	\begin{itemize}
		\item[(WR)] $\cL := \Phii (H)$, where $H$ is the cube concentric to $L$ with $\ell (H) = \sfrac{17}{16} \ell (L)$.
	\end{itemize}
\end{defn}

\begin{prop}[Tilting of planes]\label{prop:tilting}
	Let Assumption \ref{ipotesi} and Assumption \ref{parametri} hold, and assume that $\eps_0$ is sufficiently small (depending upon all other parameters). Then,  the following hold. 
	\begin{itemize}
		\item[\rm(i)]$\supp(T)\cap\bC_L\subseteq\bC_{5\sqrt{m}}$ for every $L\in\sW\cup\sS$.
	\end{itemize}
	Moreover, for every $H,L\in \sW\cup\sS$, 
	\begin{itemize}\setcounter{enumi}{1}
		\item[\rm(ii)] $|\vec\pi_L-\vec\pi_H|\le C_0\bmo^{\sfrac{1}{2}}|x_L-x_H|$ and, if $H\subseteq L$, $|\p_{\pi_H^\perp}(p_L-p_H)|\le C_0\bmo^{\sfrac{1}{2}}\ell(L)^2$,
		\item[\rm(iii)] $|\vec\pi_H-\vec\pi_0|\le C_0 \bmo^{\sfrac{1}{2}}$,
		\item[\rm(iv)]  $\supp(T)\cap\bC_H\subseteq \bC_L$ if $H\subseteq L$,
		\item [\rm(v)] $\bs(T,\bC_{36r_L}(p_L,\pi_H),p_H+\pi_H)\le C_1\bmo^{\sfrac{1}{2}}\ell(L)^{1+\beta_2}$ and $\supp(T)\cap \bC_{36r_L}(p_L,\pi_H)\subseteq\bC_L$ if $H\subseteq L$.
	\end{itemize}
	Finally, let $L\in\sW\cup\sS$, take $z\in \bC_L$, and set $z'\defeq p_L+D\Phii(x_L)(\p_{\pi_0}(z)-x_L) $. Then,
	\begin{itemize}
		\item[\rm(vi)] $	|z- z'|\le C_0 |\p_{\pi_L^\perp}(z-p_L)|$. In particular, $$|z-\Phii(\p_{\pi_0}(z))|\le C_0 |\p_{\pi_L^\perp}(z-p_L)|+\bmo^{\sfrac{1}{2}}|\p_{\pi_0}(z)-x_L|^2.$$
	\end{itemize}
\end{prop}
\begin{prop}[Whitney decomposition]\label{p:whitney}Assume the hypothesis of Proposition \ref{prop:tilting}, that $C_e ,C_h\ge C(M_0,N_0)$, and that $\eps_0$ is sufficiently small (depending upon all other parameters). Then $(\bGam, \mathscr{W})$ is a Whitney decomposition of $[-4,4]^m \subseteq \pi_0$.
	Moreover,	
	\begin{equation}\label{e:prima_parte}
		\sW^{j} = \emptyset \qquad \mbox{for all $j\le N_0+6$}
	\end{equation}
	and the following	estimates hold:
	\begin{alignat}{3}
		\label{eq0}
		&\bE (T, \bC_J,\pi_L) \le C_e \bmo \ell (J)^{2-2\delta_2}
		\quad  \text{for every } J\in \sS\cup\sW_h\cup\sW_n, \\
		\label{eq1}
		&\bE (T, \bC_J,\pi_L) \le C_e \bmo \ell (J)^{2-2\delta_2} \quad \text{and}\quad
		\bs (T, \bC_J,p_J+\pi_J) \le C_h \bmo^{\sfrac{1}{2}} \ell (J)^{1+\beta_2}
		\quad &&\text{for every } J\in \sS\cup\sW_n, \\
		\label{eq2}
		&\bE (T, \bC_L,\pi_L) \le C_1\bmo \ell (L)^{2-2\delta_2}\quad \text{and}\quad
		\bs (T, \bC_L,p_L+\pi_L) \le C_1 \bmo^{\sfrac{1}{2}} \ell (L)^{1+\beta_2}
		\quad &&\text{for every } L\in\sS\cup \sW,
	\end{alignat}
	and
	\begin{alignat}{3}
		\label{eq4}	&\supp(T)\cap\bC_L\subseteq \B_{128 r_L}(p_L)\quad&&\text{for every } L\in \sW\cup \sS,\\\label{eq3}
		&\|T\| (\bC_L)\le C(M_0) \ell(L)^m \quad&&\text{for every } L\in \sW\cup \sS.
	\end{alignat}
\end{prop}
\begin{cor}\label{c:cover}Assume the hypotheses of Proposition \ref{p:whitney} and that $\eps_0$ is small enough (depending upon all other parameters). Then, 
	\begin{itemize}
		\item [\rm(i)] $\supp(\partial(T\cap\bU ))\subseteq\partial_l \bU$, $\supp (T\mres[-4,4]^m\times \RR^n)\subseteq\bU$ and $\p_{\sharp}(T\mres \bU)=Q\a{\cM\cap \bC_{\sfrac{5\sqrt{m}}{2}}}$;
		\item [\rm(ii)]$\supp  (T)\cap \p^{-1}(\Phii (x)) \subseteq
\big\{y : |\Phii (x)-y|\le C_1\bmo^{\sfrac{1}{2}} 
\ell (L)^{1+\beta_2}\big\}$ for every $x\in L\in \sS\cup\sW$;
\item [	\rm(iii)] 
$\supp  (T)\cap \p^{-1}(\Phii(x)) \subseteq \{\Phii(x)\}$ for every $x\in \bGam$.
	\end{itemize}
\end{cor}

\section{Normal approximation}\label{sectnorm}

\begin{defn}[$\cM$-normal approximation]\label{d:app}
An {\em $\cM$-normal approximation} of $T$ is given by a pair $(\cK, F)$ such that
\begin{itemize}
	\item[(A1)] $F: \cM\to \Iq (\bU)$ is Lipschitz (with respect to the geodesic distance on $\cM$) and takes the special form 
	$F (x) = \sum_i \a{x+N_i (x)}$, with $N_i (x)\perp T_x \cM$ 
	for every $x$ and $i$.
	\item[(A2)] $\cK\subseteq \cM$ is closed, contains $\Phii \big(\bGam\cap [-4, 4]^m\big)$ and $\bT_F \res \p^{-1} (\cK) = T \res \p^{-1} (\cK)$.
\end{itemize}
The map $N = \sum_i \a{N_i}:\cM \to \Iq (\bU)$ is {\em the normal part} of $F$.
\end{defn}

\begin{thm}[Local estimates for the $\cM$-normal approximation]\label{t:approx} Assume the hypotheses of Proposition \ref{p:whitney} and  that  $\eps_0$ is sufficiently small (depending upon all other parameters). Then	there is an $\cM$-normal approximation $(\cK, F)$  such that, for every Whitney region $\cL$ associated to
	a cube $L\in \sW$, 
	\begin{gather}
		\Lip (N|
		_\cL) \le C_1 \bmo^{\gamma_2} \ell (L)^{\gamma_2} \quad\mbox{and}\quad  \|N|
		_\cL\|_{C^0}\le C_1 \bmo^{\sfrac{1}{2}} \ell (L)^{1+\beta_2},\label{e:Lip_regional}\\
		\mathcal{H}^m(\cL\setminus \cK)+ \|\bT_F - T\| (\p^{-1} (\cL)) \le C_1 \bmo^{1+\gamma_2} \ell (L)^{m+2+\gamma_2},\label{e:err_regional}\\
		\int_{\cL} |DN|^2 \le C_1 \bmo \ell (L)^{m+2-2\delta_2}\quad\mbox{and}\quad \int_{\cL} |N|^2 \le C_1 \bmo \ell (L)^{m+2+2\beta_2}.\label{e:Dir_regional}
	\end{gather}
\end{thm}


\section{Separation and splitting before tilting}\label{sectsep}
We now analyse more in detail the consequences of the various stopping conditions for the cubes in $\sW$. 

\begin{prop}[Separation]\label{p:separ}

	Assume the hypotheses of Theorem \ref{t:approx}, that
	$C_h \ge C(M_0,C_e)$ and that $\eps_0$ is sufficiently small (depending upon all other parameters). Then the following conclusions hold for every $L\in \sW_h$:
	\begin{itemize}
		\item[\rm (S2)] $L\cap H= \emptyset$ for every $H\in \sW_n$
		with $\ell (H) \le \sfrac{1}{2} \ell (L)$,
	\end{itemize}
and, for any $B_{\sfrac{\ell(L)}{4}}(q)\subseteq B_{(\sqrt{m}+1)\ell(L)}(x_L)$, setting $\Omega\defeq \Phii(B_{\sfrac{\ell(L)}{4}}(q))$, we have
	\begin{itemize}
		\item [\rm(S3)] $ C_h^2 \bmo\ell(L)^{m+2+2\beta_2}\le C_0\int_{\Omega}|N|^2$.
	\end{itemize}
\end{prop}

\begin{prop}[Splitting]\label{p:splitting} Assume  the hypotheses of Proposition~\ref{p:separ}, that $N_0\ge C(M_0)$ and that $\eps_0$ is sufficiently small (depending upon all the other parameters). Then,
	for every $L\in \sW_e$, $q\in \pi_0$ with $\dist (L, q) \le4 c_s^{-1} \ell (L)$ and $\Omega\defeq \Phii (B_{\sfrac{\ell(L)}{4}} (q))$, 
	\begin{align}
		&C_e \bmo \ell(L)^{m+2-2\delta_2} \le \ell (L)^m \bE (T, \bC_L,\pi_L) \le C_1 \int_\Omega |DN|^2 ,\label{e:split_1}\\
		&\int_{\cL} |DN|^2 \le C_1 \ell (L)^m \bE (T, \bC_L,\pi_L) \le C_1^2 \ell (L)^{-2} \int_\Omega |N|^2 , \label{e:split_2}
	\end{align}
	where $\cL$ is the Whitney region associated to the cube $L$.
\end{prop}

\section{Comparison estimates}\label{sectcomp}
\begin{prop}\label{comp1}
 Assume  the hypotheses of  Proposition \ref{p:splitting} and that $\eps_0$  is sufficiently small (depending upon all other parameters). 
 Assume that there exists $L\in \sW$ with, setting $s\defeq c_s^{-1}\ell(L)$,
 \begin{enumerate}
 	\item [\rm (a)]  $c_s^{-1}\ell(H)\le s$ for every $H\in\sW$ with $c_s^{-1}\ell(H)\ge \dist(0,H)$;
 	\item [\rm (c)]  $s\ge  \dist(0,L)$.
 \end{enumerate}
 Then, $L\in\sW_e$ and
 for every $\Omega\defeq \Phii (B_{c_ss/4} (q))$, where $q\in \pi_0$ with $\dist(L,q)\le 3 s$,
 \begin{equation}\label{comp1eq}
 \bmo s^{m+4-2\delta_2} \le C_1	\int_{\p^{-1}(\Omega)}\dist^2(x,\cM)\dd\|T\|(x).
 \end{equation}
\end{prop}

\begin{prop}\label{comp2}
	Assume  the hypotheses  of Proposition \ref{comp1}. Then,
	\begin{equation}\label{eqcomp2}
		\int_{\p^{-1}(\cB_{\sfrac{21}{8}}\setminus \cB_{\sfrac{19}{8}})}\dist^2(x,\cM)\dd\|T\|(x)\le C_1\bigg(\bmo^{1+\gamma_2}\sup_{L\in\sW:L\cap (B_{\sfrac{11}{4}}\setminus B_{\sfrac{9}{4}})\ne\emptyset}\ell(L)+\int_{{\cB_{\sfrac{21}{8}}\setminus \cB_{\sfrac{19}{8}}}}|N|^2\bigg),
	\end{equation}
	where is understood that the supremum taken over the empty set in \eqref{eqcomp2} is $0$.
\end{prop}

\section{Proofs of the results of Part \ref{sect:normapprox}}\label{proofnorm}

\subsection{Proof of the results of Section \ref{sectwh}}
\begin{proof}[Proof of Proposition \ref{prop:tilting}]
	Item $\rm(ii)$ and item $\rm(iii)$ follow from  the bound on $\|\bPsi\|_{C^2}$ and $\|\Phii_{p_H}\|_{C^2}$.
	
	We prove item $\rm(i)$. Take $z=(x,y)\in\supp(T)\cap\bC_L\subseteq\RR^{m+n}$ and write similarly $p_L=(x_L,y_L)\in\RR^{m+n}$, we know $|y|<6\sqrt{m}$, $|x_L|\le 4\sqrt{m}$, and  $|y_L|\le C_0\bmo^{\sfrac{1}{2}}$. We have, by item $\rm(iii)$,  provided that $\eps_0$ is smaller than a geometric constant,
	\begin{align*}
		|x-x_L|&=|\p_{\pi_0}(z-p_L)|\le  |\p_{\pi_L}(z-p_L)|+|\p_{\pi_0}-\p_{\pi_L}||z-p_L|
		\\&\le 64r_L+C_0\bmo^{\sfrac{1}{2}}(|x-x_L|+|y|+|y_L|)
		\le 2^{-4}+\sfrac{1}{2}|x-x_L|+2^{-4},
	\end{align*}
	where we used $64r_L\le 2^{-4}$ by  \eqref{dcasc}. Hence $|x|\le |x_L|+|x-x_L|\le 4\sqrt{m}+2^{-2}\le 5\sqrt{m}$.
	
	For item $\rm(iv)$, we can assume that $\ell(H)\le \sfrac{1}{2}\ell(L)$. Take $z\in\supp(T)\cap\bC_H$, hence by item $\rm(i)$ (and Assumption \ref{ipotesi}), $|z|\le C_0$. Then,
	\begin{align*}
		|\p_{\pi_L}(z-p_L)|&\le |p_H-p_L|+ 	|\p_{\pi_L}-\p_{\pi_H}||z-p_H|+|\p_{\pi_H}(z-p_H)|
		\\&\le 2\sqrt{m} \ell(L)+C_0\bmo^{\sfrac{1}{2}}\ell(L)+64 r_H
		\le 64 r_L,
	\end{align*}
	provided that $\eps_0$ is smaller than a geometric constant.

	We prove the second claim of item $\rm(v)$.  
	First, we notice that, if $\eps_0$ is sufficiently small (depending on $M_0,N_0,C_h$),
	\begin{equation}\label{lcdanco}
		\supp(T)\cap\bC_{J}\subseteq\B_{128 r_J}(p_J)	\quad\text{for every }J\in\sS.
	\end{equation}
	Indeed, as $J\in\sS$, $\bs(T,\bC_{J},p_{J}+\pi_{J})\le C_h\bmo^{\sfrac{1}{2}}\ell(J)^{1+\beta_2}$, so that, if $z\in\supp(T)\cap \bC_J$,
	\begin{align*}
		|z-p_J|^2\le (64r_J)^2+ (C_h\bmo^{\sfrac{1}{2}}\ell(J)^{1+\beta_2})^2\le (128 r_J)^2,
	\end{align*}
	provided $\eps_0$ is small enough.
	
	Fix $H\in \sC^j$, for some $j$. Consider the chain of ancestors $L^{N_0}\supseteq L^{N_0+1}\supseteq \dots \supseteq L^j$, where $L^i\in \sC^i$ for every $i$, and with $L^j=H$. We prove the claim by induction on $j$. The base case is $j=N_0$ and we have to prove the claim up to $j=i-1$ (as the claim is trivial for $H=L$). Assume then $j=N_0<i$. 
	Take $z\in\supp(T)\cap\bC_{5\sqrt{m}}$, by Lemma \ref{height}, $|\p_{\pi_0^\perp}(z)|\le C_0\bmo^{\sfrac{1}{2}}$, so that, recalling also the bound on $\|\Psi\|_{C^0}$ and item $\rm(iv)$,
	we compute
	\begin{align*}
		|\p_{\pi_H^\perp}(z-p_{L^{N_0}})|&\le |\p_{\pi_H^\perp}-\p_{\pi_0^\perp}||z-p_{L^{N_0}}|+|\p_{\pi_0^\perp}(z)|+|\p_{\pi_0^\perp}(p_{L^{N_0}})|
		\\&\le C_0\bmo^{\sfrac{1}{2}}\le C(M_0,N_0)\bmo^{\sfrac{1}{2}}r_{L^{N_0}},
	\end{align*}
	so that, for $z\in\supp(T)\cap\bC_{36 r_{L^{N_0}}}(p_{L^{N_0}},\pi_H)$,
	\begin{align*}
		|z-p_{L^{N_0}}|^2\le (C(M_0,N_0)\bmo^{\sfrac{1}{2}}r_{L^{N_0}})^2+(36r_{L^{N_0}})^2\le (64 r_{L^{N_0}})^2
	\end{align*}
	if $\eps_0$ is small enough (depending on $M_0$ and $N_0$). In particular,
 $$
 \supp(T)\cap \bC_{36 r_{L^{N_0}}}(p_{L^{N_0}},\pi_H)\subseteq \bC_{L^{N_0}}.
 $$
 Now we prove the inductive step. Take then $L^{i+1}$, with $N_0<i+1<j$. 
	Notice that 
	\begin{equation}\label{cdsacasaa}
		\supp(T)\cap\bC_{36r_{L^{i+1}}}(p_{L^{i+1}},\pi_H)\subseteq 	\supp(T)\cap\bC_{36r_{L^i}}(p_{L^i},\pi_H)\subseteq\bC_{L^i},
	\end{equation}
	where the first inclusion is due to the inequality $|p_{L^{i+1}}-p_{L_i}|\le 2\sqrt{m}\ell(L^i)$ and the second inclusion is the inductive assumption.
	For $z\in\supp(T)\cap\bC_{36 r_{L^{i+1}}}(p_{L^{i+1}},\pi_H)$, we compute, using  item $\rm(ii)$ and  \eqref{cdsacasaa} together with \eqref{lcdanco} (recall $L^i\in\sS$),
	\begin{align*}
		|\p_{\pi_H^\perp}(z-p_{L^{i+1}})|&\le |\p_{\pi_H^\perp}(p_{L^{i+1}}-p_{L^i})|+|\p_{\pi_H^\perp}-\p_{\pi^\perp_{L^{i}}}||z-p_{L^{i}}|+|\p_{\pi_{L^i}^\perp}(z-p_{L^i})|
		\\&\le 2\sqrt{m}\ell(L^i)+C_0\bmo^{\sfrac{1}{2}}\ell(L^i)128r_{L^i}+C_h\bmo^{\sfrac{1}{2}}\ell(L^i)^{1+\beta_2}
		\le r_{L^{i+1}},
	\end{align*}
	provided that $\eps_0$ is small enough (depending on $M_0,N_0,C_h$). This implies that 
	$$
	\supp(T)\cap\bC_{36 r_{L^{i+1}}}(p_{L^{i+1}},\pi_H)\subseteq \B_{\sqrt{1+36^2}r_{L^{i+1}}}(p_{i+1})\subseteq\bC_{L^{i+1}},
	$$
	which is the claim. 
	
	We show now the first claim of $\rm(v)$. Take $H\subseteq L$ as in the statement and set $J\defeq L$ if $L\in\sS$, otherwise set $J$ equal to the father of $L$, notice that in either case $J\in\sS$ and $\ell(J)\le 2\ell(L)$. Take $z\in\supp(T)\cap \bC_{36 r_L}(p_L,\pi_H)$. Now, by the second claim of item $\rm (v)$ and item $\rm(iv)$, $z\in \bC_J$. Using items $\rm(i)$ and  $\rm(iii)$, \eqref{lcdanco} and finally the fact that $J\in\sS$,
	\begin{align*}
		|\p_{\pi_H^\perp}(z-p_H)|&\le 	|\p_{\pi_H^\perp}(p_J-p_H)|+	|\p_{\pi_H^\perp}-\p_{\pi_J^\perp}||z-p_J|+|\p_{\pi_J^\perp}(z-p_J)|
		\\&\le C_0\bmo^{\sfrac{1}{2}}\ell(J)^2+C_0\bmo^{\sfrac{1}{2}}	\ell(J)128r_J+C_h\bmo^{\sfrac{1}{2}}\ell(J)^{1+\beta_2}\le C_1\ell(L)^{1+\beta_2}.
	\end{align*}

	We show item $\rm (vi)$.  Set $x\defeq\p_{\pi_0}(z)$. Notice that by the definition of  $\pi_L$ there exists $x''\in\pi_0$ such that, for $ z''\defeq p_L+D\Phii(x_L)(x''-x_L)$, it holds $$|\p_{\pi_L^\perp}(z-p_L)|=|z-z''|.$$
	By item $\rm(iii)$, thanks to the fact that $\p_{\pi_L }(z-p_L)=z''-p_L=\p_{\pi_L }(z''-p_L)$, we have
	\begin{equation}\label{cdss}
		|x-x''|=|\p_{\pi_0}(z-z'')|\le|\p_{\pi_0}-\p_{\pi_L}||z-z''|+ |\p_{\pi_L}(z-z'')|\le C_0 \bmo^{\sfrac{1}{2}}|z-z''|.
	\end{equation}Therefore,
	\begin{align*}
		|z-z'|\le |z-z''|+|z''-z'|\le |z-z''|+|D\Phii (x_L)(x-x'')|\le |z-z''|+C_0 |x-x''|
	\end{align*}
	whence the first claim by \eqref{cdss} and the choice of $z''$. 
	The second conclusion of item $\rm(vi)$ follows from the first conclusion, taking into account the second order Taylor expansion of $\Phii$ around $x_L$.
\end{proof}

\begin{proof}[Proof of Proposition \ref{p:whitney}]
	We first prove \eqref{e:prima_parte}. Take $L\in \sC^j$ with $N_0\le j\le N_0+6$, that is $$2^{-N_0-6}\le \ell(L)\le 2^{-N_0}.$$
	Notice that by Lemma \ref{height}, 
	\begin{equation}\label{casdacc}
		\frac{\|T\|( \bC_{5\sqrt{m}})}{\omega_m (5\sqrt{m})^m}-Q=\bE(T, \bC_{5\sqrt{m}},\pi_0),
	\end{equation}
	so that, recalling also items $\rm (i)$ and $\rm (iii)$ of Proposition \ref{prop:tilting}, 
	\begin{align*}
		\bE(T,\bC_L,\pi_L)&\le C_0\frac{1}{r_L^m}\bE(T,\bC_{5\sqrt{m}},\pi_L)\le C(M_0,N_0)\big(\bE(T,\bC_{5\sqrt{m}},\pi_0)+|\vec \pi_L-\vec\pi_0|^2\|T\|(\bC_{5\sqrt{m}})\big)\\
		&\le C(M_0,N_0)\big(\bE(T,\bC_{5\sqrt{m}},\pi_0)+\bmo (\bE(T, \bC_{5\sqrt{m}},\pi_0)+Q)\big).
	\end{align*}
	Hence, by Assumption \ref{ipotesi},
	$$
	\bE(T,\bC_L,\pi_L)\le C(M_0,N_0) \big(\bE(T, \bB_{{6\sqrt{m}}},\pi_0)+\bmo\big)\le C(M_0,N_0)\bmo,
	$$
	so that $L\notin \sW_e^j$ if $C_e\ge C(M_0,N_0)$. Also, take $z\in \supp(T)\cap \bC_L$, then, by Lemma \ref{height} with items $\rm(i)$  and $\rm(iii)$ of Proposition \ref{prop:tilting}, 
	\begin{align*}
		|\p_{\pi_L^\perp}(z-p_L)|&\le|\p_{\pi_0^\perp}(z)|+|\p_{\pi_0^\perp}(p_L)|+|	\p_{\pi_L^\perp}-	\p_{\pi_0^\perp}|(|z|+|p_L|)\\
		&\le C_0\big(\bmo^{\sfrac{1}{2}}+|\vec\pi_L-\vec\pi_0|\big)\le C_0\bmo^{\sfrac{1}{2}},
	\end{align*}
	which implies that
	$$
	\bs(T,\bC_L,p_L+\pi_L)\le C_0\bmo^{\sfrac{1}{2}},
	$$
	further meaning that $L\notin \sW_h^j$ if $C_h\ge C(N_0)$.
	
	Equations \eqref{eq0}, \eqref{eq1}, and, in particular, equation \eqref{eq2} for $L\in \sS$,  follow directly from the definition of the stopping conditions. It remains to prove  \eqref{eq2} for $L\in\sW$.
	Take $L\in \sW$ and let $J$ be its father, notice that  $J\in \sS$. Take $z\in\supp(T)\cap \bC_L$, by item $\rm(iv)$ of Proposition \ref{prop:tilting}, $z\in \bC_J$. Also, by item $\rm (ii)$ of Proposition \ref{prop:tilting}, recalling \eqref{lcdanco},
	\begin{align*}
		|\p_{\pi_L^\perp}(z-p_L)|&\le |\p_{\pi_L^\perp}(p_L-p_J)|+|\p_{\pi_L^\perp}-\p_{\pi_J^\perp}||z-p_J|+|\p_{\pi_J^\perp}(z-p_J)|\\&
		\le C_0\bmo^{\sfrac{1}{2}}\ell(J)^2+C_0\bmo^{\sfrac{1}{2}}\ell(J)128r_J+C_h\bmo^{\sfrac{1}{2}}\ell(J)^{1+\beta_2}\\&\le C(M_0,C_h)\bmo^{\sfrac{1}{2}}\ell(L)^{1+\beta_2},
	\end{align*}
	which implies  the second part of \eqref{eq2}.
	
	Now we show \eqref{eq4}. Fix $ L\in \sW\cup \sS$, by what we have just proved (i.e.\ the second part of \eqref{eq2}), 
	$\bs(T,\bC_L,p_L+\pi_L)\le C_1\bmo^{\sfrac{1}{2}}\ell(L)^{1+\beta_2}$. Hence, for every $z\in \supp(T)\cap \bC_L$,
	$$
	|z-p_L|^2\le (64r_L)^2+(C_1\bmo^{\sfrac{1}{2}}\ell(L)^{1+\beta_2})^2\le (128r_L)^2,
	$$
	provided $\eps_0$ is sufficiently small (notice that  it is crucial that $C_1$ does not depend on $\eps_0$).
	
	We now show \eqref{eq3}. Take $L\in\sW\cup\sS$. Then, $\B_{128 r_L}(p_L)\subseteq\B_{1}(p_L)\subseteq\bC_{5\sqrt{m}}$, 
	as $128r_L\le 1$, by \eqref{dcasc}. By \eqref{eq4}, the monotonicity formula for the mass and \eqref{casdacc},
	\begin{align*}
		\|T\|(\bC_L)\le \|T\|(\B_{128 r_L}(p_L))
		\le (128r_L)^m \|T\|(B_{1}(p_L))
		\le (128r_L)^m \|T\|(\bC_{5\sqrt{m}})\le (128r_L)^m C_0,
	\end{align*}
	so that the claim follows.
	
	Finally, we prove the first part of \eqref{eq2}. Take again $L\in\sW$ and let $J$ be the father of $L$, notice that $J\in \sS$. Then, using items $\rm(iv)$ and $\rm(ii)$ of Proposition \ref{prop:tilting} as well as \eqref{eq3},
	\begin{align*}
		\bE(T,\bC_L,\pi_L)&\le 4 \bE(T,\bC_L,\pi_J)+C_0\ell(L)^{-m}{\|T\|(\bC_L)}|\vec \pi_L-\vec\pi_J|^2\le C_0 \bE(T,\bC_J,\pi_J)+C_1\bmo\ell(L)^2\\&\le C_0 C_e\bmo\ell(L)^{2-2\delta_2}+C_1\bmo\ell(L)^2\le C_1\bmo\ell(L)^{2-2\delta_2},
	\end{align*}
	which is the first part of \eqref{eq2}.
\end{proof}

\begin{proof}[Proof of Corollary \ref{c:cover}]
	Recalling Lemma \ref{height},  the first two assertions of item $\rm(i)$ follow, provided that $\eps_0$ is small enough. The third assertion is proved exactly as \cite[Corollary 2.2]{DS2}. 
	
	We now 
	show  item $\rm(ii)$. By Lemma \ref{height}, for every $z\in\supp(T)\cap \bC_{5\sqrt{m}}$, there exists $z'\in\cM$ with $|z-z'|\le C_0\bmo^{\sfrac{1}{2}}$. Take then $x\in L\in\sS\cup\sW$ and set $p\defeq\Phii(x)$, fix $p'\in\supp(T)\cap\p^{-1}(p)$. By what just remarked, if $\eps_0$ is sufficiently small (depending upon a geometric quantity and on $N_0$), $|p-p'|\le C_0\bmo^{\sfrac{1}{2}}\le \sfrac{r_M}{2}$, for any $M\in\sS^{N_0}$.
	Now we show that indeed
	\begin{equation}\label{csaacasas}
		|p-p'|\le r_L.
	\end{equation}
	If \eqref{csaacasas} were false, we could take the ancestor $H$ of $L$ with largest sidelength among the ones for which $|p-p'|\ge r_H$. Therefore, if $J$ is the father of $H$, $J\in\sS$. We notice that $|p-p'|<r_J$, in particular, $|p'-p_J|\le |p'-p|+|p-p_J|\le r_J+2\sqrt{m}\ell(J)\le 2r_J$, so that $p'\in\bC_J$ and $|\p_{\pi_0}(p')-x_J|\le 2 r_J$.
	By item $\rm(v)$ of proposition \ref{prop:tilting},
	there exists $p''\in\cM$ with
	\begin{align*}
		|p-p'|\le |p'-p''|&\le C_0|\p_{\pi_J^\perp}(p'-p_J)|+\bmo^{\sfrac{1}{2}}|\p_{\pi_0}(p')-x_J|^2
		\\&\le C_0 C_h\bmo^{\sfrac{1}{2}} \ell(J)^{1+\beta_2}+4M_0^2\bmo^{\sfrac{1}{2}}\ell(J)^2,
	\end{align*}
	so that, taking $\eps_0$ small enough (depending upon all other parameters), we reach $|p-p'|\le \sfrac{r_H}{2}$, which is a contradiction. Now we can run again the very same computation, with $L$ in place of $J$, relying on \eqref{csaacasas} and \eqref{eq2}, to conclude that
	\begin{align*}
		|p-p'|\le  C_1\bmo^{\sfrac{1}{2}} \ell(L)^{1+\beta_2}+4M_0^2\bmo^{\sfrac{1}{2}}\ell(L)^2\le C_1 \bmo^{\sfrac{1}{2}}\ell(L)^{1+\beta_2}.
	\end{align*}

	We show now item $\rm(iii)$. Fix  $x\in\bGam$ and set $p\defeq \Phii(x)$. As $x\in\bGam$, we can find a infinite sequence  $(L_j)_{j\ge N_0}\in\sS$, with $L_j\in\sS^j$ for every $j\ge N_0$ and $x=\bigcap_{j\ge N_0}L_j$. The conclusion then follows by item $\rm (ii)$.
\end{proof}
\subsection{Proof of the results of Section \ref{sectnorm}}
\begin{lem}[Projections on tilted cylinders]\label{giraerigira}
	Assume the hypotheses of Proposition \ref{p:whitney} and  that  $\eps_0$ is sufficiently small (depending upon all other parameters). Let $H,L\in\sW\cap \sS$ be such that $H\subseteq L$.
	Then, for every $p\in \Phii(L)$,
	$$(\p_{p+\pi_H})_\sharp (T\mres \bC_{32r_L}(p,\pi_H))=Q\a{B_{32 r_L}(p,\pi_H)}.$$
\end{lem}
\begin{proof}
	As $|p-p_L|\le 2\sqrt{m}\ell(L)$, by items $\rm(i)$ and $\rm(v)$ of Proposition \ref{prop:tilting}, $$\supp(T)\cap \bC_{32r_L}(p,\pi_H)\subseteq\supp(T)\cap \bC_{36r_L}(p_L,\pi_H)\subseteq \supp(T)\cap\bC_L\subseteq \bC_{5\sqrt{m}},$$ so that
	$$(\p_{p+\pi_H})_\sharp(T\mres \bC_{32r_L}(p,\pi_H))=k\a{B_{32r_L}(p,\pi_H)}\quad\text{for some integer $k$}.$$
	We conclude if we show that $k=Q$. Now, if $M\in\sS\cup\sW$ is an ancestor of $L$, $|p_M-p|\le 4\sqrt{m}\ell(M)$, so that 
	$\bC_{32r_L}(p,\pi_H)\subseteq \bC_{36r_M}(p_M,\pi_H)$.
	The argument now is exactly the same as the one for the proof of \cite[Proposition 4.2]{DS2}.
\end{proof}

\begin{proof}[Proof of Theorem \ref{t:approx}]
	Take $L\in\sW$, notice that trivially $\supp(T)\cap \bC_{32 r_L}(p_L,\pi_L)\subseteq\bC_L$
	and that, by  Lemma \ref{giraerigira},	provided that $\eps_0$ is sufficiently small, 
	$$(\p_{p_L+\pi_L})_\sharp (T\mres \bC_{32r_L}(p_L,\pi_L))=Q\a{B_{32 r_L}(p_L,\pi_L)}$$
	and also, by \eqref{eq2}, \begin{equation}\label{sacsad}
		E_L\defeq\bE(T,\bC_{32r_L}(p_L,\pi_L),\pi_L)\le 2^{m}\bE(T,\bC_L,\pi_L)\le C_1\bmo\ell(L)^{2-2\delta_2}.
	\end{equation}
	Now we follow the proof of \cite[Section 6.2]{DS2}, we give the details about the parts which are different.
	If $\eps_0$ is sufficiently small we can apply \cite[Theorem 2.4]{DS1} and obtain a $C_0E_L^{\gamma_1}$-Lipschitz map $$f_L:B_{8 r_L}(p_L,\pi_L)\to\Iq (\pi_L^\perp),$$
	and we associate to it the set $K_L\subseteq B_{8 r_L}(p_L,\pi_L)$ which corresponds to the set $K$ of  \cite[Theorem 2.4]{DS1}, in particular, $T$ coincides with the graph of $f_L$ on $K_L\times\pi_L^\perp$.
	We call $$\mathscr{D}(L)\defeq(\supp(T)\cup\gr(f_L))\cap \big[(B_{8r_L}(p_L,\pi_L)\setminus K_L)\times\pi_L^\perp\big].$$
	By the estimates of \cite[Theorem 2.4]{DS1} and the choice of $\beta_2,\gamma_2$ in \eqref{freavfad},
	\begin{equation}\label{cdssdc}
		\mathcal{H}^m(\mathscr{D}(L))+\|T\| (\mathscr{D}(L))\le C_0 E_L^{1+\gamma_1}\ell(L)^m\le C_1\bmo^{1+\gamma_2}\ell(L)^{m+2+\gamma_2}.	
	\end{equation}
	Also, by \eqref{eq2}, $\bs(T,\bC_L,p_L+\pi_L)\le C_1\bmo^{\sfrac{1}{2}}\ell(L)^{1+\beta_2}$. By the construction of  \cite[Theorem 2.4]{DS1}, we obtain
	$$
	\|f_L\|_{C^0(B_{8 r_L}(p_L,\pi_L))}\le C_1\bmo^{\sfrac{1}{2}}\ell(L)^{1+\beta_2},
	$$
	in particular, $\gr(f_L)\subseteq \bU$, if $\eps_0$ is smaller than a geometric quantity.
	
	Now we take $\cL$, the Whitney region associated to $L$, and set $\cL'\defeq\Phii(J)$, where $J$ is the cube concentric to $L$ with $\ell(J)=\sfrac{9}{8}\ell(L)$.

	By Assumption \ref{ipotesi}, $\cM$, in $\B_{6\sqrt{m}}$, is given by the graph of a map $$\bPsi_{p_L}:B_{7\sqrt{m}}(p_L,\pi_L)\to \pi_L^\perp,$$ satisfying $\|\bPsi_{p_L}\|_{C^2}\le \bmo^{\sfrac{1}{2}}$, recall that $\Phii_{p_L}(x)= (x,\bPsi_{p_L}(x))$ for $x\in B_{3\sqrt{m}\ell(L)}(p_L,\pi_L)$. In particular, if $\eps_0$ is small enough, we can apply \cite[Theorem 5.1]{DSsns}  with $f=f_L$, $r=8 r_L$ and $s=3\sqrt{m}\ell(L)$ to obtain maps $F_L:\Phii_{p_L}(B_s)\to \Iq (\bU)$ and $N_L:\Phii_{p_L}(B_s)\to \Iq (\RR^{m+n})$, where we set for brevity $B_s\defeq B_s(p_L,\pi_L)$ and similarly $B_r\defeq B_r(p_L,\pi_L )$.
	Then,  taking $\eps_0\le C_0$, we have $\cL'\subseteq \Phii_{p_L}(B_s)$, as for every $z'\in\cL'$, by \eqref{cdscas},
	\begin{equation}\label{cdsacasdc}
		\begin{split}
			|\p_{\pi_L}(z'-p_L)|\le|z'-p_L|\le \sqrt{1+C_0\bmo}\sqrt{m}\sfrac{9}{8}\ell(L)\le2\sqrt{m}\ell(L)\le s 
		\end{split}
	\end{equation}
	so that we can restrict the maps to obtain $F_L:\cL'\to \Iq (\bU)$ and $N_L:\cL'\to \Iq (\RR^{m+n})$. From now on, $F_L$ and $N_L$ denote these restrictions.
	
	By \cite[Theorem 5.1]{DSsns},  $\bG_{f_L}\mres (\p^{-1}(\cL'))=\bT_{F_L}\mres(\p^{-1}(\cL'))$ and, for every $p\in \cL'$,
	$N_L(p)=\sum_i\a{F_i(p)-p}$ and $N_L(p)\perp T_p\cM$. 
	By \cite[(5.3)]{DSsns},  for every $x\in B_{s}$,
	\begin{align*}
		|N_L(\Phii_{p_L}(x))|&\le C_0\cG(f_L(x),Q\a{\bPsi_{p_L}(x)})\le C_0\big(\cG(f_L(x),Q\a{0})+\cG(Q\a{0},Q\a{\bPsi_{p_L}(x)})\big)\\&\le
		C_1\bmo^{\sfrac{1}{2}}\ell(L)^{1+\beta_2}+C_0\bmo^{\sfrac{1}{2}}\ell(L)^2\le C_1\bmo^{\sfrac{1}{2}}\ell(L)^{1+\beta_2},
	\end{align*}
	so that,  by \cite[(5.2)]{DSsns},  we also have 
	\begin{align*}
		\Lip(N_L)&\le C_1\big(\|D^2\bPsi_{p_L}\|_{C^0(B_s)}\|N_L\|_{C^0(\Phii(B_s))}+\|D\bPsi_{p_L}\|_{C^0(B_s)}+\Lip(f_L)\big)\\
		&\le C_1\big(\bmo^{\sfrac{1}{2}} C_1\bmo^{\sfrac{1}{2}}\ell(L)^{1+\beta_2}+C_0\bmo^{\sfrac{1}{2}}\ell(L)+(C_1\bmo\ell(L)^{2-2\delta_2})^{\gamma_1}\big)\\
		&\le C_1\bmo^{\gamma_2}\ell(L)^{\gamma_2},
	\end{align*}
	where we also used \eqref{sacsad} and the fact that $D\bPsi_{p_L}(0)=0$.
	To sum up, we have proved the bounds in \eqref{e:Lip_regional} with $N_L$ in place of $N$.

	Now we can follow exactly the construction in \cite[Section 6.2]{DS2} and this shows the existence of the $\cM$-normal approximation  $(\cK, F)$ satisfying \eqref{e:Lip_regional} and \eqref{e:err_regional} (notice that the step: ``third extension and conclusion'' of \cite{DS2} is not needed in our case). We recall the definition of the set $\cK$, as it will be used in the sequel. For every $L\in\sW$, consider the set $\sW(L)\defeq\{M\in\sW:M\cap L\ne \emptyset\}$, then 
	\begin{equation*}
		\cK\defeq \cM\setminus \bigg(\bigcup_{L\in\sW}\bigg(\cL'\cap\bigcup_{M\in\sW(L)}\p(\mathscr{D}(M))\bigg)\bigg).
	\end{equation*}
	Recall that  $\p$ is $C_0$-Lipschitz and that for every $L\in\sW$ and $M\in\sW(L)$, $\sfrac{1}{2}\ell(L)\le \ell(M)\le 2\ell(L)$ (the cardinality of $\sW(L)$ is bounded by a geometric constant). Hence, by \eqref{cdssdc},
	\begin{equation}\label{seipuntocinque}
		\mathcal{H}^m(\cL\setminus \cK)\le 	\mathcal{H}^m(\cL'\setminus \cK)\le \sum_{M\in \sW(L)}\sum_{H\in\sW(M)}\p(\mathscr{D}(H))\le C_1\bmo^{1+\gamma_2}\ell(L)^{m+2+\gamma_2}.
	\end{equation}
	 We do not recall here the construction of the normal approximation.

	Now we show	\eqref{e:Dir_regional}, we have to take care only of the first conclusion, as the second one follows from \eqref{e:Lip_regional}. Let $\vec{\cM}(q)$ denote the unit $m$-vector orienting $T_q\cM$, hence, for every $q\in \cM\cap \bC_{8r_L}(p_L,\pi_L)$, $|\vec \cM(q)-\vec\pi_L|\le C_0\bmo^{\sfrac{1}{2}}r_L$.
	Now take $z\in \supp(T)\cap\p^{-1}(\cL)$, say $\p(z)=z'\in\cL$, then, recalling \eqref{cdsacasdc},
	\begin{align*}
		|\p_{\pi_L}(z-p_L)|&\le |\p_{\pi_L}-\p_{T_{z'}\cM}||z-z'|+|\p_{T_{z'}\cM}(z-z')|+|\p_{\pi_L}(z'-p_L)|\\
		&\le C_0\bmo^{\sfrac{1}{2}}r_L+2\sqrt{m}\ell(L)\le 32 r_L,
	\end{align*} if $\eps_0$ is small enough, so that $\supp(T)\cap\p^{-1}(\cL)\subseteq\supp(T)\cap \bC_{32r_L}(p_L,\pi_L)\subseteq\bC_{L}$.
	Therefore, by the considerations above together with \eqref{e:err_regional} and \eqref{eq3},
	\begin{align*}
		&	\int_{\p^{-1}(\cL)}|\vec\bT_F(x)-\vec{\cM}(\p(x))|^2\dd\|\bT_F\|(x)\\&\quad\le\	\int_{\p^{-1}(\cL)}|\vec T(x)-\vec{\cM}(\p(x))|^2\dd\|T\|(x)+ C_1\bmo^{1+\gamma_2}\ell(L)^{m+2+\gamma_2}\\&\quad\le \int_{\bC_L}|\vec T(x)-\vec\pi_L|^2\dd\| T\|(x)+ C_1\bmo^{1+\gamma_2}\ell(L)^{m+2+\gamma_2}+C_1\bmo \ell(L)^{m+2}\\&\quad\le C_1 \big(r_L^m\bmo \ell(L)^{2-2\delta_2}+\bmo^{1+\gamma_2}\ell(L)^{m+2+\gamma_2}+\bmo \ell(L)^{m+2}\big)\le C_1\bmo \ell(L)^{m+2-2\delta_2}
	\end{align*}
	where we used \eqref{eq2} in the next to last inequality.  Now the conclusion follows from \cite[Proposition 3.4]{DSsns}, exactly as in \cite[Section 6.3]{DS2}.
%
\end{proof}

\subsection{Proof of the results of Section \ref{sectsep}}
\begin{proof}[Proof of Proposition \ref{p:separ}]
	Fix $L\in\sW_h$ and let $J\in\sS$ denote the ancestor of $L$ with $\ell(J)=2^6\ell(L)$ (recall \eqref{e:prima_parte}), so that, by \eqref{eq1}, $$E\defeq\bE(T,\bC_J,\pi_J)\le C_e\bmo\ell(J)^{2-2\delta_2}.$$
	By Lemma \ref{giraerigira} and Lemma \ref{heighapp} (provided $\eps_0$ is sufficiently small), there exist $y_1,\dots,y_Q\in \pi_J^\perp$ such that $$\supp(T)\cap \bC_{4 r_J}(p_J,\pi_J)\subseteq\bigcup_{i=1}^Q( \pi_J\times B_{ C_0 r_J E^{\sfrac{1}{2}}}(y_i,\pi_J^\perp)).$$
	We can assume (up to  enlarging $C_0$) that $\supp(T)\cap \bC_{4r_J}(p_J,\pi_J)\subseteq\bigcup_{i=1}^q\bS_i$, where $1\le q\le Q$ and $\bS_i=\pi_J\times A_i$, for some $A_i\subseteq\pi_J^\perp$  with $$\diam(A_i)\le C_0 r_JE^{\sfrac{1}{2}}\le C_0r_L(C_e\bmo\ell(L)^{2-2\delta_2})^{\sfrac{1}{2}}$$  and such that $(A_i)_{i=1,\dots,q}$ are open and pairwise disjoint. By the smallness of $E$, we can moreover assume that $\p_{\pi_J}(T\mres (\bS_i\cap\bC_{4r_J}(p_J,\pi_J)))=Q_i\a{B(p_J,\pi_J)}$, for integers $ Q_i\ge 1$, with $\sum_{i=1}^qQ_i=Q$.
	
	Now, as $L\in\sW_h$, there exists $\hat z\in\bC_L$ with $|\p_{\pi_L^\perp}(\hat z-p_L)|>C_h\bmo^{\sfrac{1}{2}}\ell(L)^{1+\beta_2}$. Notice that $\hat z\in\B_{4r_J}(p_J)\subseteq \bC_{4r_J}(p_J,\pi_J)$. Indeed, by \eqref{eq4} and \eqref{cdscas},
	\begin{align*}
		|\hat z-p_J|\le |\hat z-p_L|+|p_L-p_J|\le 128 r_L+\sqrt{1+C_0\bmo} \sqrt{m}\ell(J)\le 4 r_J.
	\end{align*}
	Hence, up to reordering $\bS_1,\dots, \bS_1$, say $\hat z\in \bS_1$. 
	We compute, by item $\rm (ii)$ of Proposition \ref{prop:tilting} and \eqref{eq4},  
	\begin{align*}
		|\p_{\pi_J^\perp}(\hat z-p_J)|&\ge |\p_{\pi_L^\perp}(\hat z-p_L)|-  |\p_{\pi_J^\perp}-\p_{\pi_L^\perp}||\hat z-p_L|-|\p_{\pi_J^\perp}(p_J-p_L)|\\&\ge C_h\bmo^{\sfrac{1}{2}}\ell(L)^{1+\beta_2}-C_0\bmo^{\sfrac{1}{2}} \ell(J)128 r_L-C_0 \bmo^{\sfrac{1}{2}}\ell(J)^2\\&\ge
		\bmo^{\sfrac{1}{2}}\big(C_h\ell(L)^{1+\beta_2}-C_0 M_0\ell(L)^2- C_0\ell(L)^2\big)
		\\&\ge \bmo^{\sfrac{1}{2}}\ell(L)^{1+\beta_2}\big(C_h -C_0M_0-C_0\big),
	\end{align*}
	so that, for every $z\in\bS_1$,
	\begin{align*}
		|\p_{\pi_J^\perp}(z-p_J)|&\ge 		|\p_{\pi_J^\perp}(\hat z-p_J)|-|\p_{\pi_J^\perp}(z-\hat z)|\\&\ge \bmo^{\sfrac{1}{2}}\ell(L)^{1+\beta_2}\big(C_h -C_0M_0-C_0\big)- C_0 r_L\big(C_e\bmo\ell(L)^{2-2\delta_2}\big)^{\sfrac{1}{2}}
		\\&\ge\bmo^{\sfrac{1}{2}}\ell(L)^{1+\beta_2}\big(C_h -C_0M_0-C_0\big)-C_0M_0\ell(L)\bmo^{\sfrac{1}{2}}  C_e^{\sfrac{1}{2}}\ell(L)^{1-\delta_2}
		\\&\ge \bmo^{\sfrac{1}{2}}\ell(L)^{1+\beta_2}\big(C_h-C_0M_0-C_0-C_0M_0C_e^{\sfrac{1}{2}} \big).
	\end{align*}
	Therefore, if $C_h\ge C(M_0,C_e)$, 
	\begin{equation}\label{aaaa}
		|\p_{\pi_J^\perp}(z-p_J)|\ge 	\sfrac{15}{16}C_h\bmo^{\sfrac{1}{2}}\ell(L)^{1+\beta_2}\quad\text{for every }z\in \bS_1.
	\end{equation}

	Now take  $H=L$ or $H\in\sS\cup \sW$ with $\ell(H)=\sfrac{1}{2}\ell(L)$ and $L\cap H\ne \emptyset$. Let $\pi$ be an oriented $m$-plane with $|\vec\pi-\vec\pi_J|\le C_0\bmo^{\sfrac{1}{2}}\ell(J)$ (both $\pi_H$ and $\pi_J$ are suitable). Then,
	\begin{equation}\notag
		\supp(T)\cap \bC_{32r_H}(p_H,\pi)\subseteq \supp(T)\cap\bC_{4r_J}(p_J,\pi_J)\subseteq\bC_{5\sqrt m}.
	\end{equation}
	Indeed, the second inclusion is due to item $\rm (i)$ of Proposition \ref{prop:tilting}. For the first inclusion, take $z\in 	\supp(T)\cap \bC_{32r_H}(p_H,\pi)$,
	\begin{align*}
		|\p_{\pi_J}(z-p_J)|&\le 	|\p_{\pi_J}-\p_{\pi}||z-p_J|+|\p_{\pi}(z-p_H)|+|\p_{\pi}(p_H-p_J)|\\&
		\le C_0\bmo^{\sfrac{1}{2}}\ell(J)+32r_H+4\sqrt{m}\ell(J)\le 4r_J,
	\end{align*}
	provided $\eps_0$ is small enough (recall $M_0\ge 5$).
	Hence, arguing as in the proof of \cite[Proposition 4.2]{DS2}, we see that 
	\begin{equation}\label{cdscdsaa}
		(\p_{p_H+\pi_H})_\sharp\big(T\mres (\bS_1\cap \bC_{32r_H}(p_H,\pi_H))\big)=Q_1\a{B_{32 r_H}(p_H,\pi_H)}.
	\end{equation}
	
	Take any $\bar z\in \supp(T)\cap \bS_1\cap \bC_{32r_H}(p_H,\pi_H)$.
	We compute, using \eqref{aaaa} and item $\rm (ii)$ of Proposition \ref{prop:tilting} (notice that $|\bar z-p_H|\le |\p_{\pi_H^\perp}(\bar z-p_H)|+32r_H\le |\p_{\pi_H^\perp}(\bar z-p_H)|+32r_L$)
	\begin{align*}
		|\p_{\pi_H^\perp}(\bar z-p_H)|&\ge|\p_{\pi_J^\perp}(\bar z-p_J)|-|\p_{\pi_J^\perp}(p_J-p_H)|-|\p_{\pi_H^\perp}-\p_{\pi_J^\perp}||\bar z-p_H|\\&
		\ge \sfrac{15}{16}C_h\bmo^{\sfrac{1}{2}}\ell(L)^{1+\beta_2}- C_0\bmo^{\sfrac{1}{2}}\ell(J)^2\\&\quad-C_0\bmo^{\sfrac{1}{2}}\ell(J)|\p_{\pi_H^\perp}(\bar z-p_H)|-C_0\bmo^{\sfrac{1}{2}}\ell(J)32r_L,
	\end{align*}
	which means 
	\begin{align*}
		|\p_{\pi_H^\perp}(\bar z-p_H)|\big(1+C_0\bmo^{\sfrac{1}{2}}\big)\ge \bmo^{\sfrac{1}{2}}\ell(L)^{1+\beta_2}(\sfrac{15}{16}C_h-C_0-C_0M_0).
	\end{align*}
	Therefore, if $\eps_0$ is smaller than a geometric constant and $C_h\ge C(M_0)$,  we have
	\begin{equation}\label{caaasdcas1}
		|\p_{\pi_H^\perp}(\bar z-p_H)|\ge \sfrac{5}{8}C_h \bmo^{\sfrac{1}{2}}\ell(L)^{1+\beta_2}\quad\text{for every $\bar z\in \supp(T)\cap \bS_1\cap \bC_{32r_H}(p_H,\pi_H)$}.
	\end{equation}
	\medskip\\\textbf{Step 1: we show $\rm(S2)$.} We argue  by contradiction. Take $H\in\sW_n$ with $\ell(H)\le \sfrac{1}{2}\ell(L)$ but  $L\cap H\ne \emptyset$  (from the stopping criteria,  $\ell(H)=\sfrac{1}{2}\ell(L)$). Then \eqref{cdscdsaa} and \eqref{caaasdcas1} imply that
	\begin{equation}\notag
		\bs(T,\bC_H,p_H+\pi_H)\ge \sfrac{5}{4} C_h\bmo^{\sfrac{1}{2}}\ell(H)^{1+\beta_2}>C_h\bmo^{\sfrac{1}{2}}\ell(H)^{1+\beta_2},
	\end{equation}
	which is a contradiction	since the $\rm(HT)$ condition has priority over the $\rm(NN)$ condition and $H\in \sW_n$. 
	\medskip\\\textbf{Step 2: we show $\rm(S3)$.} 
	  Recall that \eqref{cdscdsaa} and \eqref{caaasdcas1} hold with $L$ in place of $H$.
	This means that for every $x\in B_{32 r_L}(p_L,\pi_L) $, there exists $z_x\in\supp(T)\cap \bC_{32r_L}(p_L,\pi_L)$ with $\p_{p_L+\pi_L}=x$ and $|\p_{\pi_L^\perp}(z_x-p_L)|=|z_x-x|\ge \sfrac{5}{8}C_h\bmo^{\sfrac{1}{2}}\ell(L)^{1+\beta_2}$.
	
	Now we exploit the construction of the proof of Theorem \ref{t:approx}, we keep the same notation, e.g.\ we use $f_L:B_{8r_L}(p_L,\pi_L)\to\Iq(\pi_L^\perp)$, $K_L\subseteq B_{8r_L}(p_L,\pi_L)$ and $N_L:\Phi(B_s)\to\Iq(\RR^{m+n})$, where $B_s=B_{3\sqrt{m}\ell(L)}(p_L,\pi_L)$. 
Notice that, if $\eps_0$ is smaller than a geometric constant,  then $\Omega\subseteq\Phii_{p_L}(B_s)$ (indeed, if $p\in\Omega$, then $|p-p_L|\le \sqrt{1+C_0\bmo}(\sqrt{m}+1)\ell(L)\le 3\sqrt{m}\ell(L)= s$). We are going to use this fact throughout.

 Recall that,  by \eqref{cdssdc},
 \begin{equation}\label{vrfeacda}
 \mathcal{H}^m\big(\Omega\setminus\Phii_{p_L}( K_L)\big)\le
 \mathcal{H}^m\big(\Phii_{p_L}(B_s\cap\p_{p_L+\pi_L}(\mathscr{D}(L)))\big)\le C_1 \bmo^{1+\gamma_2}\ell(L)^{m+2+\gamma_2}
 \end{equation}
 and notice that 
 \begin{equation}\label{cdscasa2}
     \mathcal{H}^m(\Omega\setminus \cK)\le C_1\bmo^{1+\gamma_2}\ell(L)^{m+2+\gamma_2},
 \end{equation}  as  every  cube of the Whitney decomposition intersecting $ B_{\sfrac{\ell(L)}{4}}(q)\subseteq B_{2\sqrt{m}\ell(L)}(x_L)$ has side-length at most $C_0\ell(L)$ and we can use  \eqref{seipuntocinque}.
 
 By what remarked above, $$|f_L(x)|\ge\sfrac{5}{8} C_h\bmo^{\sfrac{1}{2}}\ell(L)^{1+\beta_2} \quad\text{for every $x\in K_L$}.$$ 
	Using \cite[(5.3)]{DSsns}, we have that for every $x\in B_s\cap K_L$
	\begin{equation}\label{cdscasa}
		\begin{split}
			2\sqrt{Q}|N_L(\Phii_{p_L}(x))|&\ge \cG(f_L(x),Q\a{\bPsi_{p_L}(x)})\ge \cG(f_L(x),Q\a{0})-\cG(Q\a{0},Q\a{\bPsi_{p_L}(x)})\\&\ge
			\sfrac{5}{8}C_h\bmo^{\sfrac{1}{2}}\ell(L)^{1+\beta_2}-C_0\bmo^{\sfrac{1}{2}}\ell(L)^2
			\ge \bmo^{\sfrac{1}{2}}\ell(L)^{1+\beta_2}(\sfrac{5}{8}C_h-C_0\ell(L)^{1-\beta_2})
			\\&	\ge \sfrac{1}{2}C_h\bmo^{\sfrac{1}{2}}\ell(L)^{1+\beta_2},
		\end{split}
	\end{equation}
	provided that $C_h\ge C_0$. 
 Recall that by $\rm(A2)$ of Definition \ref{d:app}, and the definition of $\mathscr{D}(L)$ (with the proof of  $\rm(A2)$ of Definition \ref{d:app} mind),
 \begin{equation}\label{cdscasa1}
 N(p)=N_L(p)\quad\text{for every  }p\in\Phii_{p_L}(B_s)\cap \cK\setminus  \p(\mathscr{D}(L)).
 \end{equation}
Then, by \eqref{cdscasa} and \eqref{cdscasa1},
	\begin{equation}\label{cdscasa3}
		 2\sqrt{Q}	|N(p)|\ge \sfrac{1}{2}C_h\bmo^{\sfrac{1}{2}}\ell(L)^{1+\beta_2}\quad\text{for every $p\in \Omega\cap\cK\cap\Phii_{p_L}(K_L)\setminus  \p(\mathscr{D}(L))$}.
	\end{equation}

	
	Now,  $\mathcal{H}^m(\Omega)\ge C_0^{-1} \ell(L)^m$, hence, by \eqref{cdscasa3}, \eqref{vrfeacda},  \eqref{cdscasa2}, and \eqref{cdssdc},
	\begin{align*}
		2\sqrt{Q}\int_\Omega |N|^2&\ge  2\sqrt{Q}\int_{\Omega\cap\cK \cap\Phii_{p_L}(K_L)\setminus  \p(\mathscr{D}(L))}|N|^2\\&\ge\sfrac{1}{4} C_h^2\bmo\ell(L)^{2+2\beta_2} \mathcal{H}^m\big({\Omega\cap\cK \cap\Phii_{p_L}(K_L)\setminus  \p(\mathscr{D}(L))}\big) 
		\\&\ge\sfrac{1}{4} C_h^2\bmo\ell(L)^{2+2\beta_2} \big(C_0^{-1}\ell(L)^m-C_1\bmo^{1+\gamma_2}\ell(L)^{m+2+\gamma_2}\big)
		\\&\ge C_0^{-1} C_h^2\bmo\ell(L)^{m+2+2\beta_2},
	\end{align*}
	provided that $\eps_0$ is sufficiently small  (depending upon all other parameters).
\end{proof}

\begin{lem}[Unique continuation for $\D$-minimizers]\label{l:UC}
	For every $\bar\ell \in (0,1)$ and $\bar c>0$, there exists $\bar\gamma=\bar\gamma(m,n,Q,\bar\ell,\bar c)>0$ with the following property.
	If $w: \R^m\supseteq B_{2 r} \to \Iq(\R^k)$ is $\D$-minimizing,
	$\D(w, B_r)\ge \bar c$ and  $\D(w, B_{2r}) =1$,
	then
	\[
	\D (w, B_s (q)) \ge \bar\gamma \quad \text{and}\quad \int_{B_s(q)}|w|^2\ge \bar\gamma
	\]
	for every $B_s(q)\subseteq B_{2r}$ with $s \ge \bar\ell r$
\end{lem}
\begin{proof}
	The statement  is as in \cite[Lemma 7.1]{DS2}, up to the fact that we are also claiming $\int_{B_s(q)}|w|^2\ge \bar\gamma$. This is obtained with a contradiction argument as in the proof of \cite[Lemma 7.1]{DS2}, exploiting (UC) in the proof of \cite[Lemma 7.1]{DS2}.
\end{proof}

\begin{proof}[Proof of Proposition \ref{p:splitting}]
	Fix $L\in\sW_e$, let $H\in\sS$ be its father and $J\in\sS$ its ancestor with $\ell(J)=2^6\ell(L)$.
	Notice that by item $\rm(v)$ of Proposition \ref{prop:tilting},
	\begin{equation}\label{contain}
		\supp(T)\cap \bC_{32r_J}(p_H,\pi_H)\subseteq\supp(T)\cap\bC_{36r_J}(p_J,\pi_H)\subseteq\bC_J.
	\end{equation}
	Therefore, recalling item $\rm(ii)$ of Proposition \ref{prop:tilting} and \eqref{eq3}, as $J\in\sS$,
	\begin{align*}
		E&\defeq \bE(T,\bC_{32r_J}(p_H,\pi_H),\pi_H)\\&\le 2^2 \bE(T,\bC_{32r_J}(p_J,\pi_H),\pi_J)+C_0r_J^{-m}|\vec\pi_J-\vec\pi_H|^2\|T\|(\bC_J)
		\\&\le 2^{2+m}C_e\bmo\ell(L)^{2-2\delta_2}+ C(M_0)\bmo\ell(L)^2, 
	\end{align*}
	so that, if $N_0\ge C(M_0)$,
	\begin{equation}\label{exaa}
		E\le 2^{m+3}C_e\bmo\ell(L)^{2-2\delta_2}
	\end{equation}
	(in particular, we can make $E$ arbitrary small by taking $\eps_0$ small enough).
	
	For brevity, we set $\bC\defeq \bC_{32r_J}(p_H,\pi_H)$ and $B\defeq B_{8r_J}(p_H,\pi_H)$.
	By Lemma \ref{giraerigira},
	$$(\p_{p_H+\pi_H})_\sharp (T\mres \bC)=Q\a{B_{32 r_J}(p_H,\pi_H)}.$$ Then, building upon \cite[Theorem 2.4]{DS1}, arguing as at the beginning of the proof of Theorem \ref{t:approx} (provided $\eps_0$ is sufficiently small), we obtain the following objects. The map $f\defeq f_J:B\to \Iq(\pi_H^\perp)$ and the associated sets $K\defeq K_J\subseteq B$, $\mathscr{D}(J)\subseteq B\times\pi_H^\perp$, satisfying
	\begin{equation}\label{ldsani}
		\mathcal{H}^m(K)+\mathcal{H}^m(\mathscr{D}(J))+\|T\| (\mathscr{D}(J))\le C_1\bmo^{1+\gamma_2}\ell(L)^{m+2+\gamma_2}.
	\end{equation}
	By the first claim of item $\rm(v)$ of Proposition \ref{prop:tilting} together with \eqref{contain}, 
	\begin{equation}\label{rafdsam}
		\bs(T,\bC,p_H+\pi_H)\le C_1\bmo^{\sfrac{1}{2}}\ell(L)^{1+\beta_2},
	\end{equation} so that, recalling the construction of \cite[Theorem 2.4]{DS1}
	\begin{equation}\label{rafdsa}
		\|f\|_{C^0(B)}\le C_1\bmo^{\sfrac{1}{2}}\ell(L)^{1+\beta_2}.
	\end{equation}
	\medskip
	\\\textbf{Step 1: decay estimate for $f$.}
	Take now $V\subseteq B_{8r_J}(p_H,\pi_H)$ open. By the Taylor expansion of the mass (e.g.\ \cite[Corollary 3.3]{DSsns}) and \cite[(2.6)]{DS1} (we assume that $\eps_0$ is smaller than a geometric constant),
	\begin{align*}
		\big|	\|\bT_{f}\|(V\times\pi_H^\perp)-Q\mathcal{H}^m(V)-\sfrac{1}{2}\D(f,V)\big|&\le C_0 E^{2\gamma_1}\D(f,B)\\&\le C_0 E^{2\gamma_1}E(8r_J)^m\le C_0E^{1+\gamma_1}r_J^m,
	\end{align*}
	so that, recalling \eqref{ldsani} and the relations among $\gamma_1,\gamma_2$ and $\delta_2$,
	\begin{equation}\label{diriV}
		\begin{split}
			\big|	\|T\|(V\times\pi_H^\perp)-Q\mathcal{H}^m(V)-\sfrac{1}{2}\D(f,V)\big|& \le C_0E^{1+\gamma_1}r_J^m+C_1\bmo^{1+\gamma_2}\ell(L)^{m+2+\gamma_2}\\&\le C_1\bmo^{1+\gamma_2}\ell(L)^{m+2+\gamma_2}.
		\end{split}
	\end{equation}
	
	Now we set $$\rho\defeq 64r_L.$$
	On the one hand, by \eqref{diriV},
	\begin{equation}\label{vsda1}
		\D(f, B_{4\rho}(p_H,\pi_H))	\le \omega_m(4\rho)^m\bE(T,\bC_{4\rho}(p_H,\pi_H),\pi_H)+C_1\bmo^{1+\gamma_2}\ell(L)^{m+2+\gamma_2},
	\end{equation}
	and, on the other hand, again by \eqref{diriV},
	\begin{equation}\label{vsda2}
		\D(f, B_{2\rho}(p_H,\pi_H))	\ge \omega_m (2\rho)^m\bE(T,\bC_{2\rho}(p_H,\pi_H),\pi_H)-C_1\bmo^{1+\gamma_2}\ell(L)^{m+2+\gamma_2}.
	\end{equation}
	By noticing that $\bE(T,\bC_{4\rho}(p_H,\pi_H),\pi_H)\le 2^{3m}E$ and exploiting \eqref{exaa}, we deduce, from \eqref{vsda1},
	\begin{equation}\label{vsda1a}
		\D(f, B_{4\rho}(p_H,\pi_H))	\le C_0C_e\bmo\ell(L)^{m+2-2\delta_2}+C_1\bmo^{1+\gamma_2}\ell(L)^{m+2+\gamma_2}\le C_0C_e\bmo\ell(L)^{m+2-2\delta_2},
	\end{equation}
	provided that $\eps_0$ is small enough.
	Also, if $\eps_0$ is small enough, $$\supp(T)\cap \bC_{L}\subseteq \supp(T)\cap \bC_{2\rho}(p_H,\pi_H)\subseteq\bC_J.$$ Indeed, for the first inclusion, take $z\in \supp(T)\cap \bC_{L}$ and compute, using item $\rm(ii)$ of Proposition \ref{prop:tilting},
	\begin{align*}
		|\p_{\pi_H}(z-p_H)|&\le|p_H-p_L|+ |\p_{\pi_H}-\p_{\pi_L}||z-p_L|+|\p_{\pi_L}(z-p_L)|
		\\&\le 2\sqrt{m}\ell(H)+C_0\bmo^{\sfrac{1}{2}}\ell(L)+64 r_L,
	\end{align*}
	whereas, for the second inclusion, if $z\in \supp(T)\cap \bC_{2\rho}(p_H,\pi_H)$,  using item $\rm(ii)$ of Proposition \ref{prop:tilting},
	\begin{align*}
		|\p_{\pi_J}(z-p_J)|&\le|p_J-p_H|+ |\p_{\pi_J}-\p_{\pi_H}||z-p_H|+|\p_{\pi_H}(z-p_H)|
		\\&\le 2\sqrt{m}\ell(J)+C_0\bmo^{\sfrac{1}{2}}\ell(J)+128r_L.
	\end{align*}
	We can then compute, by \eqref{eq3} and item $\rm(ii)$ of Proposition \ref{prop:tilting},
	\begin{align*}
		C_e \bmo\ell(L)^{2-2\delta_2}&\le\bE(\bC_L,\pi_L)\le 2^m \bE(T,\bC_{2\rho}(p_H,\pi_H),\pi_L)\\&\le  2^{m+2} \bE(T,\bC_{2\rho}(p_H,\pi_H),\pi_H)+ C_0 \rho^{-m}\|T\|(\bC_J)|\pi_L-\pi_H|^2
		\\&\le  2^{m+2} \bE(T,\bC_{2\rho}(p_H,\pi_H),\pi_H)+C(M_0)\bmo\ell(L)^2.
	\end{align*}
	Therefore, if $N_0\ge C(M_0)$,
	\begin{equation}\notag
		C_e \bmo\ell(L)^{2-2\delta_2}\le 2^{m+3} \bE(T,\bC_{2\rho}(p_H,\pi_H),\pi_H).
	\end{equation}
	Hence, \eqref{vsda2} reads
	\begin{equation}\label{vsda2a}
		\begin{split}
			\D(f, B_{2\rho}(p_H,\pi_H))	&\ge \omega_m (2\rho)^m2^{-m-3} C_e\bmo\ell(L)^{2-2\delta_2}-C_1\bmo^{1+\gamma_2}\ell(L)^{m+2+\gamma_2}\\&\ge  \omega_m  C_e\bmo\ell(L)^{m+2-2\delta_2},
		\end{split}
	\end{equation}
	provided that $\eps_0$ is small enough.
	\medskip\\\textbf{Step 2: harmonic approximation.}
	Let $\bar\eta\in(0,1)$ to be fixed later.
	Then, provided $E$ is small enough (depending on $\bar \eta$), we can apply \cite[Theorem 2.6]{DS1} and obtain a $\D$-minimizing map $u:B\to \Iq(\pi_H^\perp)$ with  (the second inequality is due to  \eqref{exaa})
	\begin{equation}\label{vsda3}
		r_J^{-2}\int_B\cG(f,u)^2+\int_B \big| |D f|^2-|Du|^2\big|\le \bar\eta E r_J^m\le \bar \eta C_0 M_0^m C_e\bmo\ell(L)^{m+2-2\delta_2}.
	\end{equation}
	On the one hand, by \eqref{vsda1a} and \eqref{vsda3}, 
	\begin{equation}\label{vsda1b}
		\begin{split}
			\D(u ,B_{4\rho}(p_H,\pi_H))	&\le C_0C_e \bmo\ell(L)^{m+2-2\delta_2}+ \bar \eta C_0 M_0^m C_e\bmo\ell(L)^{m+2-2\delta_2}\\&\le C_0C_e \bmo\ell(L)^{m+2-2\delta_2},
		\end{split}
	\end{equation}
	and, on the other hand, by  \eqref{vsda2a} and \eqref{vsda3}, 
	\begin{equation}\label{vsda2b}
		\begin{split}
			\D(u, B_{2\rho}(p_H,\pi_H))&\ge \omega_m  C_e\bmo\ell(L)^{m+2-2\delta_2}-  \bar \eta C_0 M_0^m C_e\bmo\ell(L)^{m+2-2\delta_2}
			\\&\ge \sfrac{\omega_m}{2}  C_e\bmo\ell(L)^{m+2-2\delta_2},
		\end{split}
	\end{equation}
	where we are assuming that $\bar \eta$ is sufficiently small, depending upon $M_0$ and a geometric quantity (hence that $\eps_0$ is sufficiently small, depending upon all other parameters).
	
	Combining \eqref{vsda1b} and \eqref{vsda2b}, we obtain
	\begin{align*}
		\frac{\D(u, B_{2\rho}(p_H,\pi_H))}{\D(u, B_{4\rho}(p_H,\pi_H))}\ge \frac{\sfrac{\omega_m}{2}  C_e\bmo\ell(L)^{m+2-2\delta_2}}{C_0C_e \bmo\ell(L)^{m+2-2\delta_2}}\ge c_0,
	\end{align*}
	where $c_0$ is a geometric constant. 
	Set $\bar \ell\defeq \sfrac{\ell(L)}{32\rho}=(2^{11}\sqrt{m}M_0)^{-1}$.
	We now apply Lemma \ref{l:UC} to  $ \D(u, B_{4\rho}(p_H,\pi_H))	^{-\sfrac{1}{2}}u$, with $2\rho$ in place of $r$, and the choice above for $\bar \ell$ to obtain that, for a  constant $\bar\gamma$ that depends on a geometric quantity and on $M_0$ (in particular, $\bar\gamma$ is independent of $\bar\eta$ and $\eps_0$),
	\begin{alignat*}{3}
		\D (u, B_{\sfrac{\ell(L)}{16}} (q',\pi_H)) &\ge \bar\gamma  \D(u, B_{4\rho}(p_H,\pi_H))&&\ge \sfrac{\omega_m}{2}\bar \gamma C_e\bmo\ell(L)^{m+2-2\delta_2} ,\\
		\int_{B_{\sfrac{\ell(L)}{16}} (q',\pi_H)}|u|^2&\ge \bar\gamma\D(u, B_{4\rho}(p_H,\pi_H)&&\ge \sfrac{\omega_m}{2}\bar\gamma C_e\bmo\ell(L)^{m+2-2\delta_2},
	\end{alignat*}
	for every $B_{\sfrac{\ell(L)}{16}} (q',\pi_H)\subseteq B_{4\rho}(p_H,\pi_H)$, where we used \eqref{vsda2b}.
	
	Therefore, assuming that $\bar\eta$  is smaller than a constant that depends only upon $M_0,\bar \gamma$ and a geometric quantity  (i.e.\ that $\eps_0$ is sufficiently small, depending upon all other parameters), using also \eqref{vsda3}, 
	\begin{equation}\label{dsaacs1}
		\begin{split}
			\D (f, B_{\sfrac{\ell(L)}{16}} (q',\pi_H)) &\ge \sfrac{\omega_m}{2}\bar \gamma C_e\bmo\ell(L)^{m+2-2\delta_2} -  \bar \eta C_0 M_0^m C_e\bmo\ell(L)^{m+2-2\delta_2}\\&\ge\sfrac{\omega_m}{4}\bar \gamma C_e\bmo\ell(L)^{m+2-2\delta_2}
		\end{split}
	\end{equation}
	and also
	\begin{equation}\label{dsaacs2}
		\begin{split}
			\int_{B_{\sfrac{\ell(L)}{16}} (q',\pi_H)}|f|^2&\ge \sfrac{\omega_m}{2}\bar\gamma C_e\bmo\ell(L)^{m+2-2\delta_2}-  \bar \eta C_0 M_0^m C_e\bmo\ell(L)^{m+2-2\delta_2}\\&\ge\sfrac{\omega_m}{4}\bar \gamma C_e\bmo\ell(L)^{m+2-2\delta_2}.
		\end{split}
	\end{equation}	 
	\medskip\\\textbf{Step 3: estimate for the $\cM$-normal approximation.}
	We define $$q'\defeq \p_{p_H+\pi_H}(\Phii(q)),$$ notice that, by \eqref{cdscas} and simple considerations,
	\begin{equation}\label{ffdadc0}
		B_{2\ell(L)}(q',\pi_H)\subseteq B_{4\rho}(p_H,\pi_H)
	\end{equation}
	so that \eqref{dsaacs1} and \eqref{dsaacs2} hold for this choice of $q'$. 
	Also, if $\eps_0$ is small enough,
	\begin{equation}\label{ffdadc}
		B_{{\ell(L)}}(q',\pi_H)\supseteq\p_{p_H+\pi_H}(\Omega)\supseteq B_{\sfrac{\ell(L)}{8}}(q',\pi_H),
	\end{equation}
	by simple geometric considerations. Indeed, the first inclusion is trivial, and, for the second inclusion, take $z'\in B_{\sfrac{\ell(L)}{8}}(q',\pi_H)$. Let then $z$ be the unique point in $\cM$ with $(z-z')\perp\pi_H$, then, by the estimates on $\Phii_{p_H}$, $|z-\Phii(q)|\le 2\sqrt{m}\sfrac{\ell(L)}{8}$ and, recalling also item $\rm(iii)$ of Proposition \ref{prop:tilting},
	\begin{align*}
		|\p_{\pi_0}(z-\Phii(q))|\le |\p_{\pi_0}-\p_{\pi_H}||z-\Phii(q)|+ |\p_{\pi_H}(z-\Phii(q))|\le C_0\bmo^{\sfrac{1}{2}}\ell(L)+\sfrac{\ell(L)}{8},
	\end{align*}
 so that $z=\Phii(z'')$ with $z''\in\Omega$, provided that $\eps_0$ is small enough.
	
	Now notice that every Whitney region $\cL'$ that intersects $\Omega$ corresponds to $L'$ with $\ell(L')\le C_0\ell(L)$. This is due to the stopping condition (NN). We are going to use this fact several times in the rest of the paragraph.
 Summing \eqref{e:err_regional} over all the Whitney regions that intersect $\Omega $,
	\begin{equation}\label{asasa}
		\mathcal{H}^m(\Omega\setminus\cK)+\|\bT_F-T\|(\p^{-1}(\Omega))\le C_1\bmo^{1+\gamma_2}\ell(L)^{m+2+\gamma_2}.
	\end{equation}
	For the same reason, using also item $\rm(ii)$ of Corollary \ref{c:cover}, for every $z\in\Omega$  and $z'\in\supp(T)\cap\p^{-1}(z)$, $|z-z'|\le C_1\bmo^{\sfrac{1}{2}}\ell(L)^{1+\beta_2}$. In particular, $$|\p_{\pi_H}(z'-q')|\le |z'-\Phii(q)|\le |z'-z|+|z-\Phii(q)|\le C_1\bmo^{\sfrac{1}{2}}\ell(L)^{1+\beta_2}+2\sfrac{\ell(L)}{4}\le\ell(L),
	$$
	if $\eps_0$ is smaller than a geometric constant. Therefore,
	$\supp(T)\cap\p^{-1}(\Omega)\subseteq \bC_{8r_J}(p_H,\pi_H),$ so that by \eqref{ldsani},
	\begin{equation}\label{asasa1}
		\|\bG_f-T\|(\p^{-1}(\Omega))\le C_1\bmo^{1+\gamma_2}\ell(L)^{m+2+\gamma_2}.
	\end{equation}
	By the same argument as above, together with \eqref{e:Dir_regional}, yields
	\begin{equation}\label{klcsakc}
		\int_\Omega |DN|^2\le C_1\bmo\ell(L)^{m+2-2\delta_2}\quad\text{and}\quad	\|N\|_{C^0(\Omega )}\le C_1\bmo^{\sfrac{1}{2}}\ell(L)^{1+\beta_2}.
	\end{equation}
	and, with \eqref{e:Lip_regional},
	\begin{equation} \label{klcsakc1}
		|DN|\le C_1\bmo^{\gamma_2}\ell(L)^{\gamma_2}\quad\text{a.e.\ on }\Omega
	\end{equation}
	(notice that we are not claiming a bound on $\Lip(N|_\Omega)$).

	Now recall  \eqref{ffdadc0}, \eqref{ffdadc}, the estimates on  $\|\bPsi_{p_H}\|_{C^2}$, and \eqref{rafdsa}, so that we can take $\eps_0$ small enough and apply \cite[Theorem 5.1]{DSsns} to obtain (by restriction) maps $F':\Omega\to \Iq(\bU)$ and $N':\Omega\to \Iq(\RR^{m+n})$ (which is the normal part of $F'$) such that 
	$
	\bT_{F'}\mres \p^{-1}(\Omega)=\bG_f\mres\p^{-1}(\Omega).
	$ 
	\medskip\\\textbf{Step 4: proof of \eqref{e:split_1}.} The first inequality of \eqref{e:split_1} is the definition of $L\in\sW_e$, so we only have to show the second inequality.
	Now we use \cite[Proposition 3.4]{DSsns} (with manifold $p_H+\pi_H$, provided $\eps_0$ is small enough), or alternatively \cite[Theorem 3.2]{DSsns}, to get
	\begin{align*}
		\D (f, B_{\sfrac{\ell(L)}{16}} (q',\pi_H))&\le \int_{\bC_{\sfrac{\ell(L)}{16}}(q',\pi_H)} |\vec\bG_f-\vec\pi_H|^2\dd\|\bG_f\|+C_0\int_{B_{\sfrac{\ell(L)}{16}}}|Df|^4
		\\&\le  \int_{\bC_{\sfrac{\ell(L)}{16}}(q',\pi_H)} |\vec\bG_f-\vec\pi_H|^2\dd\|\bG_f\|+C_0E^{2\gamma_1}\D (f, B_{\sfrac{\ell(L)}{16}}(q',\pi_H)).
	\end{align*}
	Therefore, recalling \eqref{exaa}, if $\eps_0$ is small enough,
	\begin{equation}\label{kdsaa}
		\D (f, B_{\sfrac{\ell(L)}{16}} (q',\pi_H))\le2 \int_{\bC_{\sfrac{\ell(L)}{16}}(q',\pi_H)} |\vec\bG_f-\vec\pi_H|^2\dd\|\bG_f\|.
	\end{equation}
	
	We are going to use that 
	\begin{equation}\label{sassdfa}
		\supp(T)\cap \bC_{\sfrac{\ell(L)}{16}}(q',\pi_H)\subseteq \p^{-1}(\Omega).
	\end{equation} Indeed, take $z\in \supp(T)\cap \bC_{\sfrac{\ell(L)}{16}}(q',\pi_H)$, that is $z\in\supp(T)\cap \bC_{\sfrac{\ell(L)}{16}}(q,\pi_H)$, we have to show that $\p(z)\in\Omega$. Now, by \eqref{rafdsam} and the estimate on $\|\Psi_{p_H}\|_{C^2}$, we have that there exits $z'\in \cM\cap \bC_{\sfrac{\ell(L)}{16}}(q,\pi_H)$ with $|z-z'|\le C_1\bmo^{\sfrac{1}{2}}\ell(L)^{1+\beta_2}+C_0\bmo^{\sfrac{1}{2}}M_0^2\ell(L)^2$, where we also recalled \eqref{ffdadc0}. Hence, if $\eps_0$ is small enough (depending upon all the other parameters), $|z-z'|\le \sfrac{\ell(L)}{16}$, so that $\p(z)\in \bC_{\sfrac{\ell(L)}{8}}(q',\pi_H)$, hence $\p(z)\in\Omega$ by \eqref{ffdadc}.
	
	Now, by \eqref{asasa1} and \eqref{asasa} (with \eqref{sassdfa}), using also \eqref{eq3} and the estimates on $\|\bPsi_{p_H}\|_{C^1}$ (with  \eqref{ffdadc0} and \eqref{contain}), if  $\vec\cM$ denotes the unit $m$-vector orienting $T\cM$,
	\begin{align*}
		& \int_{\bC_{\sfrac{\ell(L)}{16}}(q',\pi_H)} |\vec\bG_f-\vec\pi_H|^2\dd\|\bG_f\|\le
		\int_{\bC_{\sfrac{\ell(L)}{16}}(q',\pi_H)} |\vec T-\vec\pi_H|^2\dd\|T\|+C_1\bmo^{1+\gamma_2}\ell(L)^{m+2+\gamma_2}
		\\&\quad\le 4\int_{\bC_{\sfrac{\ell(L)}{16}}(q',\pi_H)} |\vec T(p)-\vec\cM(\p(p))|^2\dd\|T\|(p)\\&\quad\quad+
		C_0 \bmo\ell(L)^2 \|T\|(\bC_{\sfrac{\ell(L)}{16}}(q',\pi_H))+
		C_1\bmo^{1+\gamma_2}\ell(L)^{m+2+\gamma_2}
		\\&\quad\le 4\int_{\bC_{\sfrac{\ell(L)}{16}}(q',\pi_H)} |\vec T(p)-\vec\cM(\p(p))|^2\dd\|T\|(p)+C(M_0) \bmo\ell(L)^{m+2}+
		C_1\bmo^{1+\gamma_2}\ell(L)^{m+2+\gamma_2}
		\\&\quad\le 4\int_{\bC_{\sfrac{\ell(L)}{16}}(q',\pi_H)} |\vec \bT_F(p)-\vec\cM(\p(p))|^2\dd\|\bT_F\|(p)+C(M_0) \bmo\ell(L)^{m+2}+
		C_1\bmo^{1+\gamma_2}\ell(L)^{m+2+\gamma_2}
	\end{align*}
	so that, recalling \eqref{kdsaa} and using \eqref{dsaacs1}
	\begin{equation}\label{fredasc}
		\begin{split}
			\sfrac{\omega_m}{4}\bar \gamma C_e\bmo\ell(L)^{m+2-2\delta_2}&\le 8\int_{\bC_{\sfrac{\ell(L)}{16}}(q',\pi_H)} |\vec \bT_F(p)-\vec\cM(\p(p))|^2\dd\|\bT_F\|(p)\\&\quad+ C(M_0) \bmo\ell(L)^{m+2}+
			C_1\bmo^{1+\gamma_2}\ell(L)^{m+2+\gamma_2}
		\end{split}
	\end{equation}
	
	By using \cite[Proposition 3.4]{DSsns} (if $\eps_0$ is small enough), recalling \eqref{klcsakc} and \eqref{klcsakc1}, plugging in also \eqref{fredasc},
	\begin{align*}
		\int_{\Omega}|D N|^2&\ge \int_{\p^{-1}(\Omega)}|\vec \bT_{F}(p)-\vec\cM(\p(p))|^2\dd\|\bT_F\|(p)-C_0\bigg(\bmo\int_\Omega|N|^2+ \int_\Omega|DN|^4\bigg)
		\\&\ge 	\sfrac{\omega_m}{32}\bar \gamma C_e\bmo\ell(L)^{m+2-2\delta_2}- C(M_0) \bmo\ell(L)^{m+2}-
		C_1\bmo^{1+\gamma_2}\ell(L)^{m+2+\gamma_2} \\&\quad- C_1\big(\bmo^2 \ell(L)^{m+2+2\beta_2} +\bmo^{1+2\gamma_2}\ell(L)^{m+2-2\delta_2+2\gamma_2}\big)
		\\&\ge  \sfrac{\omega_m}{64}\bar \gamma \bmo\ell(L)^{m+2-2\delta_2},
	\end{align*}
	provided that we first choose  $N_0\ge C(M_0)$ and then $\eps_0$ sufficiently small (depending upon all other parameters), recalling that $\bar\gamma$ depends only on $M_0$ and a geometric quantity and that the $C_e$ appearing can be estimated from below by $5$. Therefore, the second inequality in \eqref{e:split_1} follows from \eqref{eq2}.
	\medskip\\\textbf{Step 5: proof of \eqref{e:split_2}.} 
	The first inequality in \eqref{e:split_2} follows from \eqref{e:Dir_regional}, taking into account that $L\in\sW_e$, so we only have to show the second inequality.
	By \cite[(5.4)]{DSsns}, for every $p\in\Omega$ (we write $p=\Phii_{p_H}(x)$, with $x\defeq \p_{p_H+\pi_H}(p)$),
	\begin{align*}
		2\sqrt{Q}|N'|(p)&=2\sqrt{Q}|N'|(\Phii_{p_H}(x))\ge\cG(f(x),Q\a{\bPsi_{p_H}(x)})
		\\&\ge \cG(f(x),Q\a{0})-\cG(Q\a{0},Q\a{\bPsi_{p_H}(x)})\\&\ge
		|f(x)|-C_0\bmo^{\sfrac{1}{2}}(4\rho)^2,
	\end{align*}
	so that 
	\begin{equation*}
		|f(x)|^2\le C_0 |N'|^2(\Phii_{p_H}(x))+ C_0\bmo M_0^4\ell(L)^4\quad\text{for every $x\in\p_{p_H+\pi_H}(\Omega)$.}
	\end{equation*}
	Therefore, recalling \eqref{ffdadc},  by the area formula and \eqref{dsaacs2},
	\begin{align*}
		\sfrac{\omega_m}{4}\bar \gamma C_e\bmo\ell(L)^{m+2-2\delta_2}&\le 	\int_{B_{\sfrac{\ell(L)}{16}} (q',\pi_H)}|f|^2\le \int_{\p_{p_H+\pi_H}(\Omega)}|f|^2
		\\&\le C_0\int_{\Omega}|N'|^2+C_0\bmo M_0^4\ell(L)^4\omega_m \ell(L)^m
		\\&\le C_0\int_{\Omega}|N'|^2+\sfrac{\omega_m}{4}\bar\gamma C_e \bmo\ell(L)^{m+2-2\delta_2}\big(C_0\bar\gamma^{-1} M_0^4\ell(L)^{2+2\delta_2}\big).
	\end{align*}
	Recall that $\bar \gamma$ depends only on a geometric quantity and on $M_0$, so that we can take $N_0\ge C(M_0)$ and deduce from the above chain of inequalities that 
	\begin{equation}\label{asccsa}
		\sfrac{\omega_m}{8}\bar \gamma C_e\bmo\ell(L)^{m+2-2\delta_2}\le C_0\int_{\Omega}|N'|^2.
	\end{equation}

 Now recall \eqref{ffdadc} with \eqref{ffdadc0}, hence, as in \eqref{cdscasa1}, we have that
 \begin{equation}\notag
     N(p)=N'(p)\quad\text{for every }p\in \Omega\cap\cK\setminus\p(\mathscr{D}(J)),
 \end{equation}
 so that, by \eqref{asasa} and \eqref{ldsani},
 \begin{equation}\label{asccsa1}
		\mathcal{H}^m(\{p\in\Omega:N(p)\ne N'(p)\})\le C_1\bmo^{1+\gamma_2}\ell(L)^{m+2+\gamma_2}.
	\end{equation}

	Now, we use \eqref{eq2}, we combine \eqref{asccsa} and \eqref{asccsa1} (with $|N'|\le C_0$ on $\Omega$ -- for example, recall \eqref{rafdsa}), and we use again $C_e\bmo \ell(L)^{2-2\delta_2}\le \bE(T,\bC_L,\pi_L)$ to estimate
	\begin{align*}
		\ell(L)^m\bE(T,\bC_L,\pi_L)&\le C_1\bmo\ell(L)^{m+2-\delta_2}\le
		C_1\sfrac{\omega_m}{8}\bar \gamma C_e\bmo\ell(L)^{m+2-2\delta_2}\\&\le
		C_1\int_{\Omega}|N|^2+C_1\bmo^{1+\gamma_2}\ell(L)^{m+2+\gamma_2}\\&\le 
		C_1\int_{\Omega}|N|^2+C_1\bmo^{\gamma_2} \ell(L)^m\bE(T,\bC_L,\pi_L).
	\end{align*}
	Therefore, the second inequality in \eqref{e:split_2} follows, provided that $\eps_2$ is small enough.
\end{proof}
\subsection{Proof of the results of Section \ref{sectcomp}}
\begin{proof}[Proof of Proposition \ref{comp1}]
	Let $J\in\sS$ be the ancestor of $L$ with $\ell(J)=2^5\ell(L)$. Then $|x_J|\le \sqrt{m}\ell(J)+c_s^{-1}\ell(J)\le 128\sqrt{m}\ell(J)\le M_0^{-1}128r_J\le 2r_J$, as $M_0\ge 64$, so that $|p_J|\le \sqrt{1+C_0\bmo}2r_J$. In particular, if $\eps_0$ is smaller than a geometric constant, $0\in\bC_{3r_J}(p_J,\pi_J)$. Moreover $E\defeq \bE(T,\bC_J,\pi_J)\le C_e \bmo\ell(J)^{2-2\delta_2}$, so that, we can argue as at the beginning of the proof of Proposition \ref{p:splitting}, building upon Lemma \ref{giraerigira} and Lemma \ref{heighapp}, to deduce that there exist $y_1,\dots,y_Q\in \pi_J^\perp$ such that $$\supp(T)\cap \bC_{4 r_J}(p_J,\pi_J)\subseteq\bigcup_{i=1}^Q( \pi_J\times B_{ C_0 r_J E^{\sfrac{1}{2}}}(y_i,\pi_J^\perp)).$$
	We can assume (up to  enlarging $C_0$) that $\supp(T)\cap \bC_{4r_J}(p_J,\pi_J)\subseteq\bigcup_{i=1}^q\bS_i$, where $1\le q\le Q$ and $\bS_i=\pi_J\times A_i$, for some $A_i\subseteq\pi_J^\perp$  with $$\diam(A_i)\le C_0 r_JE^{\sfrac{1}{2}}\le C_0r_J(C_e\bmo\ell(J)^{2-2\delta_2})^{\sfrac{1}{2}}$$  and such that $(A_i)_{i=1,\dots,q}$ are open and pairwise disjoint. By the smallness of $E$, we can moreover assume that $\p_{\pi_J}(T\mres (\bS_i\cap\bC_{4r_J}(p_J,\pi_J)))=Q_i\a{B_{4r_J}(p_J,\pi_J)}$, for integers $ Q_i\ge 1$, with $\sum_{i=1}^qQ_i=Q$. By Assumption \ref{ipotesi}, $0\in\bC_{3r_J}(p_J,\pi_J )$ is a point of density $Q$ for $T$. This easily implies that we can take $q=1$ (see the beginning of the proof of \cite[Lemma A.2]{DS2}), in particular,
	\begin{equation}\label{lcsa}
		\bs(T,\bC_{4r_J}(p_j,\pi_J),\pi_J)\le C_0C_e^{\sfrac{1}{2}}M_0\bmo^{\sfrac{1}{2}}\ell(J)^{2-\delta_2}.
	\end{equation}
	Now, we can slightly improve the proof of item $\rm(iv)$ of Proposition \ref{prop:tilting} to obtain that $\supp(T)\cap\bC_L\subseteq \bC_{4r_J}(p_J,\pi_J)$, we just give the main computation: take $z\in\supp(T)\cap \bC_L$,
	\begin{align*}
		|\p_{\pi_J}(z-p_J)|&\le |p_L-p_J|+ 	|\p_{\pi_J}-\p_{\pi_L}||z-p_L|+|\p_{\pi_L}(z-p_L)|
		\\&\le 2\sqrt{m} \ell(J)+C_0\bmo^{\sfrac{1}{2}}\ell(J)+64 r_L\le 128r_L
		=4 r_J,
	\end{align*}
	provided that $\eps_0$ is smaller than a geometric constant.
	Therefore  we take $z\in\supp(T)\cap\bC_L \subseteq\bC_{4r_J}(p_J,\pi_J)$ and we compute, recalling item $\rm(ii)$ of Proposition \ref{prop:tilting}, \eqref{eq4},\eqref{lcsa}, and the estimate on $\|\bPsi_{p_J}\|_{C^2}$ (together with $0\in\cM$ and the fact that $|\p_{\pi_J^\perp}(p_J)|=|\p_{p_J+\pi_J}(0)|$),
	\begin{align*}
		|\p_{\pi_L^\perp}(z-p_L)|&\le 	|\p_{\pi_L^\perp}-\p_{\pi_J^\perp}||z-p_L|+	|\p_{\pi_J^\perp}(z)|+|\p_{\pi_J^\perp}(p_J)|+|\p_{\pi_J^\perp}(p_L-p_J)|
		\\&\le C_0\bmo^{\sfrac{1}{2}}\ell(L)128r_L+C_0C_e^{\sfrac{1}{2}}M_0\bmo^{\sfrac{1}{2}}\ell(J)^{2-\delta_2}+C_0M_0^2\bmo^{\sfrac{1}{2}}\ell(L)^2+C_0\bmo^{\sfrac{1}{2}}\ell(L)^2.
	\end{align*}
	Therefore, if $C_h\ge C(M_0,C_e)$, $L\notin \sW_h$. Also, as in the proof of \cite[Proposition 3.7]{DS2}, $L\notin \sW_n$. Indeed, if it was not the case, there would be $L'\in\sW$ with $\ell(L')=2\ell(L)$ and $L'\cap L\ne \emptyset$, so that, by item $\rm(c)$, $$\dist(0,L')\le \dist(0,L)+\sqrt{m}\ell(L)\le c_s^{-1}\ell(L)+\sqrt{m}\ell(L)\le c_s^{-1}\ell(L'),$$ but $c_s^{-1}\ell(L')=2c_s^{-1}\ell(L)=2s>s$, contradicting item $\rm(a)$. Hence $L\in\sW_e$. 
	
	It remains to prove \eqref{comp1eq}.  We use Proposition \ref{p:splitting}, combining \eqref{e:split_1} with \eqref{e:split_2} to obtain 
	\begin{equation}\label{vfdaa1}
		C_e\bmo\ell(L)^{m+4-2\delta_2}\le C\int_\Omega|N|^2.
	\end{equation}
	
	As in \textbf{Step 3}  of  the proof  of Proposition \ref{p:splitting}, we have that every Whitney region $\cL'$ that intersects $\Omega$ corresponds to $L'$ with $\ell(L')\le C_0\ell(L)$. In particular, \eqref{asasa}, \eqref{klcsakc} and \eqref{klcsakc1} still hold. 
	
	Recalling \cite[Lemma 1.9]{DSsns} (see also the brief discussion after \cite[Definition 1.10]{DSsns}) and \eqref{klcsakc1},
	\begin{equation}\label{vfdaa2}
		\int_{\p^{-1}(\Omega)}\dist^2(x,\cM)\dd\|\bT_F\|(x)\ge \sfrac{1}{2}\int_{\Omega}\sum_{j}\dist^2(F^j,\cM)= \sfrac{1}{2}\int_{\Omega}|N|^2
	\end{equation}
	provided that $\eps_0$ is small enough (depending upon all other parameters).
	Notice also that for every $q\in\p^{-1}(p)\cap\Im(F)$,  by \eqref{klcsakc}
	\begin{align*}
		\dist^2(q,\cM)\le |N(\p(q))|^2\le C_1\bmo\ell(L)^{2+2\beta_2},
	\end{align*}
	so that, using  \eqref{asasa},
	\begin{equation}\label{vfdaa3}
		\int_{\p^{-1}(\Omega)}\dist^2(x,\cM)\dd\|\bT_F\|(x)\le 	\int_{\p^{-1}(\Omega)}\dist^2(x,\cM)\dd\|T\|(x)+C_1\bmo^{2+\gamma_2}\ell(L)^{m+4+2\beta_2+\gamma_2}.
	\end{equation}
	Combining \eqref{vfdaa1}, \eqref{vfdaa2} and \eqref{vfdaa3}, 
	\begin{align*}
		C_e\bmo\ell(L)^{m+4-2\delta_2}\le  C_1\int_{\p^{-1}(\Omega)}\dist^2(x,\cM)\dd\|T\|(x)+C_1\bmo^{2+\gamma_2}\ell(L)^{m+4+2\beta_2+\gamma_2},
	\end{align*}
	which implies \eqref{comp1eq}, provided that $\eps_0$ is small enough.
\end{proof}
\begin{proof}[Proof of Proposition \ref{comp2}]
	Set $\bar\ell\defeq \sup_{L\in\sW:L\cap (B_{\sfrac{11}{4}}\setminus B_{\sfrac{9}{4}})\ne\emptyset}\ell(L)$
	Recalling item $\rm(iii)$ of Corollary \ref{c:cover} and \eqref{vfcdsa}, it is enough to prove that
	\begin{equation}\notag
		\sum_{L\in\sW:L\cap (B_{\sfrac{11}{4}}\setminus B_{\sfrac{9}{4}})\ne\emptyset}\int_{\p^{-1}(\Phii(L))}\dist^2(x,\cM)\dd\|T\|(x)\le C_1\bigg(\bmo^{1+\gamma_2}\bar\ell+\int_{{\cB_{\sfrac{21}{8}}\setminus \cB_{\sfrac{19}{8}}}}|N|^2\bigg).
	\end{equation}
	Fix $L\in\sW$ with  $L\cap (B_{\sfrac{11}{4}}\setminus B_{\sfrac{9}{4}})\ne\emptyset$.
	By item $\rm(ii)$ of Corollary \ref{c:cover}, for every $x\in\supp(T)\cap\p^{-1}(\Phii(L))$, $\dist(x,\cM)\le C_1\bmo^{\sfrac{1}{2}}\ell(L)^{1+\beta_2}$. Notice now that if $\cL$ is the Whitney region associated to $L$, then $\cL\supseteq\Phii(L)$, so that we can use \eqref{e:err_regional} and obtain, recalling \cite[Lemma 1.9]{DSsns} (with \eqref{e:Lip_regional}),
	\begin{align*}
		\int_{\p^{-1}(\Phii(L))}\dist^2(x,\cM)\dd\|T\|(x)&\le \int_{\p^{-1}(\Phii(L))}\dist^2(x,\cM)\dd\|\bT_F\|(x) + C_1\bmo^{2+\gamma_2}\ell(L)^{m+4+2\beta_2+\gamma_2}\\&\le 
		2\int_{\Phii(L)} |N|^2+ C_1\bmo^{1+\gamma_2}\ell(L)^{m}\bar\ell.
	\end{align*}
	The claim follows by summing over all such $L$.
\end{proof}

\part{Logarithmic derivatives}\label{part2}
In this part we will define a modified frequency function and prove its monotonicity and boundedness in certain ``intervals of flattening''. Those intervals correspond to the scales at which the normal approximation $N$ we constructed in Part \ref{sect:normapprox} is a good approximation of $T$. We will also prove differential inequalities to compare the different integral quantities that make up the frequency. As mentioned in the Introduction, in this Part we follow quite closely the last paper of the De Lellis and Spadaro program, \cite{DS3}. In particular, Section \ref{sectif} corresponds to Section 2 there, Section \ref{freq} to Section 3 and Section \ref{freqbound} to Section 5. In keeping with Part \ref{sect:normapprox}, we isolated the proofs of all statements in Section \ref{prooffreq}, while in \cite{DS3} they are contained in the respective sections. 

\section{Intervals of flattening}\label{sectif}
\textbf{The standing assumption, for this part, is the following:}
\begin{ass}\label{ass:sectfour}
	$T$ and $\Sigma$ are as in Assumption \ref{ipotesi}. Moreover, the choices for the parameters as in \ref{parametri} (in particular, the choice of $\eps_0$) are so that all the results of Section \ref{sect:normapprox} hold. In addition, we require $\cM$ to be a minimal surface, in the sense that the mean curvature $H_\cM$ vanishes.
	\fr
\end{ass}

We assume that  the parameters as in \ref{parametri} \emph{except} $\eps_0$ are fixed at this point. We recall our convention to denote two particular types of constants in \eqref{constants}.

\bigskip

Here and after, we consider the rescaling map $\iota_{0,r}$ defined as $z\mapsto \sfrac{z}{r}$.
Define
\begin{equation}\notag
	\mathcal{R}\defeq\big\{r\in ]0,1]: \bE (T, \B_{6\sqrt{m} r},\pi_0) \le \eps_0^2\big\}.
\end{equation}
Observe that, if $(s_k)_k\subseteq \mathcal{R}$
and $s_k\nearrow s$, then $s\in \mathcal{R}$.
We cover $\cR$ with a collection $\mathcal{F}=\{I_j\}_j$ of intervals
$I_j = ]s_j, t_j]$ defined as follows.
$t_0\defeq \max \{t: t\in \mathcal{R}\}$. Next assume, by induction, to have defined $t_j$ (and hence also
$t_0 > s_0\ge t_1 > s_1 \ge \ldots > s_{j-1}\ge t_j$) and consider $$T_j^0 \defeq ((\iota_{0,t_{j}})_\sharp T)\res \B_{6\sqrt{m}}\quad\text{and}\quad\cM_j \defeq \iota_{0, t_j} (\cM).$$
Notice that $T_j$, $\cM_j$ still satisfy Assumption \ref{ass:sectfour}, with 	$\bmo^{(j)} \defeq \max\{\mathbf{c}(\cM_j)^2 , \bE(T_j, \bB_{6\sqrt{m}},\pi_0)\}$, we are going to use the obvious notation for what concerns e.g.\ $\bPsi^{(j)}$. We let $T_j$ denote the current obtained by Lemma \ref{height}, from the current $T_j^0$.

Then, we consider the Whitney decomposition $\sW^{(j)}$ of $[-4,4]^m \subseteq \pi_0$ as in Proposition \ref{p:whitney} (applied to $T_j$, $\cM_j$) and we define
	\begin{equation}\label{e:s_j}
		s_j \defeq t_j \max \left(\{c_s^{-1} \ell (L) : L\in \sW^{(j)} \mbox{ and } c_s^{-1} \ell (L) \ge \dist (0, L)\} \cup \{0\} \right) .
	\end{equation} 
	We will prove below that $\sfrac{s_j}{t_j} <\sfrac{2^{-5}}{\sqrt{m}}$. In particular this ensures that $[s_j, t_j]$ is a (non-trivial) interval. 
	Next, if $s_j =0$ we stop the induction. Otherwise we let $t_{j+1}$ be the largest element
	in $\mathcal{R}\cap ]0, s_j]$ and proceed as above. Note moreover the following simple consequence of \eqref{e:s_j}:
	\begin{itemize}
		\item[(Stop)] If $s_j >0$ and $\bar{r} \defeq \sfrac{s_j}{t_j}$, then there is $L\in \sW^{(j)}$ with 
		\begin{equation}\notag
			\ell(L) = c_s\bar r \qquad \mbox{and} \qquad L\cap \bar B_{\bar r} (0, \pi_0)\neq \emptyset.
		\end{equation}
		\item[(Go)] If $\rho > \bar{r} \defeq \sfrac{s_j}{t_j}$, then
		\begin{equation}\notag
			\ell (L) < c_s \rho \qquad \mbox{for all } L\in \sW^{j(k)} \mbox{ with $L\cap B_\rho (0, \pi_0) \neq \emptyset$.}
		\end{equation}
		In particular the latter inequality is true for every $\rho\in ]0,3]$ if $s_j =0$.
\end{itemize}

The next proposition lists some quick consequences of the definition. It is for the most part completely analogous to Proposition 2.2 in \cite{DS3}, but note that in point (iv) we keep track of the $j$ dependence with $\bmo^{(j)}$ -- we will need this in Proposition \ref{compH} below, to control $\bmo^{(j)}$ with something at the preceding scale. Given two sets $A$ and $B$, we define their
{\em separation} as the number ${\sep} (A,B) \defeq \inf \{|x-y|:x\in A, y\in B\}$.

\begin{prop}\label{p:flattening} For some constant $C_1$,
	\begin{enumerate}[label=\rm(\roman*)]
		\item $s_j < \sfrac{t_j2^{-5}}{\sqrt{m}}$ and the family $\mathcal{F}$ is either countable
		and $t_j\downarrow 0$, or finite and $I_j = ]0, t_j]$ for the largest $j$;
		\item the union of the intervals of $\cF$ cover $\cR$;
		\item if $r \in ]\sfrac{s_j}{t_j},3[$ and $J\in \mathscr{W}^{(j)}_n$ intersects
		$B\defeq \p_{\pi_0} (\cB_r )$,
		then $J$ is in the
		domain of influence $\mathscr{W}_n^{(j)} (H)$ (see \cite[Definition 3.3]{DS2})
		of a cube $H\in \mathscr{W}^{(j)}_e$ with
		\[
		\ell (H)\le 3  c_s r \quad \text{and}\quad 
		\max\left\{{\sep} (H, B),  {\sep} (H, J)\right\}
		\le 3\sqrt{m} \ell (H) \le \sfrac{3 r}{16};
		\]
		\item  for
		every $r\in]\sfrac{s_j}{t_j},3[$, $$\bE (T_j, \B_r ,\pi_0)\le C_1  \bmo^{(j)} r^{2-2\delta_2};$$
		\item  for
		every $r\in]\sfrac{s_j}{t_j},3[$, $$\sup \{ \dist (x,\cM_j): x\in \supp(T_j) \cap \p^{-1}_j(\cB_r)\} \le C_0 (\bmo^{(j)})^{\sfrac{1}{2}} r^{1+\beta_2}.$$ 

	\end{enumerate}
\end{prop}

\section{Frequency function and first variations}\label{s:frequency}\label{freq}

For every interval of flattening $I_j = ] s_j, t_j]$,
let $N_j$ be the normal approximation of $T_j$
on $\cM_j$ in Theorem \ref{t:approx}. 
For convenience of notation, in this subsection, \textbf{we occasionally suppress the index $j$ when it is clear from the context.}

Consider the following Lipschitz (piecewise linear) function $\phi:[0+\infty[ \to [0,1]$ given by
\begin{equation*}
	\phi (r) \defeq
	\begin{cases}
		1 & \text{for } r\in [0,\textstyle{\sfrac{1}{2}}],\\
		2-2r & \text{for } r\in  ]\textstyle{\sfrac{1}{2}},1],\\
		0 & \text{for } r\in  ]1,+\infty[.
	\end{cases}
\end{equation*}

\begin{defn}[Frequency functions]
	For every $r\in ]0,3]$ we define: 
	\[
	\bD (r) \defeq \int_{\cM^j} \phi\left(\frac{d(p)}{r}
	\right)|D N|^2(p) \dd\mathcal{H}^m(p) \quad\text{and}\quad
	\bH (r) \defeq - \int_{\cM^j} \phi'\left(\frac{d (p)}{r}\right)\frac{|N|^2(p)}{d(p)} \dd\mathcal{H}^m(p) ,
	\]
	where $d (p)$ is the geodesic distance on $\cM$ between $p$ and $0=\Phii (0)$.
	If $\bH (r) > 0$, we define the {\em frequency function}
	 as $$\bI (r) \defeq \frac{r\bD (r)}{\bH (r)}.$$
\end{defn}

\begin{defn}
	We let $\de_{\hat r}$ denote the derivative with respect to arclength along geodesics starting at $0$. We set
	\begin{equation*}
				\bE (r) \defeq - \int_\cM \phi'\left({\frac{d(p)}{r}}\right)\sum_{i=1}^Q \langle
		N_i(p), \de_{\hat r} N_i (p)\rangle \dd\mathcal{H}^m(p),
	\end{equation*}
\begin{equation*}
			\bG (r) \defeq - \int_{\cM} \phi'\left({\frac{d(p)}{r}}\right)\dd\mathcal{H}^m(p) \left|\de_{\hat r} N (p)\right|^2 \dd\mathcal{H}^m(p)
		\quad\text{and}\quad
		\bSigma (r) \defeq\int_\cM \phi\left({\frac{d(p)}{r}}\right) |N|^2(p) \dd\mathcal{H}^m(p).
\end{equation*}

\end{defn}

\begin{rem}Observe that all these functions of $r$ are absolutely continuous
	and, therefore, classically differentiable at almost every $r$.
	Moreover, the following rough estimates easily follows from
	Theorem \ref{t:approx} and the condition (Go):
	\begin{alignat}{2}\label{e:rough}
		\bD(r) &\le \  \int_{\cB_r} |DN|^2 &\le C_1 \bmo r^{m+2-2\delta_2} \quad \text{for every}\quad r\in\left]{\sfrac{s}{t}},3\right[,\\
		\bH(r)& \le  r^{-1}\int_{\cB_r } |N|^2& \le  C_1 \bmo r^{m+1+2\beta_2} \quad \text{for every}\quad r\in\left]{\sfrac{s}{t}},3\right[,\label{N:rough}
	\end{alignat}
	Indeed, since $N$ vanishes identically on the set $\mathcal{K}$ of Theorem \ref{t:approx}, it suffices to sum
	the estimate \eqref{e:Dir_regional} over all the different cubes $L$ (of the corresponding Whitney decomposition) for which $\Phii (L)$ intersects the geodesic ball $\cB_r$.\fr 
\end{rem}

\begin{thm}[Almost monotonicity of the frequency]\label{t:frequency}
	Assume that $\eps_0$ is small enough (depending upon all other parameters). Then, for every
	$[a,b]\subseteq [\sfrac{s_j}{t_j}, 3]$ with $\bH_j \vert_{[a,b]} >0$, 
	\begin{equation}\notag
		\bI_j (a) \le C_1 (1 + \bI_j (b)).
	\end{equation}
\end{thm}

\begin{prop}[First variation estimates]\label{p:variation}
	For every $\gamma_3$ sufficiently small (depending upon all other parameters) there is a constant $ C=(C_1,\gamma_3)>0$  (in the sense that $C$ depends on $\gamma_3$ as well as the parameters as for $C_1$) such that, if $\eps_0$ is sufficiency small (depending upon a geometric quantity),  then for a.e.\ $r\in[\sfrac{s}{t},3]$
	\begin{eqnarray}
			&&\big|\bH' (r) - {\sfrac{(m-1)}{r}} \bH (r) - {\sfrac{2}{r}}\bE(r)\big|\le  C \bH (r), \label{e:H'}\\
		&&\big|\bD (r)  - \sfrac{1}{r} \bE (r)\big| \le C \big(\bD(r)+r\bH(r)\big)^{1+\gamma_3} + C \eps_0^2 \bSigma (r),\label{e:out}\\
&&\begin{aligned}[t]
			\big| \bD'(r) - {\sfrac{(m-2)}{r}} \bD(r) - {\sfrac{2}{r^2}}\bG (r)\big|&\le C (C_1,\gamma_3) \big(\bD (r)+r\bH(r)\big)^{\gamma_3} \big(\bD' (r) +\sfrac{1}{r} \bD(r)+\bH(r)\big)\\&\quad+C_1\big( \bD (r)+\bH(r)\big),\label{e:in}
\end{aligned}\\
		&&\bSigma (r) +r\bSigma'(r) \le C   r^2 \bD (r) +Cr\bH(r).\label{e:Sigma1}
	\end{eqnarray}
\end{prop}

\begin{prop}\label{small} 	Assume that $\eps_0$ is small enough (depending upon all other parameters). Then, 
		\begin{equation}\notag
			\int_{\cB_r}|N|^2\le 4 \int_{\B_{2r}}\dist^2(x,\cM)\dd\|T\|(x)\quad\text{for every }r\in [\sfrac{s}{t},3].
		\end{equation}

\end{prop}

\section{Boundedness of the frequency function}\label{freqbound}
\begin{prop}\label{compH}
If $\eps_0$ is small enough (depending upon the parameters as for $C_1$),
for any $j$ such that $t_j=s_{j-1}$,
	\begin{equation*}
	\bmo^{(j)}\le C_1	\bmo^{(j-1)} \Big(\frac{s_{j-1}}{t_{j-1}}\Big)^{2-2\delta_2}\le C_1^2\bH_j(3).
	\end{equation*} 
	
\end{prop}
\begin{thm}[Boundedness of the frequency function]\label{t:boundedness}
	 Assume that $\eps_0$ is small enough (depending upon the parameters as for $C_1$). Assume moreover that  $\supp (T)\cap \bB_r$ is not contained in $\cM$, for every $r\in ]0,1[$. Then
	 \begin{itemize}
	 	\item [\rm$(i)$] If the intervals of flattening are $j_0 < \infty$, then there is $\rho\in(0,1)$ such that
	 	\begin{equation}\label{e:finita1}
	 	\bH_{j_0} >0 \mbox{ on $]0, \rho[$} \quad \mbox{and} \quad \limsup_{r\to 0} \bI_{j_0} (r)< \infty .
	 \end{equation}
	 	\item [\rm$(ii)$] 	If the intervals of flattening are infinitely many, then there is a number $j_0\in \mathbb N$ such that
	 	\begin{equation}\label{e:finita2}
	 		\bH_j>0 \mbox{ on $]\sfrac{s_j}{t_j}, { 3}[$ for all $j\ge j_0$} , \qquad 
	 		\sup_{j\ge j_0} \sup_{r\in ]\sfrac{s_j}{t_j}, {  3}[} \bI_j (r) <\infty,\quad\text{and}\quad\inf_{\{j\ge j_0:t_{j}<s_{j-1}\}}\bH_j(3)>0.
	 	\end{equation}
	 \end{itemize}
\end{thm}

\section{Proof of the results of Part \ref{part2}}\label{prooffreq}
\subsection{Proof of the results of Section \ref{sectif}}
\begin{proof}[Proof of Proposition \ref{p:flattening}]
	We can follow \textit{verbatim} the proof of \cite[Proposition 2.2]{DS3}, except for items $\rm (i)$ and $\rm (iv)$.  For item $\rm(i)$, notice that for some $L\in \sW^{(j)}$, by \eqref{e:prima_parte},
	\begin{equation*}
		{s_j}= t_j c_s^{-1}\ell(L)\le t_j c_s^{-1}2^{-N_0-6},
	\end{equation*}
	so that the conclusion follows from \eqref{fdasc}.
	We turn to item $\rm(iv)$. In the case in which $r\ge 2^{-N_0}$, we can still use the argument of \cite{DS3}, otherwise, we have to slightly modify the argument. Then, let $k\ge N_0$ be the smallest natural number such that $2^{-k+1}>r$ and let $L\in\sC^{k}$ by a cube with $0\in L$ (i.e.\ $\ell(L)=2^{-k}$). By the condition (GO), no ancestor of $L$ belongs to $\sW^{(j)}$, and the same holds for $L$. Therefore, $L\in\sS^{(j)}$.
	Also, we have $|x_L|\le \sqrt{m}\ell(L)$ and hence, by \eqref{cdscas}, $|p_L|\le 2  \sqrt{m}\ell(L)$, so that $\B_r\subseteq\bC_L^{(j)}$.
	Therefore, using the fact that $L\in\sS^{(j)}$, \eqref{eq3}, and the bound on $\|\bPsi^{(j)}\|_{C^2}$,
	\begin{align*}
		\bE(T_j,\B_r,\pi_0)&\le \big(\sfrac{r_L}{r}\big)^m\bE\big(T_j,\bC_L^{(j)},\pi_0\big)\le C_1\bE\big(T_j,\bC_L^{(j)},\pi_0\big)\\&\le C_1\bE\big(T_j,\bC_L^{(j)},\pi_L^{(j)}\big)+C_1 r_L^{-m}|\vec \pi_0-\vec\pi_L^{(j)}|^2\|T_j\|(\bC_L^{(j)})
		\\&\le C_1\bmo^{(j)}\ell(L)^{2-2\delta_2}+ C_1 \bmo^{(j)}r^2,
	\end{align*}
	so that we conclude the proof.
	\end{proof}
\subsection{Proof of the results of Section \ref{s:frequency}} We still avoid overloading the notation with the index $j$.
Theorem \ref{t:frequency} follows from Proposition \ref{p:variation} exactly as in the proof of \cite[Theorem 3.2]{DS3} (we have to plug into the estimates of Proposition \ref{p:variation} the bound $\bI\ge 1$, see Lemma \ref{l:poincare'} below).
We now turn to the proof of Proposition \ref{t:frequency}, we prove the various estimates separately. \eqref{e:H'}  follows exactly as its counterpart in \cite{DS3}, see \cite[Section 3.1]{DS3}. Towards the proof of \eqref{e:Sigma1}, we have a lemma, which is the counterpart of \cite[Lemma 3.6]{DS3} and is proved with the very same computations (see \cite[Section 3.2]{DS3}). Notice that the validity of \eqref{e:Sigma1} is part of the statement of Lemma \ref{l:poincare'}.
\begin{lem}\label{l:poincare'} It holds that
	\begin{alignat}{1}\notag
		\bSigma (r) \le C_0 r^2\bD (r) + &C_0 r \bH (r) \quad \text{and} \quad 
		\bSigma' (r) \le C_0 \bH (r),
	\\\notag
		\int_{\mathcal{B}_r } |N|^2 &\le C_0\bSigma (r) + C_0r\bH(r),
\\\notag
		\int_{\mathcal{B}_r} |DN|^2 &\le C_0\bD (r) + C_0r \bD' (r).
	\end{alignat}
	In particular, if $\bI \ge 1$, then \eqref{e:Sigma1} holds and
	\begin{equation}\notag
		\int_{\mathcal{B}_r } |N|^2 \le C_0r^2\bD (r).
	\end{equation}
\end{lem}
Therefore, it remains to prove \eqref{e:out} and \eqref{e:in}. We follow closely the arguments of \cite[Section 3.3 and Section 4]{DS3}. In particular, we can compute first variations of $\bT_F$ exactly as in \cite[Section 3.3]{DS3} and obtain, keeping the same notation (see \cite[(3.19)]{DS3}),
\begin{equation}\label{e:Err4-5}
	|\delta \bT_F (X)| = |\delta \bT_F (X) - \delta T (X)| \le \underbrace{\int_{\supp (T)\setminus \im (F)}  \big|\dv_{\vec T} X\big| \dd\|T\|
		+ \int_{\im (F)\setminus \supp (T)} \big|\dv_{\vec \bT_F} X\big| \dd\|\bT_F\|}_{{\rm Err}_4},
\end{equation}
where we notice that ${\rm Err}_5$  of \cite[(3.19)]{DS3} is not present in our situation as $T$ is an area-minimizing current and hence $\delta T (X) =\delta T(X^\perp)=0$ (in \cite{DS1,DS2,DS3}, the current $T$ is area-minimizing with respect to perturbations whose support lie in a prescribed Riemannian manifold, and this causes the presence of the additional error term ${\rm Err}_5$).
Therefore, with the same argument of \cite[Section 3.3]{DS3}, building upon \cite{DSsns}, we obtain  (towards  \eqref{e:out} and \eqref{e:in}, respectively)
\begin{gather}\label{e:out1}
	\big| \bD (r) - r^{-1} \bE (r)\big| \le \sum_{j=1}^4 \big|{\rm Err}^o_j\big|,	\\
\big| \bD' (r) - \sfrac{m-2}{r} \bD (r) - \sfrac{2}{r^2} \bG (r)\big|\label{e:in1}
\le C_0 \bD (r) + \sum_{j=1}^4 \big|{\rm Err}^i_j\big|,
\end{gather}
where we notice that the terms ${\rm Err}_5^o$ and ${\rm Err}_5^i$ do not appear here, as explained above. 
By \cite[(3.21)--(3.23)]{DS3} and \cite[(3.26)--(3.28)]{DS3} (which are estimates proved in \cite{DSsns}), keeping the same notation of \cite{DS3}, we have the estimates
\begin{gather}
	{\rm Err}_1^o = - Q \int_\cM \varphi_r \langle H_\cM, \etaa\circ N\rangle,\label{e:outer_resto_1}\\
	|{\rm Err}_2^o| \le C_0 \int_\cM |\varphi_r| |A|^2|N|^2,\label{e:outer_resto_2}\\
	|{\rm Err}_3^o| \le C_0 \int_\cM \big(|N| |A| + |DN|^2 \big) \big( |\varphi_r| |DN|^2  + |D\varphi_r| |DN| |N| \big)\label{e:outer_resto_3},
\end{gather} 
and  the estimates
\begin{gather}
	{\rm Err}_1^i = - Q \int_{ \cM}\big( \langle H_\cM, \etaa \circ N\rangle\, \div_{\cM} Y + \langle D_Y H_\cM, \etaa\circ N\rangle\big),\label{e:inner_resto_1}\\
	|{\rm Err}_2^i| \le C_0 \int_\cM |A|^2 \left(|DY| |N|^2  +|Y| |N| |DN|\right), \label{e:inner_resto_2}\\
	|{\rm Err}_3^i|\le C_0 \int_\cM \Big( |Y| |A| |DN|^2 \big(|N| + |DN|\big) + |DY| \big(|A||N|^2 |DN| + |DN|^4\big)\Big)\label{e:inner_resto_3}.
\end{gather}
 It is crucial to notice that, by \eqref{e:outer_resto_1} and \eqref{e:inner_resto_1},
\begin{equation}\label{err1}
	{\rm Err}_1^o={\rm Err}_1^i=0,
\end{equation}
being $\cM$ a minimal surface, by Assumption \ref{ass:sectfour} (i.e.\ $H_\cM=0$). 
We turn to the estimate of ${\rm Err}_j^o$ and ${\rm Err}_j^i$, for $j=2,3,4$, following \cite[Section 4]{DS3} closely. 

We refer to the construction of \cite[Section 4.1]{DS3} for what concerns ``families of subregions'' and the reference can be followed \textit{verbatim} (this require $\eps_0$ smaller than a geometric constant). We then assume that the reader is familiar with the construction  of \cite[Section 4.1]{DS3}, and we now recall briefly some objects introduced in  \cite[Section 4.2]{DS3}, since we need to show how the estimates of \cite{DS3} can be adapted to our case. First, we take $(\mathcal{B}^i)_i $ and $ (\mathcal{U}_i)_i $ as in \cite[Section 4.2]{DS3}, and we recall $$\mathcal{V}_i \defeq \mathcal{U}_i \cap\p\big(\supp (\bT_F) \Delta \supp (T)\big)\subseteq\mathcal{U}_i\setminus\cK.$$
From  Theorem \ref{t:approx} (with the same argument of \cite{DS3}), we obtain the following estimates,
\begin{gather}
	\int_{\mathcal{U}_i} |DN|^2 \le C_1 \bmo \ell_i^{m+2-2\delta_2},\notag\\
	\|N\|_{C^0 (\mathcal{U}_i)} + \sup_{p\in \supp (T) \cap \p^{-1} (\mathcal{U}_i)} |p- \p (p)| \le C_1 \bmo^{\sfrac{1}{2}} \ell_i^{1+\beta_2},\label{e:N_sopra}\\
	\Lip (N|_{\mathcal{U}_i}) \le C_1 \bmo^{\gamma_2} \ell_i^{\gamma_2},\notag\\\mathcal{H}^m(\mathcal{U}_i\setminus\mathcal{V}_i)+
	\mass (T\res \p^{-1} (\mathcal{V}_i)) + \mass (\bT_F \res \p^{-1} (\mathcal{V}_i)) \le C_1 \bmo^{1+\gamma_2} \ell_i^{m+2+\gamma_2}.\label{e:errori_massa},
\end{gather}
which correspond to \cite[(4.5)--(4.8)]{DS3}. We do not add the corresponding estimate to \cite[(4.4)]{DS3}, as we do not need it, and indeed we are not able to show it in our case, but we add the first summand in \eqref{e:errori_massa}, for future reference.
Now we state the last ingredient used in the error estimates, which corresponds to \cite[Lemma 4.5]{DS3}.
\begin{lem}\label{l:exponent_match}With the same assumptions of Proposition \ref{p:variation}, 
	\begin{align}
		\sum_i \big(\inf_{\cB^i} \varphi_r \big) \bmo 
		\ell_i^{m+2+\sfrac{\gamma_2}{4}} \le C_1 \big(\bD (r)+r\bH(r)\big),\notag\\
		\sum_i \bmo \ell_i^{m+2+\sfrac{\gamma_2}{4}} \le C_1 \left( \bD (r) + r \bD' (r)+r\bH(r)\right),\notag
		\\\sum_i \bmo \ell_i^{m+4-2\delta_2}\le C_1\int_{\cB_r}|N|^2. \label{grvfd}
	\end{align}
	Moreover, there exists $\alpha>0$ such that, for every $t>0$.
	\begin{equation}\notag
		\sup_i
		\bmo^t\Big[{\ell_i}^t + \Big(\inf_{\cB^i} \varphi_r\Big)^{\sfrac{t}{2}} \ell_i^{\sfrac{t}{2}}\Big]
		\le C_1^t \big(\bD (r)+r\bH(r)\big)^{\alpha t}
	\end{equation}
\end{lem}
\begin{proof}
	As in \cite{DS3}, the main estimates towards the proof of the lemma are:
	\begin{alignat}{3}
		\int_{\cB^i} \varphi_r |N|^2 &\ge C_1^{-1} \bmo \inf_{\cB^i} \varphi_r \ell_i^{m+2+2\beta_2}&\ge(C_0C_1)^{-1}  \bmo^{\sfrac{1}{a}} \Big[\ell_i^2 + \Big(\inf_{\cB^i} \varphi_r\Big) \ell_i\Big]^{\sfrac{1}{(2a)}}  \quad 
		&\text{if $L_i\in \mathscr{W}_h$},\label{e:N_sotto}\\
		\int_{\cB^i} \varphi_r |DN|^2 &\ge C_1^{-1} \bmo \inf_{\cB^i} \varphi_r \ell_i^{m+2-2\delta_2}&\ge (C_0C_1)^{-1} \bmo^{\sfrac{1}{a}} \Big[\ell_i^2 + \Big(\inf_{\cB^i} \varphi_r\Big) \ell_i\Big]^{\sfrac{1}{(2a)}}
		\quad&\text{if $L_i \in \mathscr{W}_e$}. \label{e:Dirichlet_sotto}
	\end{alignat}
	Indeed, recall that by \cite[(4.2)]{DS3}, for $\ell_i\defeq\ell(L_i)$,
	$$
	\inf_{\cB^i}\varphi_r\ge (4r)^{-1}\ell_i.
	$$
	Now, the first inequality of \eqref{e:N_sotto} follows from (S3) of Proposition \ref{p:separ} (notice that, by the construction of \cite{DS3}, $\cB_i=\Phii(B(L_i))$, with $B(L_i)\subseteq B_{(\sqrt{m}+1)\ell(L)}(x_L)$ in the case $L\in\sW_h$). Also, first inequality of \eqref{e:Dirichlet_sotto} follows from \eqref{e:split_1} of Proposition \ref{p:splitting} and a similar argument. At this point, the proof of \cite[Lemma 4.5]{DS3}, can be followed \textit{verbatim} to obtain everything but \eqref{grvfd} (we do not use here that $\bI\ge 1$, but we still exploit Lemma \ref{l:poincare'}). Now we show \eqref{grvfd}. 
	First, we notice that 
	\begin{alignat}{2}
			\int_{\cB_i} |N|^2&\ge C_1^{-1}\bmo\ell_i^{m+2+2\beta_2}\quad&\text{if $L_i \in \mathscr{W}_h$},\\
				\int_{\cB_i} |N|^2&\ge C_1^{-1}\bmo\ell_i^{m+4-2\delta_2}\quad&\text{if $L_i \in \mathscr{W}_e$}.
	\end{alignat}
This is proved exactly as for \eqref{e:N_sotto} and \eqref{e:Dirichlet_sotto} (relying also on \eqref{e:split_2}). Then \eqref{grvfd} follows, taking into account that $(\cB^i)_i$ are pairwise disjoint and are contained in $\cB_r$, since, for every $i$, $\cB_i=\Phii(B(L_i))$, and $B(L_i)\subseteq \p_{\pi_0}(\cB_r)$, by the construction of \cite{DS3}.
	\end{proof}
	
	We turn now to discuss the counterpart of \cite[Section 4.3]{DS3}. We claim that, for what concerns outer variations, provided that $\gamma_3$ is smaller than a geometric constant,
	\begin{gather}\label{esto1}
		|{\rm Err}^o_2 |\le C_1 \eps_0^2 \bSigma (r),\\
			|{\rm Err}^o_3 |\le\label{esto2} C(C_1,\gamma_3)\bD(r)^{\gamma_3}\big(r^2\bD(r)+r\bH(r)\big),\\
		|{\rm Err}^o_4 |\le 	C (C_1,\gamma_3) \big(\bD (r)+r\bH(r)\big)^{1+\gamma_3},\label{esto3}
	\end{gather}
	and, for what concerns  inner variations, still assuming that $\gamma_3$ is smaller than a geometric constant,
\begin{gather}\label{esti1}
		\left|{\rm Err}^i_2\right| 	\le C_1\big( \bD (r)+\bH(r)\big),
		\\\label{esti2}
		\left|{\rm Err}^i_3\right| \le C(C_1,\gamma_3)\bD(r)^{\gamma_3}\big(\bD'(r)+\sfrac{1}{r}\bD(r)+r\bH(r)\big),\\
	|{\rm Err}^i_4| \le C (C_1,\gamma_3) \big(\bD (r)+r\bH(r)\big)^{\gamma_3} \big(\bD' (r) +\sfrac{1}{r} \bD(r)+\bH(r)\big).\label{esti3}
\end{gather}
Equations \eqref{esto1}--\eqref{esto2} and \eqref{esti1}--\eqref{esti2} follow with the machinery developed in \cite{DS3}, in particular we can follow \textit{verbatim} \cite[Section 4.3]{DS3}, building upon Lemma \ref{l:exponent_match} and the inequalities \eqref{e:outer_resto_2}--\eqref{e:outer_resto_3}, and \eqref{e:inner_resto_2}--\eqref{e:inner_resto_3}.
Also, \eqref{esto3} and \eqref{esti3} follow as in  \cite[Section 4.3]{DS3}, again building upon Lemma \ref{l:exponent_match},  and \eqref{e:N_sopra} and \eqref{e:errori_massa}, with the definition of ${\rm Err}_4$ in \eqref{e:Err4-5}. Notice that we are not plugging in that $\bI\ge 1$.

With estimates \eqref{esto1}-- \eqref{esto3} and  \eqref{esti1}-- \eqref{esti3}, together with \eqref{err1}, \eqref{e:out} and \eqref{e:in} follow by \eqref{e:out1} and \eqref{e:in1}, exactly as in \cite{DS3}.

\begin{proof}[Proof of Proposition \ref{small}]
Fix $r\in [\sfrac{s}{t},3]$. We use  the same notation of \cite[Section 4]{DS3} that we have recalled above.
We estimate, for a fixed $i$, using \eqref{e:N_sopra} and \eqref{e:errori_massa} as well as \cite[Lemma 1.9]{DSsns} (as in the proof of Proposition \ref{comp1})
\begin{align*}
\int_{\mathcal{U}_i}|N|^2 \le \int_{\mathcal{V}_i}|N|^2+\int_{\mathcal{U}_i\setminus\mathcal{V}_i}|N|^2\le
C_1\bmo^{2+\gamma_2}\ell_i^{m+4+\gamma_2+2\beta_2}+2\int_{\p^{-1}(\mathcal{U}_i)}\dist^2(x,\cM)\dd\|T\|.
\end{align*}
Now we sum the previous inequality, using \eqref{grvfd} and  recalling that   $(\mathcal{U}_i)_i$ is a partition of $\cB_r$, to obtain
$$
\int_{\cB_r}|N|^2\le C_1 \bmo^{1+\gamma_2}\int_{\cB_r}|N|^2+2\int_{\p^{-1}(\cB_r)}\dist^2(x,\cM)\dd\|T\|(x),
$$
so that the, provided $C_1\eps_0^{2+2\gamma_2}\le \sfrac{1}{2}$,
$$
\int_{\cB_r}|N|^2\le 4\int_{\p^{-1}(\cB_r)}\dist^2(x,\cM)\dd\|T\|(x).
$$
It remains to prove that $\supp(T)\cap \p^{-1}(\cB_r)\subseteq \B_{2r}$. Take $q\in\supp(T)\cap \p^{-1}(p)$, for $p\in\cB_r$. Then $p=\Phii(x)$ where $x\in L\cap B_r$. As $\sfrac{s}{t}\le r$, $\ell(L)\le r$, so that the conclusion is by Corollary \ref{c:cover}, provided that $\eps_0$ is small enough.
\end{proof}
\subsection{Proof of the results of Section \ref{freqbound}}
\begin{proof}[Proof of Proposition \ref{compH}]
	We set $r_{j-1}\defeq \sfrac{s_{j-1}}{t_{j-1}}\le \sfrac{2^{-5}}{\sqrt{m}}$ (by item $\rm (i)$ of Proposition \ref{p:flattening}), which will play the role of $s$ in Proposition \ref{comp1}.
	 By the definition of stopping time $s_{j-1}$ and Proposition \ref{comp1}, there exists $L\in\sW^{(j),e}$ with  $r_{j-1}=c_s^{-1}\ell(L)$ and
	\begin{equation}\label{comp1eqz}
		\bmo^{(j-1)}r_{j-1}^{m+4-2\delta_2} \le C_1	\int_{\p^{-1}(\Omega)}\dist^2(x,\cM_{j-1})\dd\|T_{j-1}\|(x),
	\end{equation}
	for every $\Omega\defeq \Phii (B_{c_s{r_{j-1}}/4} (q))$, where $q_j\in \pi_0$ with $\dist(L,q)\le 3 r_{j-1}$. Notice that $q\in\partial B_{\sfrac{5r_{j-1}}{2}}$ is suitable as $r_{j-1}\ge \dist(0,L)$ and that $\Omega\subseteq  \cB_{\sfrac{21r_{j-1}}{8}}\setminus \cB_{\sfrac{19r_{j-1}}{8}}$, so that \eqref{comp1eqz} reads 
	\begin{equation}\label{comp1eqzz}
		\bmo^{(j-1)} r_{j-1}^{m+4-2\delta_2} \le C_1	\int_{\p^{-1}(\cB_{\sfrac{21r_{j-1}}{8}}\setminus \cB_{\sfrac{19r_{j-1}}{8}})}\dist^2(x,\cM_{j-1})\dd\|T_{j-1}\|(x).
	\end{equation}
	Now we rescale \eqref{comp1eqzz} by $r_{j-1}=\sfrac{t_j}{t_{j-1}} $  (notice e.g.\ that $T_j = ((\iota_{0,t_{j}})_\sharp T)\res \bB_{6\sqrt{m}}= ((\iota_{0,r_{j-1}})_\sharp T_{j-1})\res \bB_{6\sqrt{m}}$) and we obtain:
	\begin{equation}\notag
		\bmo^{(j-1)} r_{j-1}^{2-2\delta_2} \le C_1	\int_{\p^{-1}(\cB_{\sfrac{21}{8}}\setminus \cB_{\sfrac{19}{8}})}\dist^2(x,\cM_{j})\dd\|T_{j}\|(x),
	\end{equation}
	so that, by Proposition \ref{comp2},
	\begin{equation}\label{vmfdf1}
		\bmo^{(j-1)} r_{j-1}^{2-2\delta_2}\le C_1\bigg((\bmo^{(j)})^{1+\gamma_2}+\int_{{\cB_{\sfrac{21}{8}}\setminus \cB_{\sfrac{19}{8}}}}|N_j|^2\bigg).
	\end{equation}
	Now we notice that, with the same notation of Assumption \ref{ipotesi}, $$\mathbf{c}(\cM_{j})\le C_0 r_{j-1}\mathbf{c}(\cM_{j-1})\le C_0r_{j-1}(\bmo^{(j-1)})^{\sfrac{1}{2}}.$$ Also, by item $\rm(iv)$ of Proposition \ref{p:flattening}, $$\bE(T_{j},\B_{6\sqrt{m}},\pi_0) =\bE(T_{j-1},\B_{6\sqrt{m}r_{j-1}},\pi_0)\le C_1\bmo^{(j-1)} r_{j-1}^{2-2\delta_2}.$$ Hence,
	\begin{equation}\label{vmfdf}
		\bmo^{(j)}=\max\{\mathbf{c}(\cM_{j})^2,\bE(T_{j},\B_{6\sqrt{m}},\pi_0) \} \le \max\{ C_0r_{j-1}^2 \bmo^{(j-1)},C_1 \bmo^{(j-1)} r_{j-1}^{2-2\delta_2}\}\le C_1  \bmo^{(j-1)}  r_{j-1}^{2-2\delta_2}.
	\end{equation}
	Therefore, if $\eps_0$ is smaller than a constant (that depends on $C_1$), by \eqref{vmfdf1} and \eqref{vmfdf}, 
	$$
	\bmo^{(j-1)} r_{j-1}^{2-2\delta_2}\le C_1 \int_{{\cB_{3}\setminus \cB_{\sfrac{3}{2}}}}|N_j|^2,
	$$
	which implies the claim by the definition of $\bH_{j}$, recalling \eqref{vmfdf} once again.
\end{proof}
\begin{proof}[Proof of Theorem \ref{t:boundedness}]
	Consider case $(i)$ first, i.e.\ when the intervals of flattening are finitely many, say $j_0$. Let us drop the index $j_0$ for ease of notation. We claim that for every $\rho_0>0$  there exists $\rho\in]0,\rho_0[$ with $\bH(\rho)>0$. Otherwise, $N|_{\cB_{\rho_0}}=0$, so that no cube in $\sW^{(j)}$ intersects $B_{\sfrac{\rho_0}{4}}$, by Proposition \ref{p:separ} and Proposition \ref{p:splitting}. By item $\rm(iii)$ of Corollary \ref{c:cover}, $\supp(T)\cap \p^{-1}(\cB_{\sfrac{\rho_0}{4}})\subseteq\cM$, contradicting the assumption of the Theorem.

	By the paragraph above, we have $\rho_1\in ]0,1[$ with $\bH(\rho_1)>0$. 
	We can prove the first part of \eqref{e:finita1} by contradiction: let $\rho_0\defeq\sup\{r\in]0,\rho[:\bH(r)=0\}$, notice that $\bH(\rho_0)=0$ by continuity of $\bH$.
	By Theorem \ref{t:frequency}, 	$\bI (r) \leq C_1 (1+ \bI (\rho))$ for every $r \in ]\rho_0,\rho[$.
	By letting $r\searrow \rho_0$, we then conclude
\begin{equation}\label{ccsda}
		\rho_0 \bD (\rho_0) \le C_1(1+ \bI (\rho)) \bH (\rho_0) = 0,
\end{equation}
	that is, $N_j\vert_{\cB_{\rho_0}}=0$ which we have  excluded in the previous paragraph.
	Therefore, since $\bH >0$ on $]0, \rho_1[$, we can now apply
	Theorem~\ref{t:frequency} to conclude \eqref{e:finita1}.  
	
	We assume therefore that the intervals of flattening are infinitely many, i.e.\ that we are in case $(ii)$. We partition the extrema
	$t_j$ of the intervals of  flattening into two different classes: the class $(A)$ when $t_j = s_{j-1}$ and the class $(B)$ when $t_j<s_{j-1}$. We are going to prove separately \eqref{e:finita2} for the intervals $[s_j,t_j]$ according to $t_j\in (A)$ or $t_j\in (B)$.
	\medskip\\\textbf{Class $(A)$}. Take $j$ with $t_j\in (A)$. Hence $t_j=s_{j-1}$. Therefore we have, by \eqref{e:rough} and Proposition \ref{compH}, 
	\begin{equation}\label{at3}
		\bI_j(3)=\frac{3 \bD_j(3)}{\bH_j(3)}\le  C_1,
	\end{equation}
	(in particular $\bH_j(3)>0$). We argue now as above to prove that $\bH_j>0$ on $]\sfrac{s_j}{t_j},3[$. If this were not the case, take $\rho_0\defeq\sup\{r\in]0,3[:\bH_j(r)=0\}$, $\bH_j(\rho_0)=0$. Now we can argue exactly as in the first part of the proof, see \eqref{ccsda} to derive a contradiction. This proves the first part of \eqref{e:finita2} (in the case $(A$)) and hence, we can use Theorem \ref{t:frequency} to prove also the second part of \eqref{e:finita2} (in the case $(A$)), by recalling \eqref{at3}.
	\medskip\\\textbf{Class $(B)$} 
	We assume that there are infinitely many $(t_j)_j$ belonging to the class $(B)$, otherwise there is nothing to show. Notice that it is enough to show the last assertion of \eqref{e:finita2}. Indeed, if we manage to do this, by \eqref{e:rough}, we obtain \eqref{at3} and then we can conclude as before.

Assume then by contradiction that the last claim of \eqref{e:finita2} does not hold. 	Now, for every $j\in (B)$, there exists $\eta_j\in ]0,1[$ with  $\bE(T,\B_{6\sqrt{m}(1+\eta_j)t_j},\pi_0)\ge\eps_0^2$, by the definition of the intervals of flattening. Hence, we can take an infinite subset of $(B)$, say $(B)'$, for which (recall also the contradiction assumption)
	\begin{equation}\label{vgraf}
	(\iota_{0,t_j})_\sharp T\rightarrow S,\quad\eps_0^2\le \bE(S,\B_1,\pi_0)\quad\text{and}\quad \lim_{j\in (B)'}  \int_{\cB_3\setminus\cB_{\sfrac{3}{2}}}|N_j|^2=0,
\end{equation}
where $S$ is an area minimizing cone 
	(we refer to the final part of the proof of \cite[Thorem 5.1]{DS3} for details).
	
	Now we want to  show that \eqref{vgraf}  can not be true, hence concluding the proof.  Observe first that $\cM_j=\iota_{0,t_j}(\cM)\rightarrow\pi_0$. 	
	Notice also that by items $\rm(ii)$ and $\rm(iii)$ of Corollary \ref{c:cover} and by \eqref{vfcdsa}, provided that $\eps_0$ is small enough (depending on $C_1$, but not on $j$), if $x\in \supp(T_j)\cap \B_{\sfrac{21}{8}}\setminus \B_{\sfrac{19}{8}}$, then $x\in \supp(T_j)\cap \p^{-1}(\cB_{\sfrac{23}{8}}\setminus \cB_{\sfrac{17}{8}})$ (where, again, we are not making explicit the dependence of $\cB$ on $\cM_j$). Hence, using Proposition \ref{comp2},
\begin{equation}\label{vgraf1}
		\int_{  \B_{\sfrac{21}{8}}\setminus \B_{\sfrac{19}{8}}}\dist^2(x,\cM_j)\dd\|T_j\|(x)\le C_1\bigg( (\bmo^{(j)})^{1+\gamma_2}\sup_{L\in\sW^{(j)}:L\cap (B_3\setminus B_{\sfrac{3}{2}})\ne\emptyset}\ell(L)+\int_{{\cB_{3}\setminus \cB_{\sfrac{3}{2}}}}|N_j|^2\bigg).
\end{equation}

Now  set $A\defeq B_{\sfrac{11}{4}}\setminus B_{\sfrac{9}{4}}$ and $A'\defeq B_{\sfrac{23}{8}}\setminus B_{\sfrac{17}{8}}$. Fix $j\in(B)'$ and take $L\in\sW^{(j)}$ with $L\cap A\ne\emptyset$ such that 
$$
\ell(L)=\max_{J\in\sW^{(j)}: J\cap A\ne \emptyset}\ell(J)\eqdef\bar\ell^{(j)}.
$$  If $L\in\sW^{(j),h}$, take $B\defeq B_{\sfrac{\ell(L)}4}(q)\subseteq B_{(\sqrt{m}+1)\ell(L)}(x_L)\cap A'$. If $L\in\sW^{(j),e}$, take $B\defeq B_{\sfrac{\ell(L)}4}(q)\subseteq A'$ with $\dist(L,q)\le 4c_s^{-1}\ell(L)$. If $L\in \sW^{(j),n}$, by Proposition \ref{p:separ}, (see \cite[Corollary 3.2]{DS2}), take $H\in \sW^{(j),e}$  with $L\subseteq B_{3\sqrt{m}\ell(H)}(x_H)$. Then take $B\defeq  B_{\sfrac{\ell(H)}4}(q)\subseteq A'$  with $\dist(H,q)\le 4c_s^{-1}\ell(H)$. Notice that $\Omega\defeq\Phii(B)\subseteq\cB_{{3}}\setminus\cB_{\sfrac{3}{2}}$. By Proposition \ref{p:separ}, if $L\in\sW^{(j),h}$, $$C_h^2\bmo\ell(L)^{m+2+2\beta_2}\le C_0\int_\Omega|N_j|\le C_0 \int_{\cB_{{3}}\setminus\cB_{\sfrac{3}{2}}}|N_j|^2,$$
whereas, Proposition \ref{p:separ}, if $L\in\sW^{(j),e}\cup\in\sW^{(j),n} $ (and $\ell(H)\ge \ell(L)$ for the case $L\in\sW^{(j),n}$),
$$
C_e\bmo\ell(L)^{m+4-2\delta_2}\le C_1^2\int_\Omega|N_j|\le C_1^2 \int_{\cB_{{3}}\setminus\cB_{\sfrac{3}{2}}}|N_j|^2.
$$
In any case, by the last conclusion of \eqref{vgraf}, we have that 
\begin{equation}\label{vfdsaa}
	\lim_{(B)'\ni j\rightarrow\infty}\bar\ell^{(j)}=0.
\end{equation}

	Now, exploiting the convergence of $T_j\to S$ with the convergence of $\cM_j\to\pi_0$,
	$$
		\int_{  \B_{\sfrac{21}{8}}\setminus \B_{\sfrac{19}{8}}}\dist^2(x,\pi_0)\dd\|S\|(x)=	\lim_{(B)'\ni j\rightarrow\infty}\int_{  \B_{\sfrac{21}{8}}\setminus \B_{\sfrac{19}{8}}}\dist^2(x,\cM_j)\dd\|T_j\|(x),
	$$
	whereas, by \eqref{vfdsaa} and \eqref{vgraf} again,
	$$
	\lim_{(B)'\ni j\rightarrow\infty}\bigg((\bmo^{(j)})^{1+\gamma_2}\sup_{L\in\sW^{(j)}:L\cap (B_3\setminus B_{\sfrac{3}{2}})\ne\emptyset}\ell(L)+\int_{{\cB_{3}\setminus \cB_{\sfrac{3}{2}}}}|N_j|^2\bigg)=0.
	$$
	This, together with \eqref{vgraf1} again, contradicts the second conclusion of \eqref{vgraf}.
\end{proof}

\phantomsection\appendix
\section{Height bound}
\begin{lem}\label{heighapp}
	Let $m,n,Q$ be positive integers. Then there exist $\bar\eps=\bar\eps(m,n,Q)>0$ and $\bar C_0=\bar C_0(m,n,Q)$ with the following property. Let $r>0$ and let  $T$ be an $m$-dimensional (locally) area minimizing current in $\RR^{m+n}$ with 
\begin{gather*}
	\partial T \res \bC_{8r} = 0,\\
	  (\p_{\pi_0})_\sharp( T\mres \bC_{8r} )= Q \a{B_{8r} },
\end{gather*}
and with 	  $\bE\defeq\bE(T,\bC_{8r})<\bar\eps$.
Then there are $\{y_1,\dots,y_Q\}\subseteq\RR^n$ such that  $$\supp(T)\cap \bC_r\subseteq\bigcup_{i=1}^Q\big(\RR^m\times B_{\bar C_0r\bE^{\sfrac{1}{2}}}(y_i)\big).$$
\end{lem}
\begin{proof}
 We can clearly assume with no loss of generality that $r=1$. In what follows, $C_0$ denotes a constant that depends only on $m,n,Q$ and may vary during the proof.
	Assuming $\bar\eps$ sufficiently small, we can apply \cite[Theorem 2.4]{DS1} and \cite[Theorem 2.6]{DS1} (with $\bar\eta=1$) to obtain maps $f,w:B_{2}\rightarrow\Iq(\RR^n)$ and a set $K\subseteq B_{2}(x)$ (notice that, in our situation, the maps $u$ and $w$ of  \cite[Theorem 2.6]{DS1} coincide). 
	Let $\{y_1,\dots,y_Q\}\subseteq\RR^n$ so that $w(0)=\sum_{i=1}^Q \a{y_i}$ and set ${\boldsymbol{\pi}}\defeq\bigcup_{i=1,\dots,Q}(y_i+\pi_0)$. Take $z\in\supp(T)\cap \bC_{\sfrac{3}{2}}$, by the monotonicity formula, $\|T\|(\B_{\sfrac{1}{4}}(z))\ge \omega_m (\sfrac{1}{4})^m$. By \cite[(2.5)]{DS1} $\mathcal{H}^m(B_{2}\setminus K)\le C_0 \bE^{1+\gamma_1}$ and by \cite[(2.8)]{DS1},
	\begin{equation}\label{fvedcdf}
		\int_{B_{2}}\cG(f,w)^2\le  C_0 \bE.
	\end{equation}
	Now set $L\defeq \p_{\pi_0}\big(\supp(T)\cap \B_{\sfrac{1}{4}}(z)\cap (K\times\pi_0^\perp)\big)$, notice that 
	\begin{align*}
\|T\|(\B_{\sfrac{1}{4}}(z))&\le	\|T\|(\B_{\sfrac{1}{4}}(z)\cap (K\times\pi_0^\perp))+\|T\|((B_{2}\setminus K)\times\pi_0^\perp) \\&\le\|T\|(L\times \pi_0^\perp)+C_0\bE\le Q\mathcal{H}^m(L)+C_0\bE
	\end{align*}
so that, provided $\bar\eps$ is small enough (depending on $m,n,Q$), using also a Chebyshev argument with \eqref{fvedcdf}, we deduce that there exists $z'\in \B_{\sfrac{1}{2}}(z)\cap\bC_{\sfrac{7}{4}}\cap\supp(T)$ with $z'\in \gr(w)$. Hence, $\supp(T)\cap  \bC_{\sfrac{3}{2}}\subseteq \B_{1}(\gr(w)\cap \bC_{\sfrac{7}{4}})$.

Now, by \cite[(2.8)]{DS1} and \cite[Theorem 3.9]{DSq}, for some $\alpha=\alpha(m,Q)\in (0,1)$,
\begin{equation}\label{cedscas}
	\cG(w(x),w(y))\le C_0\bE^{\sfrac{1}{2}}|x-y|^\alpha\quad\text{for every }x,y\in B_{\sfrac{7}{4}},
\end{equation}
in particular, if $\bar\eps$ is small enough,
\begin{equation}\label{vfefdcacs}
	\supp(T)\cap \bC_{\sfrac{3}{2}}\subseteq \B_2(\boldsymbol{\pi}).
\end{equation}

Now we follow the proof of \cite[Lemma 6.2]{DMSk3}. We compute, recalling \cite[Lemma 1.9]{DSsns},
\begin{align*}
	\int_{\bC_{\sfrac{3}{2}}}\dist^2(x,\boldsymbol{\pi})\dd\|T\|(x)&=	\int_{\bC_{\sfrac{3}{2}}\cap (K\times\pi_0^\perp)}\dist^2(x,\boldsymbol{\pi})\dd\|T\|(x)+\int_{\bC_{\sfrac{3}{2}}\setminus (K\times\pi_0^\perp)}\dist^2(x,\boldsymbol{\pi})\dd\|T\|(x)\\&\le C_0 \int_{B_{\sfrac{3}{2}}\cap(K\times\pi_0^\perp)}\cG(f,w(0))^2+C_0\|T\|({\bC_{\sfrac{3}{2}}\setminus (K\times\pi_0^\perp)})
\\&\le C_0\int_{B_{\sfrac{3}{2}}}\cG(f,w)^2+C_0\int_{B_{\sfrac{3}{2}}}\cG(w,w(0))^2+C_0\bE
\\&\le C_0\bE +C_0\bE \int_{B_{\sfrac{3}{2}}}|x|^{2\alpha}+C_0\bE\le C_0\bE,
\end{align*}
where in the first inequality we used \eqref{vfefdcacs} and in the last inequality we used \cite[(2.8)]{DS1} and \eqref{cedscas}.
Hence, the conclusion follows from (a scaled version of) \cite[Theorem 3.2]{DMSk3}.	
\end{proof}

	 \bibliographystyle{abbrv}
	\bibliography{Biblio.bib}
\end{document}